\documentclass[reqno]{amsart}

\usepackage{tikz}
\usetikzlibrary{cd}
\usepackage{graphicx}
\graphicspath{{images/}}
\usepackage{amsmath}
\usepackage{amsfonts}
\usepackage{amssymb}
\usepackage{amsthm}
\usepackage{hyperref} 
\hypersetup{pdftoolbar=False}
\usepackage{mathrsfs}
\usepackage[american]{babel}

\numberwithin{equation}{section}
\numberwithin{figure}{section}
\numberwithin{table}{section}

\newcommand{\tprod}[1][]{\textstyle\prod#1\displaystyle}

		\newcounter{romi}
\newenvironment{romanize}{\begin{list}{\addtocounter{romi}{1}\hspace{-3em}\textbf{(\roman{romi})}}{}}{\setcounter{romi}{0}\end{list}}
		\newcounter{alphi}

\newcommand{\mtrx}[1]{\left[ \begin{array}#1 \end{array}\right]}

\newcommand{\qand}{\quad\text{and}\quad}
\newcommand{\qfor}{\quad\text{for}\quad}

\theoremstyle{definition}
\newtheorem{thm}{Theorem}[section]
\newtheorem{prop}[thm]{Proposition}
\newtheorem{lem}[thm]{Lemma}
\newtheorem{cor}[thm]{Corollary}
\newtheorem{df}[thm]{Definition}
\newtheorem{rem}[thm]{Remark}
\newtheorem{exa}[thm]{Example}
\newenvironment{ex}{\begin{exa}}{\hfill$\lozenge$\end{exa}}

\newcommand{\Z}{\mathbb{Z}}

\newcommand{\C}{\mathbb{C}}

\newcommand{\lan}{\langle}
\newcommand{\ran}{\rangle}
\newcommand{\llan}{\lan\!\lan}
\newcommand{\rran}{\ran\!\ran}

\newcommand{\p}{\partial}

\newcommand{\op}{\oplus}

\newcommand{\cE}{\mathcal{E}}
\newcommand{\cF}{\mathcal{F}}
\newcommand{\fg}{\mathfrak{g}}

\newcommand{\fgch}{\mathfrak{g}^\vee}

\newcommand{\Gch}{G^\vee}
\newcommand{\cI}{\mathcal{I}}

\newcommand{\OP}{\mathbb{OP}}

\newcommand{\Pch}{P^\vee}

\newcommand{\cR}{\mathcal{R}}
\newcommand{\ft}{\mathfrak{t}}

\newcommand{\fu}{\mathfrak{u}}

\newcommand{\cU}{\mathcal{U}}
\newcommand{\cW}{\mathcal{W}}

\newcommand{\cX}{\mathcal{X}}
\newcommand{\cZ}{\mathcal{Z}}

\newcommand{\CP}{\mathbb{CP}}								
\newcommand{\Gr}{\mathrm{Gr}}								
\newcommand{\SC}{\mathrm{sc}}								
\newcommand{\indx}{\mathrm{index}}
\newcommand{\GL}{\mathrm{GL}}								
\newcommand{\SL}{\mathrm{SL}}								
\newcommand{\OG}{\mathrm{OG}}							
\newcommand{\LG}{\mathrm{LG}}								
\newcommand{\Mat}{\mathrm{Mat}}						
\newcommand{\Spin}{\mathrm{Spin}}					
\newcommand{\Sp}{\mathrm{Sp}}								

\newcommand{\X}{X}																	
\newcommand{\G}{G}																	
\renewcommand{\P}{P}																
\newcommand{\g}{\fg}																
\newcommand{\up}{\fu_+}													
\newcommand{\um}{\fu_-}													
\newcommand{\cartan}{\ft}													
\newcommand{\parab}{\mathfrak{p}}					
\newcommand{\UEAp}{\cU_+}											
\newcommand{\CUEAp}{\widehat{\cU}_+}			
\newcommand{\UEAm}{\cU_-}											
\newcommand{\CUEAm}{\widehat{\cU}_-}		
\newcommand{\torus}{T}														
\newcommand{\unip}{U_+}													
\newcommand{\unim}{U_-}													
\newcommand{\Che}{e}															
\newcommand{\Chf}{f}																
\newcommand{\Chh}{h}															
\newcommand{\x}{x}																	
\newcommand{\y}{y}																	
\newcommand{\borelp}{B_+}											
\newcommand{\borelm}{B_-}											
\newcommand{\ParabolicIndices}{I_P}					
\newcommand{\IP}{I^P}															

\newcommand{\roots}{\Phi}												
\newcommand{\proots}{\Phi_+}									
\newcommand{\nroots}{\Phi_-}										
\newcommand{\sroots}{\Delta}										
\newcommand{\sr}{\al}																
\newcommand{\chars}{\cX}											

\newcommand{\dG}{\Gch}													
\newcommand{\droots}{\Phi^\vee}							
\newcommand{\dchars}{\cX^\vee}							
\newcommand{\fwt}[1][i]{{\om_{#1}}}						
\newcommand{\dfwt}[1][i]{{\om^\vee_{#1}}}	
\newcommand{\pdroots}{\Phi^\vee_+}				
\newcommand{\ndroots}{\Phi^\vee_-}					
\newcommand{\sdroots}{\Delta\!^\vee}					
\newcommand{\sdr}{\al^\vee}										
\newcommand{\drt}{\al^\vee}											
\newcommand{\dP}{\Pch}														
\newcommand{\dtorus}{T^\vee}									
\newcommand{\dunip}{U^\vee_+}								
\newcommand{\dunim}{U^\vee_-}								
\newcommand{\dborelp}{B^\vee_+}						
\newcommand{\dborelm}{B^\vee_-}						
\newcommand{\dg}{\fgch}													
\newcommand{\dup}{\fu^\vee_+}								
\newcommand{\dum}{\fu^\vee_-}								
\newcommand{\dcartan}{\ft^\vee}							
\newcommand{\dparab}{\mathfrak{p}^\vee}
\newcommand{\dChe}{e^\vee}										
\newcommand{\dChf}{f^\vee}											
\newcommand{\dChh}{h^\vee}										
\newcommand{\dx}{x^\vee}												
\newcommand{\dy}{y^\vee}												

\newcommand{\udG}{\widetilde{G}^\vee}						
\newcommand{\udP}{\widetilde{P}^\vee}							
\newcommand{\udtorus}{\widetilde{T}^\vee}				
\newcommand{\udborelp}{\widetilde{B}^\vee_+}	
\newcommand{\udborelm}{\widetilde{B}^\vee_-}	

\newcommand{\weyl}{W}														
\newcommand{\wo}{w_0}														
\newcommand{\weylp}{W_P}												
\newcommand{\wop}{w_P}													
\newcommand{\invdtorus}{T^P}									
\newcommand{\cosets}{W^P}											
\newcommand{\wP}{w^P}														
\newcommand{\wPinv}{(\wP)^{-1}}								
\newcommand{\wPrime}{w'}												
\newcommand{\wPPrime}{w''}											
\newcommand{\ellwP}{\ell}													
\newcommand{\ellwPrime}{{\ell'}}								
\newcommand{\ellwPPrime}{{\ell''}}							
\newcommand{\ds}{\dot{s}}												
\newcommand{\bs}{\bar{s}}												
\newcommand{\dw}{\dot{w}}											
\newcommand{\bw}{\bar{w}}											
\newcommand{\bwo}{\bw_0}											
\newcommand{\dwop}{\dw_P}										
\newcommand{\bwop}{\bw_P}										
\newcommand{\dwP}{\dw^P}											
\newcommand{\bwP}{\bw^P}											
\newcommand{\bwPinv}{(\bwP)^{-1}}						
\newcommand{\dwPPrime}{\dw''}								
\newcommand{\wrep}{w_c}													
\newcommand{\bwrep}{\bar{\wrep}}						
\newcommand{\dwrep}{\dot{\wrep}}						
\newcommand{\Normalizer}{N}									

\newcommand{\mX}{X^\vee}											
\newcommand{\dRichard}{\cR^\vee}						
\newcommand{\pot}{\cW}													
\newcommand{\loc}{\mathrm{loc}}							
\newcommand{\Jac}{\lan\p\pot\ran}						
\newcommand{\decomps}{\cZ^\vee_P}				
\newcommand{\RisoZ}{\Psi}												
\newcommand{\dunimP}{U_-^P}									

\newcommand{\dgdual}{(\dg)^*}								
\newcommand{\dChedual}[1][_i]{(\dChe#1)^*}	
\newcommand{\dChfdual}[1][_i]{(\dChf#1)^*}		
\newcommand{\dCUEAp}{\widehat{\cU}^\vee_+}	
\newcommand{\dUEAm}{\cU^\vee_-}									
\newcommand{\dCUEAm}{\widehat{\cU}^\vee_-}	

\newcommand{\SumOfEs}{\cE^*}
\newcommand{\SumOfFs}{\cF^*}
\newcommand{\potZ}{\pot_{\decomps}}						
\newcommand{\opendecomps}{\cZ^\circ_P}			
\newcommand{\opendunim}{U^\circ_-}							
\newcommand{\potZo}{\pot_{\opendecomps}}		
\newcommand{\wPrimeSubExp}{\cI}									

\newcommand{\potEHX}{L}																

\newcommand{\dfwtrep}[1][i]{V_{\dfwt[#1]}}						
\newcommand{\hwt}[1][i]{v_{\dfwt[#1]}^+}						
\newcommand{\lwt}[1][i]{v_{\dfwt[#1]}^-}							

\newcommand{\minwt}{\la^\vee}																
\newcommand{\minrep}{V(\minwt)}													
\newcommand{\minrepwts}{M(\minwt)}													
\newcommand{\wt}{\mu^\vee}																	
\newcommand{\wtv}[1]{v_{#1}}													
\newcommand{\wtvmu}{v_{\mu^\vee}}												

\newcommand{\PerrinQ}{Q}															
\newcommand{\NumOcc}{m}														
\newcommand{\wPindex}{\!\!\mathscr{J}\!}
\newcommand{\rtb}{\beta}
\newcommand{\ihat}{i^*}

\newcommand{\itil}{\tilde{\imath}}
\newcommand{\jtil}{\tilde{\jmath}}
\newcommand{\lex}{\prec}											
\newcommand{\wPrimeSubsets}{\mathcal{S}}

\newcommand{\LGT}[1]{#1}											
\newcommand{\LGA}{\LGT{A}}										
\newcommand{\LGB}{\LGT{B}}										
\newcommand{\LGC}{\LGT{C}}										
\newcommand{\LGD}{\LGT{D}}									
\newcommand{\LGE}{\LGT{E}}										
\newcommand{\LGF}{\LGT{F}}										
\newcommand{\LGG}{\LGT{G}}									

\newcommand{\al}{\alpha}

\newcommand{\de}{\delta}

\newcommand{\la}{\lambda}
\newcommand{\om}{\omega}

\calclayout

\newcommand{\ali}[1]{\begin{align*} #1 \end{align*}}
		\newcounter{marginboxes}

\begin{document}
\title[Laurent polynomial LG-models for cominuscule homogeneous spaces]{Laurent polynomial Landau-Ginzburg models for cominuscule homogeneous spaces}
\author{Peter Spacek}
\address{SMSAS, University of Kent, Canterbury, CT2 7FS, United Kingdom}
\email{ps508@kent.ac.uk}
\date{\today}

\begin{abstract}
In this article we construct Laurent polynomial Landau-Ginzburg models for cominuscule homogeneous spaces. These Laurent polynomial potentials are defined on a particular algebraic torus inside the Lie-theoretic mirror model constructed for arbitrary homogeneous spaces in \cite{Rietsch_Mirror_Construction}. The Laurent polynomial takes a similar shape to the one given in \cite{Givental_Equivariant_Gromov_Witten_invariants} for projective complete intersections, i.e.~it is the sum of the toric coordinates plus a quantum term. We also give a general enumeration method for the summands in the quantum term of the potential in terms of the quiver introduced in \cite{CMP_Quantum_cohomology_of_minuscule_homogeneous_spaces}, associated to the Langlands dual homogeneous space. This enumeration method generalizes the use of Young diagrams for Grassmannians and Lagrangian Grassmannians and can be defined type-independently. The obtained Laurent polynomials coincide with the results obtained so far in \cite{Pech_Rietsch_Williams_Quadrics} and \cite{Pech_Rietsch_Lagrangian_Grassmannians} for quadrics and Lagrangian Grassmannians. We also obtain new Laurent polynomial Landau-Ginzburg models for orthogonal Grassmannians, the Cayley plane and the Freudenthal variety.
\end{abstract}

\maketitle
\setcounter{tocdepth}{1}
\tableofcontents

\section{Introduction}
Consider an arbitrary complete homogeneous space $\X=\G/\P$ for a simple and simply-connected complex algebraic group $\G$. In \cite{Rietsch_Mirror_Construction}, Rietsch gives a general construction of a Landau-Ginzburg model that recovers Peterson's presentation of the small quantum cohomology $qH^*(\X)$ in \cite{Peterson}.  In the subsequent articles \cite{Rietsch_Marsh_Grassmannians, Pech_Rietsch_Odd_Quadrics, Pech_Rietsch_Williams_Quadrics, Pech_Rietsch_Lagrangian_Grassmannians} this general construction is worked out in special cases of cominuscule homogeneous spaces (namely Grassmannians, odd quadrics, all quadrics, and Lagrangian Grassmannians, respectively) to formulate the potentials of these Landau-Ginzburg models in projective coordinates. To obtain these potentials, the articles first formulated Laurent polynomial expressions for these on an algebraic torus.

On comparing the methods used to determine these Laurent polynomial potentials, a general, type-independent method has emerged, which is what we will describe here. This is achieved by modifying the method as used in \cite{Pech_Rietsch_Odd_Quadrics} in order to circumvent considerations that only hold in the case of odd quadrics. In particular, we rely on the general structure of minuscule representations as described by the article \cite{Green_Reps_from_polytopes}. 

We obtain the following expression for the Laurent polynomial potential:
\begin{equation}
\sum_{i=1}^\ellwP a_i + q\frac{\sum_{(i_j)\in\wPrimeSubExp} a_{i_1}\cdots a_{i_{\ellwPrime}}}{\prod_{i=1}^\ellwP a_i},
\label{eq:Intro_Explicit_LP_LG_model}
\end{equation}
where $\ell=\dim(\X)$, the $a_i$ are the toric coordinates for $i\in\{1,\ldots,\ell\}$, and $q$ is the quantum parameter. The set $\wPrimeSubExp$ is the set of subexpressions of a certain Weyl group element $\wPrime$ in a fixed reduced expression of the minimal coset representative $\wP$ of the longest Weyl group element, see equation \eqref{eq:Fixed_reduced_expression_for_wP_and_wop} for $\wP$, equation \eqref{eq:Def_of_wPrime} for $\wPrime$ and Definition \ref{df:wPrimeSubExp} for $\wPrimeSubExp$. For the full statement of the result, see Theorem \ref{thm:Explicit_LP_LG_model}.

Notice that the expression in \eqref{eq:Intro_Explicit_LP_LG_model} is reminiscent of the Laurent polynomial for projective complete intersections given in \cite{Givental_Equivariant_Gromov_Witten_invariants}. Indeed, it is given as the sum of the toric coordinates $a_i$ plus a quantum term consisting of a homogeneous polynomial divided by the product of all the toric coordinates. We give a second type-independent description for this homogeneous polynomial: we replace the summation over $\wPrimeSubExp$ by a summation over the set $\wPrimeSubsets$ of special subsets of the quiver $\PerrinQ_\X$ associated to $\wP$ by \cite{Perrin_Small_resolutions_of_minuscule_Schubert_varieties,CMP_Quantum_cohomology_of_minuscule_homogeneous_spaces}, see Definitions \ref{df:PerrinQuiver} and \ref{df:wPrimeSubsets} as well as Corollary \ref{cor:Explicit_LP_LG_model_with_quiver_subsets}. These subsets of $\PerrinQ_\X$ can be considered as generalizations of Young tableaux used in a similar way.

We use the second type-independent expression to obtain Laurent polynomial potentials for all cominuscule homogeneous spaces: Grassmannians $\Gr(k,n)$, quadrics $Q_n$, Lagrangian Grassmannians $\LG(n,2n)$, orthogonal Grassmannians $\OG(n,2n)$, the Cayley plane $\LGE_6/\P_6$, and the Freudenthal variety $\LGE_7/\P_7$. The obtained expressions for quadrics and Lagrangian Grassmannians coincide with those given earlier in \cite{Pech_Rietsch_Williams_Quadrics} and \cite{Pech_Rietsch_Lagrangian_Grassmannians}. This is to be expected from the fact that the type-independent expression is a generalization of these cases. However, to the best of our knowledge, the expressions for orthogonal Grassmannians (for general $n$), the Cayley plane and the Freudenthal variety are new.

\begin{ex}
To illustrate expression \eqref{eq:Intro_Explicit_LP_LG_model}, consider a quadric $Q_d$ of dimension $d$, then the Laurent polynomial becomes:
\[
a_1+a_2+\ldots +a_d + q\frac{a_1+a_d}{a_1a_2\cdots a_d}.
\]
We find that the numerator equals $a_1+a_d$ as it turns out that $\wPrime=s_1$ and $s_1$ only appears as the first and last simple reflection in the reduced expression of $\wP$.

Secondly, consider the orthogonal Grassmannian $\OG(5,10)$: we find
\[
a_1+a_2+\ldots+a_{10} + q\frac{a_1a_2a_3 + a_1a_2a_{10} + a_1a_5a_{10} + a_1a_9a_{10}+a_6a_9a_{10}}{a_1a_2a_3a_4a_5a_6a_7a_8a_9a_{10}}.
\]
The homogeneous polynomial in the numerator corresponds nicely to five subsets of the quiver associated to $\wP$, see Example \ref{ex:OG}.
\end{ex}

The Laurent polynomial potentials will facilitate finding expressions using projective coordinates for the Landau-Ginzburg models constructed in \cite{Rietsch_Mirror_Construction} analogously to \cite{Rietsch_Marsh_Grassmannians,Pech_Rietsch_Odd_Quadrics,Pech_Rietsch_Williams_Quadrics,Pech_Rietsch_Lagrangian_Grassmannians}. This might even be done type-independently. Furthermore, in \cite{Rietsch_Mirror_Construction}, Rietsch also conjectures that Landau-Ginzburg models she constructed give rise to oscillatory integrals that are solutions to the quantum differential equations of $\X$, see Conjecture 8.1 there. This conjecture is verified in \cite{Pech_Rietsch_Williams_Quadrics} using their Laurent polynomial expression to describe a flat section of the Dubrovin connection. Thus, we expect that progress can be made in resolving this conjecture for other cominuscule examples using the results obtained here.

The outline of this article is as follows. We start in section \ref{sec:notation} with recalling some of the fundamentals required and fixing notation. This is followed by a short presentation of the results of \cite{Rietsch_Mirror_Construction} in section \ref{sec:Lie_theoretic_Mirror}. In section \ref{sec:structure_of_minuscule_representations} we restrict to cominuscule homogeneous spaces and consider the general structure of minuscule representations. Next, in section \ref{sec:LP_for_LG} we state our Laurent polynomial expression (Theorem \ref{thm:Explicit_LP_LG_model}) for the potential restricted to an open dense subset. We prove this expression in section \ref{sec:Proof_of_LP_for_LG}, postponing the proof of a number of intermediate results to section \ref{sec:Proof_of_Lemmas}. We deduce an alternative description of the quantum term (Corollary \ref{cor:Explicit_LP_LG_model_with_quiver_subsets}) using subsets of a specific quiver in section \ref{sec:quiver_enumeration}, which simplifies the calculation of the Laurent polynomials. Finally, we apply the expression in Corollary \ref{cor:Explicit_LP_LG_model_with_quiver_subsets} to all the cominuscule homogeneous spaces in section~\ref{sec:Application}, verifying that the expression coincide with \cite{Pech_Rietsch_Williams_Quadrics} and \cite{Pech_Rietsch_Lagrangian_Grassmannians} for quadrics and Lagrangian Grassmannians, and obtaining new Laurent polynomial potentials for orthogonal Grassmannians (subsection \ref{subsec:OG}), the Cayley plane (subsection \ref{subsec:Cayley}) and the Freudenthal variety (subsection \ref{subsec:Freudenthal}).

\subsection*{Acknowledgments.} We would like to thank Dr C.M.A. Pech for the many helpful discussions and her improvements to the readability of this text.

\section{Conventions and notation}\label{sec:notation}
Let $\X$ be a complete homogeneous space (also known as a generalized flag variety) for a simple and simply-connected complex algebraic group $\G$ of rank $n$. In this section and in section \ref{sec:Lie_theoretic_Mirror} we do not make any further assumptions on $\X$, but in the remaining sections will specialize to the case in which $\X$ is cominuscule, see section \ref{sec:structure_of_minuscule_representations}.

Write $\g$ for the Lie algebra of $\G$ and fix a set of Chevalley generators $(\Che_1,\Chf_1,\Chh_1,\ldots,\Che_n,\Chf_n,\Chh_n)$, where $h_i=[e_i,f_i]$ for $i\in\{1,\ldots,n\}$. This gives the decomposition $\g = \up\op\cartan\op\um$, where $\up$ is generated by $\{\Che_i~|~1\le i\le n\}$, $\um$~is generated by $\{\Chf_i\}$ and the Cartan subalgebra $\cartan$ is spanned by $\{\Chh_i\}$. We denote by $\UEAp$ and $\UEAm$ the universal enveloping algebras of $\up$ and $\um$ respectively, and we write their completions as $\CUEAp$ and $\CUEAm$.

Let $\torus$ be the maximal torus and $\unip$ and $\unim$ be the nilpotent subgroups of $\G$ that have $\cartan$, $\up$ and $\um$ as Lie algebras, respectively. Note that we can consider $\unip$ and $\unim$ as lying inside $\CUEAp$ and $\CUEAm$ respectively and that they are generated by the one-parameter subgroups
\[
\x_i(a) = \exp(a\,\Che_i) \qand \y_i(a) = \exp(a\,\Chf_i),
\]
for $i\in\{1,\ldots,n\}$ and $a\in\C$. Here $\exp(a\,\Che_i) = 1 + a\,\Che_i + \tfrac12a^2\,\Che_i^2 + \ldots \in \CUEAp$ and $\exp(a\,\Chf_i)\in\CUEAm$ is given analogously. The subgroup $\borelp=\torus\unip$ defines a Borel subgroup, and its opposite is given by $\borelm=\torus\unim$. They have $\cartan\op\up$ and $\cartan\op\um$ as respective associated Lie algebras. There is now a unique parabolic subgroup $P$ containing $\borelp$ such that $\X=\G/\P$. The Lie algebra $\parab$ of $\P$ satisfies $\up\op\cartan \subset\parab\subset\up\op\cartan\op\um=\g$, i.e.~it is generated by $\Che_i$ and $\Chh_i$ for all $i\in\{1,\ldots,n\}$ and by \emph{certain} $\Chf_i$, but not necessarily all. We will denote by $\ParabolicIndices\subset\{1,\ldots,n\}$ the set of indices such that $\parab$ is generated as a Lie algebra as
\begin{equation}
\parab = \lan \Che_i,\Chh_i,\Chf_j~|~i\in\{1,\ldots,n\},~j\in\ParabolicIndices\ran
\label{eq:Parabolic_Lie_algebra}
\end{equation}
and its complement is denoted by $\IP=\{1,\ldots,n\}\setminus\ParabolicIndices$.

We write $\chars$ for the lattice of characters $\chi:\torus\to\C^*$ of the maximal torus (written additively). Within $\chars$, we denote the set of roots by $\roots\subset\chars$ and a base of simple roots $\sroots=\{\sr_1,\ldots,\sr_n\}$ is determined by the Chevalley generators. The associated sets of positive and negative roots are denoted by $\proots$ and $\nroots$ respectively. We denote the cocharacter lattice by $\dchars$, the coroots by $\droots$ and the simple coroots by $\sdr_i:\chars\to\C$.

With a given root system $\roots$ and character lattice $\chars$, there exists a unique group $\dG$ determined by having as root system the coroots $\droots$ and as character lattice the cocharacter lattice $\dchars$ of $\G$. The character lattice $\dchars$ of $\dG$ also determines a maximal torus $\dtorus$ in $\dG$. The pair $(\dG,\dtorus)$ is called the \emph{Langlands dual pair} associated to $(\G,\torus)$; we call $\dG$ the \emph{Langlands dual group}.
\begin{rem}\label{rem:dG-is-adjoint}
As $\G$ is assumed to be simply-connected, $\dG$ will be adjoint.
\end{rem}
The Langlands dual group $\dG$ inherits the base of simple roots $\sdroots=\{\sdr_1,\ldots,\sdr_n\}$, which in turn determines the decomposition of the Lie algebra $\dg$ of $\dG$ into $\dg=\dum\op\dcartan\op\dup$. The Langlands dual groups $\dunip$, $\dunim$, $\dborelp$, $\dborelm$ and $\dP$ are now defined analogously to above, and we write $\pdroots$ and $\ndroots$ for the sets positive and negative roots of $\dG$. We also obtain Chevalley generators $(\dChe_1,\dChf_1,\dChh_1,\ldots,\dChe_n,\dChf_n,\dChh_n)$ and define the corresponding one-parameter subgroups of $\dunip$ and $\dunim$
\begin{equation}
\dx_i(a) = \exp(a\, \dChe_i) \qand \dy_i(a) = \exp(a\,\dChf_i)
\label{eq:dual_one_parameter_subgroups}
\end{equation}
for $i\in\{1,\ldots,n\}$ and $a\in\C$. Here $\exp(a\,\dChe_i)=1+a\,\dChe_i+\frac12a^2(\dChe_i)^2+\ldots\in\dCUEAp$ in the completed universal enveloping algebra of $\dup$, and analogously for $\exp(a\,\dChf_i)$. Note that the parabolic subgroup $\dP$ is associated to the same set $\ParabolicIndices$ as $\P$: that is, its Lie algebra is given by
\begin{equation}
\dparab = \lan \dChe_i,\dChh_i,\dChf_j~|~i\in\{1,\ldots,n\},~j\in\ParabolicIndices\ran.
\label{eq:Lie_algebra_of_dual_parabolic_subgroup}
\end{equation}
Thus, the complement of the set of indices is the same, so this is denoted by $\IP=\{1,\ldots,n\}\setminus\ParabolicIndices$ as well.
\begin{rem}\label{rem:LanglandsDuality}
When the Dynkin diagram of $\G$ is simply-laced, $\dG$ has the same Dynkin diagram (with the same numbering of the vertices). When the Dynkin diagram of $\G$ has a double or triple edge, the Dynkin diagram of $\dG$ is obtained by reversing the arrows at these edges. Explicitly, if $\G$ is of type $\LGA_n$, $\LGD_n$ or $\LGE_n$, then $\dG$ is of the same type; if $\G$ is of type $\LGB_n$, then $\dG$ is of type $\LGC_n$ and vice versa; finally, for $\G$ of type $\LGF_4$ and $\LGG_2$, $\dG$ is of the same type but has the reverse numbering of vertices.
\end{rem}
Apart from the Chevalley generators $(\dChe_1,\dChf_1,\dChh_1,\ldots,\dChe_n,\dChf_n,\dChh_n)$ for $\dg$, we will need the corresponding dual maps $\dChedual,\dChfdual\in\dgdual$ as well, satisfying
\begin{equation}
\dChedual(\dChe_j) = \de_{ij} = \dChfdual(\dChf_j) \qand \dChedual(\dChf_j)=0=\dChfdual(\dChe_j)
\label{eq:dualdualChevalleyGenerators}
\end{equation}
and vanishing on $\dcartan$ and the other root spaces of $\dg$.

We extend these maps to be defined on arbitrary products of the Chevalley generators using the inclusions of $\dup$ and $\dum$ into their completed universal algebras $\dCUEAp$ and $\dCUEAm$. This in turn allows us to define $\dChedual$ and $\dChfdual$ on $\dunip$ and $\dunim$ through the identification of the one-parameter subgroups $\dx_i(a)\in\dunip$ and $\dy_i(a)\in\dunim$ with $\exp(a\,\dChe_i)\in\dCUEAp$ and $\exp(a\,\dChf_i)\in\dCUEAm$ respectively. Equivalently, $\dChedual$ and $\dChfdual$ are defined as the unique group homomorphisms $\dunip\to\C$ and $\dunim\to\C$ such that
\begin{equation}
\dChedual(\dx_j(a)) = a\de_{ij} = \dChfdual(\dy_j(a)), 
\label{eq:e*_and_f*}
\end{equation}

As $\dG$ is in general not simply-connected, we will need to consider the universal cover $\udG$ of $\dG$ in section \ref{sec:Proof_of_LP_for_LG}. As before, we define the universal covers $\udP$, $\udtorus$, 
$\udborelp$ and $\udborelm$. 
Note that the cover of $\dunip$ is in fact isomorphic to $\dunip$ and the same holds for $\dunim$, so we simply identify them.
\begin{rem}
Considering Remark \ref{rem:LanglandsDuality}, we note that for a simply-connected group $\G$ with a simply-laced Dynkin diagram we have $\udG\cong\G$ as they are both simply-connected and of the same type.
\end{rem}

We turn to the Weyl groups of $\dG$ and $\dP$. We denote by $\weyl$ the Weyl group\footnote{Note that the Weyl groups of $\dG$ and $\G$ are isomorphic, as they only depend on the underlying Coxeter diagram of the Dynkin diagrams and these are invariant under Langlands duality. We therefore omit the ``$\vee$'' from notation. The same holds for the Weyl group $\weylp$ of $\dP$.} of $\dG$, that is the Weyl group associated to the Dynkin diagram of $\dG$. The Weyl group is generated by the simple reflections denoted by $s_i = s_{\sdr_i}$ for $\sdr_i\in\sdroots$ and any expression for $w\in\weyl$ of the form $w=s_{i_1}\cdots s_{i_j}$ with $j$ minimal is called a reduced expression; in this case the integer $j$ is called the length of $w$ and denoted by $\ell(w)=j$. The longest element of $\weyl$ is denoted by $\wo$. We obtain the Weyl group of $\dP$, denoted by $\weylp$, by removing the simple reflections $\{s_i~|~i\in\IP\}$ from the generators of $\weyl$, compare equation \eqref{eq:Lie_algebra_of_dual_parabolic_subgroup}. Note that $\weylp$ is a Weyl group in its own right, associated to the Dynkin diagram of $\dG$ with the vertices marked by $\IP$ removed. The longest element of $\weylp$ is denoted by $\wop$. To each $s_i\in\weyl$, we associate two elements in $\dG$:
\begin{equation}
\ds_i = \dx_i(1)\dy_i(-1)\dx_i(1) \qand \bs_i = \dx_i(-1)\dy_i(1)\dx_i(-1)=\ds_i^{-1},
\label{eq:def_ds_and_bs}
\end{equation}
and we extend this to an arbitrary $w\in\weyl$ with reduced expression $w=s_{i_1}\cdots s_{i_d}$ by setting $\dot w=\ds_{i_1}\cdots\ds_{i_d}$ and $\bar w =\bs_{i_1}\cdots\bs_{i_d}$. Note that $\bar w$ is not equal to $\dot w^{-1}$ in general: it has the reverse product of simple reflections, i.e.~$\dot w^{-1} = \bs_{i_d}\cdots\bs_{i_1}$. Moreover, note that $\ds_i$ and $\bs_i$ only differ by a torus element and both normalize $T$, so that they indeed map to the same element $s_i$ under the identification of $\weyl$ with $\Normalizer_{\dG}(\dtorus)/\dtorus$, the quotient of the normalizer of the maximal torus $\dtorus$ in $\dG$ by the torus.

We let 
\begin{equation}
\invdtorus=(\dtorus)^{\weylp}\subset\dtorus
\label{eq:df_invdtorus}
\end{equation} 
be the part of $\dtorus$ that is invariant under the action $\weylp\times\dtorus\to\dtorus$ given by $(w,t)\mapsto \dw t\dw^{-1}$. Clearly, $\invdtorus$ has dimension $\#\IP$.

We denote by $\cosets\subset\weyl$ the set of minimal coset representatives of $\weyl/\weylp$ (i.e.~for every coset the representative of minimal length), and we denote the minimal representative of $\wo\weylp$ by $\wP$. Note that $\wo=\wP\wop$, and we fix reduced expressions 
\begin{equation}
\wP=s_{r_1}\cdots s_{r_\ellwP} \qand \wop=s_{q_1}\cdots s_{q_m},
\label{eq:Fixed_reduced_expression_for_wP_and_wop}
\end{equation}
so we obtain a reduced expression $\wo=s_{r_1}\cdots s_{r_\ellwP}s_{q_1}\cdots s_{q_m}$.

\section{Rietsch's Lie-theoretic mirror model}\label{sec:Lie_theoretic_Mirror}
In \cite{Rietsch_Mirror_Construction}, Rietsch constructs a Landau-Ginzburg model for general homogeneous spaces $\X=\G/\P$ for $\P$ an arbitrary parabolic subgroup. The mirror variety there is a subvariety of the open Richardson variety associated to $(\wop,\wo)$. This open Richardson variety is given by:
\begin{equation}
\mX = \dRichard_{\wop,\wo} = \bigl(\dborelp\wop\dborelm\cap\dborelm\wo\dborelm\bigr)/\dborelm~\subset~\dG/\dborelm,
\label{eq:Richardson_var}
\end{equation}
see \cite{Rietsch_Mirror_Construction}, section 2. This variety turns out to be related the following subset of $\dG$:
\begin{equation}
\decomps = \dborelm\bwo^{-1}\cap\dunip\invdtorus\bwop\dunim~\subset~\dG.
\label{eq:decomps}
\end{equation}
Namely, there exists an isomorphism
\begin{equation}
\RisoZ:\mX\times\invdtorus\overset\sim\longrightarrow\decomps,
\label{eq:Decomposition_Isomorphism}
\end{equation}
whose inverse is given by $z\mapsto (z\dborelm,t)$ where $z=u_+t\bwop u_-$. 
\begin{rem} 
These results are analogous to the statements in \cite{Rietsch_Mirror_Construction}, section 4.1, although we have modified the definition of $\decomps$ compared to \cite{Rietsch_Mirror_Construction} to facilitate calculations in section \ref{sec:Proof_of_Lemmas}. Indeed, $\RisoZ$ can be obtained as follows: For a given $t\in\invdtorus$, a class $r\dborelm\in\dRichard_{\wop,\wo}$ allows by definition a representative of the form $r=b_-\bwo^{-1} b_-'=b_+\bwop tu_-$ (where $b_-,b_-'\in\dborelm$, $b_+\in\dborelp$ and $u_-\in\dunim$). Note that this representative is unique up to right-multiplication with an element of $\dunim$. Thus, writing $b_-'=t'u_-'$ for $t'\in\dtorus$ and $u_-'\in\dunim$, the element $r(u_-')^{-1}$ is independent of the choice of representative. Moreover, it is an element of $\dborelm\bwo^{-1}$, as $b_-\bwo^{-1} t'=b_-t''\bwo^{-1}$ for some $t''\in\dtorus$. Writing $b_+=t_+u_+$ for $t_+\in\dtorus$ and $u_+\in\dunip$, we find that $\RisoZ(z) = t_+^{-1}r(u_-')^{-1}$ is well-defined and an element of $\decomps$ (using $\bwop t=t\bwop$).
\end{rem}
Now, in \cite{Peterson}, Peterson gave a presentation of the quantum cohomology of a generalized flag variety $\G/\P$ as the coordinate ring of what is subsequently called the \emph{Peterson variety} (see e.g.~\cite{Rietsch_Mirror_Construction}, paragraph 3.2). The coordinate ring of a well-chosen open stratum of this (non-reduced) variety gives the quantum cohomology localized at the quantum parameters (see e.g.~\cite{Rietsch_Mirror_Construction}, equation (3.2)). In \cite{Rietsch_Mirror_Construction}, this open stratum is then shown to be isomorphic to the critical locus of a certain function on $\decomps$. Thus, using the isomorphism $\RisoZ$ of equation \eqref{eq:Decomposition_Isomorphism}, we obtain a subvariety of $\mX\times\invdtorus$ whose coordinate ring is also isomorphic to the localized quantum cohomology of $\X$.

Instead of presenting these statements in more detail, we will instead follow the reformulation of these results presented in \cite{Rietsch_Marsh_Grassmannians}, Theorem 6.5; see also section 4.2 of \cite{Pech_Rietsch_Odd_Quadrics}. There the critical locus is replaced with all of $\mX\times\invdtorus$, but one needs to take the quotient of the coordinate ring by the derivatives with respect to a potential.
\begin{thm}[\cite{Rietsch_Mirror_Construction}, Theorem 4.1, \emph{Lie-theoretic Landau-Ginzburg model}]\label{thm:Lie-theoretic_LG_model}
Let $\X=\G/\P$ be a complete homogeneous space with $\G$ a simple, simply-connected algebraic group over $\C$ and with $\P$ a (not necessarily maximal) parabolic subgroup. There exists a potential $\pot:\mX\times\invdtorus\to\C$ (given in Definition \ref{df:potential}) such that
\[
qH^*(\X)_\loc \cong \C[\mX\times\invdtorus]/\Jac,
\]
where $qH^*(\X)_\loc$ is the (small) quantum cohomology of $\X$ with all quantum parameters inverted and where $\Jac$ is the ideal generated by the derivatives of $\pot$ along $\mX$.
\end{thm}
The potential $\pot$ is presented in \cite{Pech_Rietsch_Odd_Quadrics} as the pull-back of a potential defined on $\decomps$ along the isomorphism $\RisoZ:\mX\times\invdtorus\overset\sim\longrightarrow\decomps$ from equation \eqref{eq:Decomposition_Isomorphism}. To state this potential, we introduce to the following subset of $\dunim$:
\begin{equation}
\dunimP = \dunim\cap \dborelp\bwop\bwo\dborelp~\subset~\dunim.
\label{eq:unique_decomps_unipotents}
\end{equation}
This set has the following property, which will also be important in section \ref{sec:LP_for_LG}:
\begin{lem}[\cite{Pech_Rietsch_Odd_Quadrics}, Proposition 5.1]\label{lem:z_has_unique_decomposition}
Every $z\in\decomps$ has a \emph{unique} decomposition $z=u_+t\bwop u_-$ with $u_+\in\dunip$, $t\in\invdtorus$ and $u_-\in\dunimP$. In particular, fixing $(u_-,t)$ determines $u_+$.
\end{lem}
\begin{rem}
The proof of this result in \cite{Pech_Rietsch_Odd_Quadrics} (Proposition 5.1) can be carried over to the general case without any modification, so it will be omitted here. Note that our definition of $\dunimP$ coincides with the one used in equation (7) of \cite{Pech_Rietsch_Odd_Quadrics} as $(\wP)^{-1}=\wop\wo$ and the chosen representatives only differ by a torus element, so that $\dborelp\bwop\bwo\dborelp=\dborelp(\dwP)^{-1}\dborelp$.
\end{rem}
The potential on $\decomps$ is now defined as follows:
\begin{df}\label{df:potential}
Define the potential $\potZ:\decomps\to\C$ as the map:
\[
\potZ:\quad z=u_+t\bwop u_- \longmapsto \SumOfEs(u_+^{-1}) + \SumOfFs(u_-),
\]
where $\SumOfEs = \sum_{i=1}^n\dChedual$ and $\SumOfFs=\sum_{i=1}^n\dChfdual$, and where the decomposition of $z=u_+t\bwop u_-$ is the unique decomposition with $u_-\in\dunimP$ as stated in Lemma \ref{lem:z_has_unique_decomposition}. Moreover, the potential $\pot:\mX\times\invdtorus\to\C$ mentioned in Theorem \ref{thm:Lie-theoretic_LG_model} is given by $\pot=\potZ\circ\RisoZ$ with $\RisoZ$ given in equation \eqref{eq:Decomposition_Isomorphism}.
\end{df}

\section{Cominuscule homogeneous spaces and minuscule representations}\label{sec:structure_of_minuscule_representations}
In section \ref{sec:LP_for_LG} we will give a Laurent polynomial expression for the potential $\potZ$ restricted to an open algebraic torus inside $\decomps$, after assuming the homogeneous space is \emph{cominuscule}. In this section we will discuss this property, fix further notation and finally consider \emph{minuscule representations} of a given Lie algebra.

We will maintain all the assumptions and conventions of section \ref{sec:notation}. In particular, we assume that $\X=\G/\P$ is a homogeneous space for a complex algebraic group $\G$ of rank $n$ that is both simple and simply-connected. Now, we assume in addition that $\X$ is \emph{minimal}. This is equivalent to the assumption that $\P$ is a \emph{maximal} parabolic subgroup; that is, $\IP=\{k\}$ for a single index $k\in\{1,\ldots,n\}$, see equation \eqref{eq:Parabolic_Lie_algebra}. Thus, $\P$ is associated to a single vertex $k$ of the Dynkin diagram of $\G$. This fact is denoted by $\P=\P_k$. Note that the Langlands dual group $\dP$ is associated to the $k$-th vertex of the Dynkin diagram of $\dG$ (see equation \eqref{eq:Lie_algebra_of_dual_parabolic_subgroup} and Remark \ref{rem:LanglandsDuality}), so we can also write $\dP=\dP_k$.

Because of the maximality of $\dP=\dP_k$, we know that $\weylp=\lan s_i~|~i\neq k\ran$. In particular, the invariant torus $\invdtorus\subset\dtorus$ is one-dimensional and $\sdr_k:\invdtorus\to\C^*$ gives an isomorphism ($\dG$ being adjoint). The element $\wop s_k\in\weyl$ will turn out to be of particular interest; we will write $\wPPrime\in\cosets$ for the minimal coset representative of $\wop s_k\weylp$ and denote its length by $\ell(\wPPrime)=\ellwPPrime\le\ellwP=\ell(\wP)$. Define $\wPrime\in\weyl$ by 
\begin{equation}
\wPrime = \wP(\wPPrime)^{-1}
\label{eq:Def_of_wPrime}
\end{equation}
and write $\ell(\wPrime)=\ellwPrime$; clearly $\ellwP=\ellwPrime+\ellwPPrime$.

The second assumption we will impose on $\X=\G/\P_k$ is that it is a \emph{cominuscule} homogeneous space. A minimal homogeneous space $\X=\G/\P_k$ is called (co-) minuscule if the fundamental weight $\fwt[k]$ is {(co-)} minuscule. Recall that the fundamental weights $\{\fwt[1],\ldots,\fwt[n]\}$ form a basis of the character lattice $\chars$ dual to the simple coroot basis $\sdroots=\{\sdr_1,\ldots,\sdr_n\}$ of the cocharacter lattice $\dchars$. A fundamental weight $\fwt$ is called \emph{minuscule} if it satisfies one of the following equivalent conditions (see also \cite{Bourbaki}, section VI.1, exercise 24):
\begin{romanize}
\item For every $\drt\in\droots$, $\llan \fwt,\drt\rran\in\{-1,0,+1\}$, where $\llan\cdot,\cdot\rran:\chars\times\dchars\to\C$ denotes the dual pairing.
\item For $\sdr_0$ the longest root of the root system $\droots$, $\llan\fwt,\sdr_0\rran=1$.
\item The coefficient of $\sdr_i$ in $\sdr_0$ is 1.
\end{romanize}
A fundamental weight $\fwt$ is called \emph{cominuscule} if the corresponding coweight $\dfwt$ is minuscule. (Recall that the coweights $\{\dfwt[1],\ldots,\dfwt[n]\}$ form a basis of $\dchars$ dual to the basis of simple root  $\sroots=\{\sr_1,\ldots,\sr_n\}$ of $\chars$.) The list of minuscule and cominuscule fundamental weights is well-known; we have included it in Table \ref{tab:CominusculeSpaces} together with the associated minimal homogeneous spaces $\X=\G/\P_k$.

\begin{table}[tbh]%
\[
\begin{array}{ccc|c|lcc}
\multicolumn{3}{c|}{\text{type and weight}} & \text{(co-) minuscule} & \text{variety} & \dim & \indx \vphantom{\dfrac MM}\\\hline
A_{n-1} & 
\begin{tikzpicture}[baseline = -0.75ex, scale = 0.625]
\coordinate (1) at (0,0);
\coordinate (2) at (0.625,0);
\coordinate (3) at (1.25,0);
\coordinate (3+) at (1.75,0);
\coordinate (n-) at (2.25,0);
\coordinate (n) at (2.75,0);
\draw[thick] (1)--(3+);
\draw[thick, dotted] (3+)--(n-);
\draw[thick] (n-)--(n);
\draw[black,fill=black] (1) circle (3pt); 
\draw[black,fill=black] (2) circle (3pt);
\draw[black,fill=black] (3) circle (3pt);
\draw[black,fill=black] (n) circle (3pt);
\end{tikzpicture}
&\text{any }k &\text{both}& \Gr(k,n) & k(n-k) & n  \vphantom{\dfrac MM}\\
B_n & 
\begin{tikzpicture}[baseline = -0.75ex, scale = 0.625]
\coordinate (1) at (0,0);
\coordinate (2) at (0.625,0);
\coordinate (2+) at (1.125,0);
\coordinate (n-1-) at (1.625,0);
\coordinate (n-1) at (2.125,0);
\coordinate (n) at (3.125,0);
\draw[thick] (1)--(2+);
\draw[thick, dotted] (2+)--(n-1-);
\draw[thick] (n-1-)--(n-);
\draw[thick,double,double distance = 1pt] (n-1)--(n);
\draw[black,fill=black] (1) circle (3pt); 
\draw[black,fill=white] (2) circle (3pt);
\draw[black,fill=white] (n-1) circle (3pt);
\draw[black,fill=white] (n) circle (3pt);
\node at (2.625,0) {$\boldsymbol{>}$};
\end{tikzpicture}
& 1 & \text{cominuscule} & Q_{2n-1} & 2n-1 & 2n-1    \vphantom{\dfrac MM}\\
B_n & 
\begin{tikzpicture}[baseline = -0.75ex, scale = 0.625]
\coordinate (1) at (0,0);
\coordinate (2) at (0.625,0);
\coordinate (2+) at (1.125,0);
\coordinate (n-1-) at (1.625,0);
\coordinate (n-1) at (2.125,0);
\coordinate (n) at (3.125,0);
\draw[thick] (1)--(2+);
\draw[thick, dotted] (2+)--(n-1-);
\draw[thick] (n-1-)--(n-);
\draw[thick,double,double distance = 1pt] (n-1)--(n);
\draw[black,fill=white] (1) circle (3pt); 
\draw[black,fill=white] (2) circle (3pt);
\draw[black,fill=white] (n-1) circle (3pt);
\draw[black,fill=black] (n) circle (3pt);
\node at (2.625,0) {$\boldsymbol{>}$};
\end{tikzpicture}
& n &\text{minuscule} & \OG(n,2n+1) &\frac12n(n+1) & 2n   \vphantom{\dfrac MM}\\
C_n & 
\begin{tikzpicture}[baseline = -0.75ex, scale = 0.625]
\coordinate (1) at (0,0);
\coordinate (2) at (0.625,0);
\coordinate (2+) at (1.125,0);
\coordinate (n-1-) at (1.625,0);
\coordinate (n-1) at (2.125,0);
\coordinate (n) at (3.125,0);
\draw[thick] (1)--(2+);
\draw[thick, dotted] (2+)--(n-1-);
\draw[thick] (n-1-)--(n-);
\draw[thick,double,double distance = 1pt] (n-1)--(n);
\draw[black,fill=black] (1) circle (3pt); 
\draw[black,fill=white] (2) circle (3pt);
\draw[black,fill=white] (n-1) circle (3pt);
\draw[black,fill=white] (n) circle (3pt);
\node at (2.625,0) {$\boldsymbol{<}$};
\end{tikzpicture}
& 1 & \text{minuscule} & \CP^{2n-1} & 2n-1 & 2n    \vphantom{\dfrac MM}\\
C_n & 
\begin{tikzpicture}[baseline = -0.75ex, scale = 0.625]
\coordinate (1) at (0,0);
\coordinate (2) at (0.625,0);
\coordinate (2+) at (1.125,0);
\coordinate (n-1-) at (1.625,0);
\coordinate (n-1) at (2.125,0);
\coordinate (n) at (3.125,0);
\draw[thick] (1)--(2+);
\draw[thick, dotted] (2+)--(n-1-);
\draw[thick] (n-1-)--(n-);
\draw[thick,double,double distance = 1pt] (n-1)--(n);
\draw[black,fill=white] (1) circle (3pt); 
\draw[black,fill=white] (2) circle (3pt);
\draw[black,fill=white] (n-1) circle (3pt);
\draw[black,fill=black] (n) circle (3pt);
\node at (2.625,0) {$\boldsymbol{<}$};
\end{tikzpicture}
& n & \text{cominuscule} & \LG(n,2n) &\frac12n(n+1) & n+1    \vphantom{\dfrac MM}\\
D_n & 
\begin{tikzpicture}[baseline = -0.75ex, scale = 0.625]
\coordinate (1) at (0,0);
\coordinate (2) at (0.625,0);
\coordinate (2+) at (1.125,0);
\coordinate (n-2-) at (1.625,0);
\coordinate (n-2) at (2.125,0);
\coordinate (n-1) at (2.75,0.2);
\coordinate (n) at (2.75,-0.2);
\draw[thick] (1)--(2)--(2+);
\draw[thick, dotted] (2+)--(n-2-);
\draw[thick] (n-2-)--(n-2)--(n-1);
\draw[thick] (n-2)--(n);
\draw[black,fill=black] (1) circle (3pt); 
\draw[black,fill=white] (2) circle (3pt);
\draw[black,fill=white] (n-2) circle (3pt);
\draw[black,fill=white] (n-1) circle (3pt);
\draw[black,fill=white] (n) circle (3pt);
\end{tikzpicture}
& 1 & \text{both} & Q_{2n-2} &2n-2 & 2n-2    \vphantom{\dfrac MM}\\
D_n & 
\begin{tikzpicture}[baseline = -0.75ex, scale = 0.625]
\coordinate (1) at (0,0);
\coordinate (2) at (0.625,0);
\coordinate (2+) at (1.125,0);
\coordinate (n-2-) at (1.625,0);
\coordinate (n-2) at (2.125,0);
\coordinate (n-1) at (2.75,0.2);
\coordinate (n) at (2.75,-0.2);
\draw[thick] (1)--(2)--(2+);
\draw[thick, dotted] (2+)--(n-2-);
\draw[thick] (n-2-)--(n-2)--(n-1);
\draw[thick] (n-2)--(n);
\draw[black,fill=white] (1) circle (3pt); 
\draw[black,fill=white] (2) circle (3pt);
\draw[black,fill=white] (n-2) circle (3pt);
\draw[black,fill=black] (n-1) circle (3pt);
\draw[black,fill=black] (n) circle (3pt);
\end{tikzpicture}
& n-1\text{ or }n & \text{both} & \OG(n,2n) &\frac12n(n-1) & 2n-2     \vphantom{\dfrac MM}\\
E_6 & 
\begin{tikzpicture}[baseline = 0.25ex, scale = 0.625]
\coordinate (1) at (0,0);
\coordinate (3) at (0.625,0);
\coordinate (4) at (1.25,0);
\coordinate (5) at (1.875,0);
\coordinate (6) at (2.5,0);
\coordinate (2) at (1.25,0.625);
\draw[thick] (1)--(3)--(4)--(5)--(6);
\draw[thick] (4)--(2);
\draw[black,fill=black] (1) circle (3pt); 
\draw[black,fill=white] (2) circle (3pt);
\draw[black,fill=white] (3) circle (3pt);
\draw[black,fill=white] (4) circle (3pt);
\draw[black,fill=white] (5) circle (3pt);
\draw[black,fill=black] (6) circle (3pt);
\end{tikzpicture}
& 1\text{ or }6 & \text{both} & \OP^2 = \LGE^\SC_6/P_6 & 16 & 12    \vphantom{\dfrac MM}\\
E_7 & 
\begin{tikzpicture}[baseline = 0.25ex, scale = 0.625]
\coordinate (1) at (0,0);
\coordinate (3) at (0.625,0);
\coordinate (4) at (1.25,0);
\coordinate (5) at (1.875,0);
\coordinate (6) at (2.5,0);
\coordinate (7) at (3.125,0);
\coordinate (2) at (1.25,0.625);
\draw[thick] (1)--(3)--(4)--(5)--(6)--(7);
\draw[thick] (4)--(2);
\draw[black,fill=white] (1) circle (3pt); 
\draw[black,fill=white] (2) circle (3pt);
\draw[black,fill=white] (3) circle (3pt);
\draw[black,fill=white] (4) circle (3pt);
\draw[black,fill=white] (5) circle (3pt);
\draw[black,fill=white] (6) circle (3pt);
\draw[black,fill=black] (7) circle (3pt);
\end{tikzpicture}
& 7 & \text{both} & \LGE^\SC_7/P_7 & 27 & 18   \vphantom{\dfrac M{\frac MM}} \\\hline
\end{array}
\]
\caption{Table listing for each type the fundamental weights that are minuscule, cominuscule or both, the associated homogeneous spaces and their dimensions and indexes. In this table, $Q_n$ denotes a quadric of dimension $n$; $\Gr(k,n)$ denotes the Grassmannian of $k$-dimensional subspaces in $\C^n$; $\LG(n,2n)$ denotes the Lagrangian Grassmannian of maximal isotropic subspaces with respect to the standard symplectic form; $\OG(n,2n)$ and $\OG(n,2n+1)$ denote (one of the two isomorphic connected components of) the orthogonal Grassmannians of maximal isotropic subspaces with respect to the standard quadratic form; $\OP^2=\LGE^\SC_6/P_6$ denotes the Cayley plane which is a homogeneous space for $\LGE^\SC_6$, the simply-connected Lie group of type $\LGE_6$; and finally $\LGE^\SC_7/\P_7$ is called the Freudenthal variety and is homogeneous for $\LGE^\SC_7$, the simply-connected Lie group of type $\LGE_7$.  Note that the two varieties that are only minuscule are redundant: the type-$B_n$ minuscule variety $\OG(n,2n+1)$ is isomorphic to the variety $\OG(n+1,2n+2)$ which is both minuscule and cominuscule as a type-$D_{n+1}$ homogeneous space; similarly, the type-$C_n$ minuscule variety $\CP^{2n-1}$ is of course the same as $\Gr(1,2n)$, which is both minuscule and cominuscule as a type-$A_{2n-1}$ homogeneous space. Adapted from \cite{CMP_Quantum_cohomology_of_minuscule_homogeneous_spaces}.}
\label{tab:CominusculeSpaces}
\end{table}

Thus, assuming that $\X=\G/\P_k$ is cominuscule means that $\dfwt[k]$ is minuscule, which in turn implies that the fundamental weight representation $\dfwtrep[k]$ is minuscule; that is, the Weyl group acts transitively on the weight spaces of $\dfwtrep[k]$. Here the simple reflection $s_i\in\weyl$ acts on a vector $\wtvmu$ of weight $\wt$ by mapping it to the vector $\bs_i\cdot\wtvmu$ of weight $s_i(\wt)$.

Recall that the fundamental weight representation $\dfwtrep[k]$ is the highest weight representation of $\dg$ with $\dfwt[k]$ as highest weight. For any choice of a highest-weight vector $\hwt[k]$, we obtain the representation as $\dfwtrep[k] = \dUEAm\cdot\C\hwt[k]$, where $\dUEAm$ is the universal enveloping algebra of $\dum$ (see for example \cite{Humphreys_Lie_Algebras}, Theorem 20.2). Thus, $\{\dChf_{i_1}\cdots \dChf_{i_j}\cdot\hwt[k]~|~j\ge0\}$ spans $\dfwtrep[k]$. We want to compare the actions of $\dUEAm$ and $\weyl$, so we will need some results on the structure of minuscule representations.
\begin{rem}
To be able to apply the results directly to our case, we will change the notation of the following theorem to conform to our situation; for example, we write $\dg$ for a general Lie algebra, as we want to apply the theorem to the minuscule highest weight representation of the Lie algebra $\dg$ of the (adjoint) Langlands dual group $\dG$ associated to the (simply-connected) group $\G$ such that $\X=\G/\P_k$.
\end{rem}
We denote the Cartan integers of a root system $\droots$ by $a_{ij}\in\Z$, $i,j\in\{1,\ldots,n\}$. They are given by $a_{ij} = 2\frac{(\sdr_j,\sdr_i)}{(\sdr_i,\sdr_i)}$ for any choice of non-degenerate, symmetric, bilinear form $(\cdot,\cdot)$ (e.g.~the Killing form). We also use the notation $c(\wt,\sdr_i)= 2\frac{(\wt,\sdr_i)}{(\sdr_i,\sdr_i)}$ for $\wt$ a general weight. 

\begin{thm}[Green \cite{Green_Reps_from_polytopes}]\label{thm:Green_structure_minuscule_reps}
Let $\dg$ be a simple Lie algebra and fix a set of simple roots $\sdroots=\{\sdr_1,\ldots,\sdr_n\}$ and Chevalley generators $(\dChe_i,\dChf_i,\dChh_i)$. Suppose $\minrep$ is a minuscule representation of $\dg$ with highest weight $\minwt$. Denote by $\minrepwts$ the weights of $\minrep$ and let $\wt\in\minrepwts$ be an arbitrary weight. The following statements hold:
\begin{romanize}
\item $c(\wt,\sdr_i)\in\{-1,0,1\}$, and $\wt-c\sdr_i\in\minrepwts$ if and only if $c=c(\wt,\sdr_i)$.
\item Each of the weight spaces is one-dimensional.
\item Given a highest weight vector $\wtv{\minwt}^+$, we can find a basis $\{\wtvmu~|~\wt\in\minrepwts\}$ with the following properties: $\wtvmu$ has weight $\wt$ and the basis vector of the highest weight $\wtv{\minwt}$ coincides with $\wtv{\minwt}^+$; the Chevalley generators act on $\wtvmu$ as
\ali{
\dChe_i\cdot\wtvmu &= \left\{ \begin{array}{ll} \wtv{\wt+\sdr_i},&\text{if $c(\wt,\sdr_i)=-1$,} \\ 0,&\text{otherwise,} \end{array} \right. \\
\dChf_i\cdot\wtvmu &= \left\{ \begin{array}{ll} \wtv{\wt-\sdr_i},&\text{if $c(\wt,\sdr_i)=+1$,} \\ 0,&\text{otherwise,} \end{array} \right. 
}
and $\dChh_i\cdot\wtvmu = c(\wt,\sr_i)\,\wtvmu$.
\item For any $v\in\minrep$ and any $i\in\{1,\ldots,n\}$, we have $(\dChe_i)^2\cdot v = 0$ and $(\dChf_i)^2\cdot v=0$.
Moreover, if $j\in\{1,\ldots,n\}$ is such that the Cartan integer $a_{ij} =-1$ (i.e.~when in the Dynkin diagram we have 
\begin{tikzpicture}[baseline = -0.5ex, scale = 0.625]
\coordinate (i) at (0,0);
\coordinate (j) at (0.75,0);
\draw[thick] (i)--(j);
\draw[black,fill=black] (i) circle (3pt); 
\draw[black,fill=black] (j) circle (3pt);
\node at (i) [left = 3pt] {$i$};
\node at (j) [right = 3pt] {$j$};
\end{tikzpicture}
or 
\begin{tikzpicture}[baseline = -0.5ex, scale = 0.625]
\coordinate (i) at (0,0);
\coordinate (j) at (0.8,0);
\coordinate (mw) at (0.4,0);
\draw[thick,double,double distance = 1pt] (i)--(j);
\draw[black,fill=black] (i) circle (3pt); 
\draw[black,fill=black] (j) circle (3pt);
\node at (mw) {$\boldsymbol{>}$};
\node at (i) [left = 3pt] {$i$};
\node at (j) [right = 3pt] {$j$};
\end{tikzpicture}
\hspace{-1ex})\footnote{Strictly speaking, we also have $a_{ij} = -1$ when we have 
\begin{tikzpicture}[baseline = -0.75ex, scale = 0.5]
\coordinate (i) at (0,0);
\coordinate (j) at (0.75,0);
\draw[black,fill=black] (i) circle (3pt); 
\draw[black,fill=black] (j) circle (3pt);
\draw (i)--(j) node[midway] {$>$};
\draw ([yshift = -2pt]i)--([yshift = -2pt]j);
\draw ([yshift = +2pt]i)--([yshift = +2pt]j);
\node at (i) [left = 3pt] {$i$};
\node at (j) [right = 3pt] {$j$};
\end{tikzpicture}
\hspace{-1ex}, but these edges only appear in the Dynkin diagram of type $\LGG_2$, and the corresponding Lie algebra does not have any minuscule representations.} we have both $\dChe_i\dChe_j\dChe_i\cdot v=0$ and $\dChf_i\dChf_j\dChf_i\cdot v=0$. Finally, if $a_{ij}<0$ we have $\dChe_i\dChf_j\cdot v = 0 = \dChf_j\dChe_i\cdot v$.
\end{romanize}
\end{thm}
The following corollary is obtained directly by applying Theorem \ref{thm:Green_structure_minuscule_reps} to the definition of $\ds_i$ and $\bs_i$ given in equation \eqref{eq:def_ds_and_bs}.
\begin{cor}\label{cor:Action_of_s_e_f}
With the assumptions of Theorem \ref{thm:Green_structure_minuscule_reps}, write $c=c(\wt,\sdr_i)$, then we have $s_i(\wt) = \wt-c\sdr_i$ and
\ali{
\ds_i\cdot\wtvmu &= \left\{\begin{array}{ll}
\wtv{\wt+\sdr_i} = \dChe_i\cdot\wtvmu, & \text{if $c=-1$,}\\
\wtvmu, & \text{if $c=\hphantom{+}0$},\\
-\wtv{\wt-\sdr_i} = -\dChf_i\cdot\wtvmu, & \text{if $c=+1$,}
\end{array}\right. \\
\bs_i\cdot\wtvmu &= \left\{\begin{array}{ll}
-\wtv{\wt+\sdr_i} = -\dChe_i\cdot\wtvmu, & \text{if $c=-1$,}\\
\wtvmu, & \text{if $c=\hphantom{+}0$,}\\
\wtv{\wt-\sdr_i} = \dChf_i\cdot\wtvmu, & \text{if $c=+1$.}
\end{array}\right.
\intertext{Conversely, we have}
\dChe_i\cdot\wtvmu &= \left\{\begin{array}{ll}
\wtv{\wt+\sdr_i} = \ds_i\cdot\wtvmu = -\bs_i\cdot\wtvmu,  & \text{if $c=-1$,}\\
0, & \text{otherwise,}
\end{array}\right. \\
\dChf_i\cdot\wtvmu &= \left\{\begin{array}{ll}
\wtv{\wt-\sdr_i} = \bs_i\cdot\wtvmu = -\ds_i\cdot\wtvmu,  & \text{if $c=+1$,}\\
0, & \text{otherwise.}
\end{array}\right.
}
\end{cor}

We now return to the case where $\dg$ is the Lie algebra of the (adjoint) Langlands dual group $\dG$ of the (simply-connected) Lie group $\G$ such that $\X=\G/\P_k$ is cominuscule. Thus, the highest weight representation $\dfwtrep[k]$ of $\dg$ is minuscule and we can apply Corollary \ref{cor:Action_of_s_e_f} to $\dfwtrep[k]$ to obtain the following facts regarding the action of the Weyl group $W$ on the highest weight vector $\hwt[k]$:
\begin{lem}\label{lem:action_of_weyl_elts}
Consider the highest weight vector $\hwt[k]$ of the minuscule fundamental weight representation $\dfwtrep[k]$, where $k$ is such that $\X=\G/\P_k$.
\begin{romanize}
\item Given an arbitrary $w\in W$ with minimal coset representative $\wrep\in\cosets$, then we have $\bw\cdot\hwt[k]=\bwrep\cdot\hwt[k]$ and $\dw\cdot\hwt[k]=\dwrep\cdot\hwt[k]$.
\item An element $\wrep\in\cosets$ with reduced expression $\wrep=s_{i_1}\cdots s_{i_c}$ acts on the vector $\hwt[k]$ by $\bwrep\cdot\hwt[k] = \dChf_{i_1}\cdots\dChf_{i_c}\cdot\hwt[k]$ and by $\dwrep\cdot\hwt[k] = (-1)^c\dChf_{i_1}\cdots\dChf_{i_c}\cdot\hwt[k]$. 
\item Conversely, if $\dChf_{i_1}\cdots\dChf_{i_j}\cdot\hwt[k]$ is non-zero of weight $\wt$, then $s_{i_1}\cdots s_{i_j}$ is a reduced expression for the (unique) element $w_c\in\cosets$ such that $w_c\cdot\dfwt[k]=\wt$.
\item In particular, $\bwop,\dwop$ and their inverses act trivially on $\hwt[k]$ and we have that the lowest weight vector defined by $\lwt[k]=\bwo\cdot\hwt[k]$ satisfies $\lwt[k]=\bwP\cdot\hwt[k]=\dChf_{r_1}\cdots\dChf_{r_\ellwP}\cdot\hwt[k]$, where $\wP=s_{r_1}\cdots s_{r_\ellwP}$ is the reduced expression fixed in equation \eqref{eq:Fixed_reduced_expression_for_wP_and_wop}. Moreover, if $\dChf_{i_1}\cdots\dChf_{i_j}\cdot\hwt[k] = \lwt[k]$, then $s_{i_1}\cdots s_{i_j} = \wP$ and this is a reduced expression.
\end{romanize}
\end{lem}
\begin{proof} \textbf{(i)} Recall the identity $s_j(\dfwt[k]) = \dfwt[k]$ for $j\neq k$, i.e.~when $s_j\in\weylp$. This implies that an arbitrary $w\in\weyl$ acts on $\dfwt[k]$ by its minimal coset representative $\wrep$ in $\cosets$. Corollary \ref{cor:Action_of_s_e_f} implies that $\bw\cdot\hwt[k]=\bwrep\cdot\hwt[k]$ and $\dw\cdot\hwt[k]=\dwrep\cdot\hwt[k]$.

\textbf{(ii)} We need to show that each factor of $\bwrep=\bs_{i_1}\cdots\bs_{r_c}$ acts as $\dChf_i$ on $\hwt[k]$. Considering Corollary \ref{cor:Action_of_s_e_f}, each $\bs_i$ acts either by $\dChf_i$, $-\dChe_i$ or as the identity map. Clearly, none of these factors acts as the identity, as we could remove it from the product, which would contradict the minimality of the coset representative $\dwrep$. Moreover, none of the factors acts as $-\dChe_i$ either, because of the following argument:

Let $\bs_i$ be the right-most factor acting as $-\dChe_i$. As $\hwt[k]$ is the highest weight vector, we have $\dChe_i\cdot\hwt[k]=0$ (as $\dChe_i$ raises the height of the weight), so there must be a number of $\bs_j$ in between $\bs_i$ and $\hwt[k]$ acting as $\dChf_j$. Let $\dChf_j$ be the factor next to $\dChe_i$. There are three cases: $j\neq i$ and $a_{ij}<0$; $j\neq i$ and $a_{ij}=0$; and $j=i$. When $j\neq i$ and $a_{ij}<0$, Theorem \ref{thm:Green_structure_minuscule_reps} \textbf{(iv)} tells us that $\dChe_i\dChf_j\cdot v=0$ in the representation, which is impossible. When $j\neq i$ and $a_{ij}=0$, we know that $\dChe_i$ and $\dChf_j$ commute in the Lie algebra, so we can assume without loss of generality that only the case $j=i$ occurs. In the case $j=i$ we obtain $\dChe_i\dChf_j\cdot v=v$ according to Theorem \ref{thm:Green_structure_minuscule_reps} \textbf{(iii)}. However, this is in contradiction with the fact that $\wrep$ is a minimal coset representative. Thus, all of the factors $\bs_i$ of $\bwP$ act as $\dChf_i$.

For the equality $\dwrep=(-1)^c\dChf_{i_1}\cdots\dChf_{i_c}\cdot\hwt[k]$, we use an analogous argument combined with the fact that $\ds_i$ acts as either $\dChe_i$, $-\dChf_i$ or the identity due to Corollary \ref{cor:Action_of_s_e_f}.

\textbf{(iii)} Let $w=s_{i_1}\cdots s_{i_j}$ and let $w_c\in\cosets$ denote the minimal length representative of $w\weylp$. Given a reduced expression $w_c=s_{i'_1}\cdots s_{i'_c}$, parts \textbf{(i)} and \textbf{(ii)} imply that 
\[
\bw\cdot\hwt[k]=\bw_c\cdot\hwt[k]=\dChf_{i'_1}\cdots\dChf_{i'_c}\cdot\hwt[k].
\]
Thus, $\bw\cdot\hwt[k]$ has weight $\dfwt[k] - \sdr_{i'_1}-\cdots-\sdr_{i'_c}$ by Theorem \ref{thm:Green_structure_minuscule_reps} \textbf{(iii)}.


On the other hand, $\dChf_{i_1}\cdots\dChf_{i_j}\cdot\hwt[k] = \bs_{i_1}\cdots\bs_{i_j}\cdot\hwt[k]=\bw\cdot\hwt[k]$ by Corollary \ref{cor:Action_of_s_e_f}, since none of the factors act as the zero map. This has two implications. Firstly, $w=s_{i_1}\cdots s_{i_j}$ is a reduced expression: else, one of the factors of $\dChf_{i_1}\cdots\dChf_{i_j}$ must act as the identity map, which contradicts Theorem \ref{thm:Green_structure_minuscule_reps} \textbf{(iii)}. Secondly, the weight of $\bw\cdot\hwt[k]$ can also be written as $\dfwt[k]-\sdr_{i_1}-\ldots-\sdr_{i_j}$.

We conclude that 
\[
\dfwt[k] - \sdr_{i'_1}-\cdots-\sdr_{i'_c}=w(\dfwt[k])=\dfwt[k]-\sdr_{i_1}-\ldots-\sdr_{i_j}.
\]
Clearly, this can only hold when $c=\ell(w_c)=\ell(w)=j$.

Thus, we have $w = w_c\in\cosets$ and we have already shown that $s_{i_1}\cdots s_{i_j}$ is a reduced expression for $w$. Moreover, $w(\dfwt[k])=\wt$ by definition and this determines $w$ uniquely, proving \textbf{(iii)}.

\textbf{(iv)} As $\wop\in\weylp$, \textbf{(i)} implies that $\bwop$, $\dwop$ and their inverses acts trivially on $\hwt[k]$. As $\wP$ is defined as the minimal coset representative of $\wo$, we conclude that $\lwt[k]=\bwo\cdot\hwt[k]=\bwP\cdot\hwt[k]$ by \textbf{(i)}, and since $\wP=s_{r_1}\cdots s_{r_\ellwP}$ is the reduced expression we fixed in section \ref{sec:notation}, \textbf{(ii)} implies that $\lwt[k] = \bs_{r_1}\cdots\bs_{r_\ellwP}\cdot\hwt[k] = \dChf_{r_1}\cdots\dChf_{r_\ellwP}\cdot\hwt[k]$. Part \textbf{(iii)} implies the last statement directly.
\end{proof}

\section{Statement of the Laurent polynomial potential}\label{sec:LP_for_LG}
In this section we state the main result of this article, Theorem \ref{thm:Explicit_LP_LG_model}, which is an explicit Laurent polynomial expression for $\potZ$ on an open, dense algebraic torus $\opendecomps$ inside $\decomps$, whenever $\X=\G/\P$ is a \emph{cominuscule} homogeneous space. Recall that $\X=\G/\P$ is cominuscule when $\P=\P_k$ is maximal and $\dfwt[k]$ is minuscule.

First we define another subset of $\dunim$:
\begin{df}\label{df:opendunim}
Recall the reduced expression $\wP = s_{r_1}\cdots s_{r_\ellwP}$ fixed in equation \eqref{eq:Fixed_reduced_expression_for_wP_and_wop}. Let $\opendunim\subset\dunim$ be the algebraic torus of elements $u_-$ that can be written as
\begin{equation}
u_- = \dy_{r_\ellwP}(a_\ellwP) \cdots \dy_{r_1}(a_1)
\label{eq:Elements_of_Bruhat_cell_in_U-}
\end{equation}
with $a_i\in\C^*$. 
\end{df}
\begin{lem}\label{lem:open_dunim_inside_dunimP}
We have $\opendunim\subset\dunimP$ open and dense, where $\dunimP=\dunim\cap\dborelp\bwop\bwo\dborelp$ was defined in equation \eqref{eq:unique_decomps_unipotents}.
\end{lem}
\begin{proof}
(See also section 5.2 of \cite{Pech_Rietsch_Odd_Quadrics}.) Note that 
\[
\dborelp\ds_i\dborelp = \dborelp\dx_i(1)\dy_i(-1)\dx_i(1)\dborelp = \dborelp\dy_i(-1)\dborelp,
\]
so that $\dy_i(a)\in\dborelp\ds_i\dborelp$, as the $-1$ can be scaled to any $a\in\C^*$ since $\dtorus\subset\dborelp$. Of course, $s_{r_\ellwP}\cdots s_{r_1}=\wPinv$ is reduced, so the Bruhat lemma (see for example \cite{Humphreys_Linear_Algebraic_Groups}, Lemma 29.3.A) implies that
\[
\opendunim\subset\dborelp\ds_{r_\ellwP}\cdots\ds_{r_1}\dborelp = \dborelp\bwPinv\dborelp
\]
and it remains to show that $\dborelp\bwPinv\dborelp=\dborelp\bwop\bwo\dborelp$. This follows from the fact that $\wo=\wP\wop$ and the fact that $\ds_i$ and $\bs_i$ only differ by a torus element.

It is clear that $\opendunim\subset\dunimP$ is an open subset and it is dense as both have dimension $\ell(\wop\wo)$ and $\dunimP$ is irreducible \cite{Lusztig_Total_Positivity_in_Reductive_Groups}.
\end{proof}
\begin{df}\label{df:opendecomps}
We define the open, dense algebraic torus $\opendecomps\subset\decomps$ as
\[
\opendecomps = \dborelm\bwo^{-1}\cap\dunip\invdtorus\bwop\opendunim~ \subset~ \decomps.
\]
\end{df}
Note that the fact that this variety is non-zero and an algebraic torus follows from Lemma \ref{lem:z_has_unique_decomposition} and the resulting isomorphism $\decomps\to\dunimP\times\invdtorus:z\mapsto(u_-,t)$.

The following is immediate from the definition and Lemma \ref{lem:z_has_unique_decomposition}:
\begin{cor}\label{cor:opendecoms_unique_decomp}
Every $z\in\opendecomps$ can be factorized in two ways: as $b_-\bwo^{-1}$ for $b_-\in\dborelm$; and as $z=u_+t\bwop u_-$ with $u_-\in\opendunim$ of the form \eqref{eq:Elements_of_Bruhat_cell_in_U-}, $u_+\in\dunip$ and $t\in\invdtorus$. Moreover, the latter decomposition is \emph{unique} with $u_+$ determined by a choice of $(u_-,t)$.
\end{cor}

We will define the Laurent polynomial expression for the potential on this algebraic torus $\opendecomps$. It turns out that this expression is indexed by the subexpressions of $\wPrime$ in $\wP$. Recall from equation \eqref{eq:Def_of_wPrime} that $\wPrime\in\weyl$ is defined by $\wP=\wPrime\wPPrime$, where $\wP,\wPPrime\in\cosets$ are the minimal coset representatives of $\wo\weylp$ and $\wop s_k\weylp$ (with $k$ such that $\P=\P_k$). Moreover, their lengths are denoted by $\ell(\wP)=\ellwP$, $\ell(\wPrime)=\ellwPrime$ and $\ell(\wPPrime)=\ellwPPrime$ and satisfy $\ellwP=\ellwPrime+\ellwPPrime$.
\begin{df}\label{df:wPrimeSubExp}
Let $\wPrimeSubExp$ be the set indexing reduced subexpressions for $\wPrime$ occurring inside the fixed reduced expression $\wP=s_{r_1}\cdots s_{r_\ellwP}$ of equation \eqref{eq:Fixed_reduced_expression_for_wP_and_wop}. In other words:
\[
\wPrimeSubExp = \bigl\{(i_1,\ldots,i_{\ellwPrime})~\big|~1\le i_1<i_2<\ldots<i_{\ellwPrime}\le\ellwP\text{ and }\wPrime=s_{r_{i_1}}\cdots s_{r_{i_{\ellwPrime}}}\bigr\}.
\]
\end{df}
\begin{rem}\label{rem:NoFixedRedExpForwPrime}
Note that the reduced expression $\wPrime = s_{r_{i_1}}\cdots s_{r_{i_{\ellwPrime}}}$ is \emph{not} fixed, i.e.~if $(i_j),(i'_j)\in\wPrimeSubExp$, then we do \emph{not} necessarily have $r_{i_j} = r_{i'_j}$ for all $j$.
\end{rem}
\begin{thm}[An explicit Laurent-polynomial Landau-Ginzburg model]\label{thm:Explicit_LP_LG_model}
Let $\X=\G/\P$ be a \emph{cominuscule} complete homogeneous space with $\G$ a \emph{simply-connected}, simple, complex algebraic group and $\P=\P_k$ a (maximal) parabolic subgroup. The~restriction $\potZo$ of $\potZ$ to $\opendecomps$ has the following Laurent polynomial expression:
\begin{equation}
\potZo(z) = \sum_{i=1}^\ellwP a_i + q\frac{\sum_{(i_j)\in\wPrimeSubExp} a_{i_1}\cdots a_{i_{\ellwPrime}}}{\prod_{i=1}^\ellwP a_i}.
\label{eq:Explicit_LP_expression_for_the_potential}
\end{equation}
Here $z\in\opendecomps$ is uniquely decomposed as $z=u_+t\bwop u_-$ with $u_-=\dy_{r_\ellwP}(a_\ellwP)\cdots\dy_{r_1}(a_1)\in\opendunim$ as in Corollary \ref{cor:opendecoms_unique_decomp}, and $q\in\C^*$ is given by $q=\sdr_k(t)$ (with $t\in\invdtorus$). Finally, the set $\wPrimeSubExp$ is given in Definition \ref{df:wPrimeSubExp}.
\end{thm}
The proof of this statement follows in sections \ref{sec:Proof_of_LP_for_LG} and \ref{sec:Proof_of_Lemmas}. In section \ref{sec:quiver_enumeration} we rewrite the summation over $\wPrimeSubExp$ into a summation over subsets of a quiver associated to $\wP$ by \cite{CMP_Quantum_cohomology_of_minuscule_homogeneous_spaces}, see Corollary \ref{cor:Explicit_LP_LG_model_with_quiver_subsets}. In section \ref{sec:Application} we apply Theorem \ref{thm:Explicit_LP_LG_model} and its reformulation as Corollary \ref{cor:Explicit_LP_LG_model_with_quiver_subsets} to all the cominuscule homogeneous spaces, leading to new Laurent polynomial potentials for the cominuscule homogeneous spaces of type $\LGD_n$, $\LGE_6$ and $\LGE_7$, see subsections \ref{subsec:OG}, \ref{subsec:Cayley} and \ref{subsec:Freudenthal}, respectively. We also work out an example for each of the families of cominuscule homogeneous spaces.

\section{Proof of the Laurent polynomial expression}\label{sec:Proof_of_LP_for_LG}
In section \ref{sec:LP_for_LG} we stated our main result in Theorem \ref{thm:Explicit_LP_LG_model}. This section is dedicated to proving this theorem. 
Before we get started on the proof, let us introduce the following notation:
\begin{df}
Write $\dfwtrep$ for the irreducible representation of $\dg$ with highest weight $\dfwt$, and denote by $\hwt$ a choice of a highest weight vector. Denote by $\lwt = \bwo\cdot\hwt$ the associated lowest weight vector. Note that the weight space of $\hwt$ is one-dimensional, so that the projection of an arbitrary $v\in\dfwtrep$ to this weight space (parallel to the other weight spaces) is a scalar multiple of $\hwt$; we denote this scalar by $\lan v,\hwt\ran$.
\end{df}
\begin{rem}\label{rem:Lifts_to_universal_cover}
Recall that every representation of a Lie algebra induces a representation of the associated simply-connected Lie group. Here, the highest weight representation $\dfwtrep$ of $\dg$ induces a representation of the universal cover $\udG$ of $\dG$. Since this representation does not always descend to a representation of $\dG$, we need to work on $\udG$ instead. Because we identified $\dunip$ and $\dunim$ with their universal covers, we consider the factors $u_+$ and $u_-$ of $z=u_+t\bwo u_-\in\opendecomps$ as elements of $\udG$. The same holds for the one-parameter subgroups $\dx_j(a)\in\dunip$ and $\dy_j(a)\in\dunim$ for $j\in\{1,\ldots,n\}$, which we also consider as elements of $\udG$. However, the elements $\ds_i$ and $\bs_i$ of $\dG$ associated to $s_i\in\weyl$ have multiple lifts to $\udG$; we choose the lifts to be the elements obtained by taking the product of the one-parameter subgroups in $\udG$, i.e.
\[
\ds_i = \dx_i(1)\dy_i(-1)\dx_i(1)~\in~\udG \qand \bs_i = \dx_i(-1)\dy_i(1)\dx_i(-1)~\in~\udG.
\]
Note that we abuse notation and denote these lifts in the same way as the original elements. Also note that with these choices we still have $\bs_i=\ds_i^{-1}$. The elements $\dw,\bw\in\udG$ associated to $w\in\weyl$ are similarly defined by $\dw=\ds_{i_1}\cdots\ds_{i_j}$ and $\bw=\bs_{i_1}\cdots\bs_{i_j}$ respectively, where $w=s_{i_1}\cdots s_{i_j}$ is a reduced expression. This fixes the lift of $z=u_+t\bwo u_-$ up to a choice of lift of $t\in\invdtorus\subset\dtorus$; the choices differ by a factor in $\ker(\sdr_k)\subset\udtorus$, so all the lifts have $\sdr_k(t)=q$ and we choose one arbitrarily (we will continue to abuse notation and also denote the lift of $t\in\invdtorus$ by $t\in\udtorus$).
\end{rem}
\begin{rem}
Expressions of the form $\lan g\cdot\hwt,\hwt\ran$ are a priori only defined for $g\in\udG$, so expressions of that form will always assume the group element $g$ to be elements of the universal cover $\udG$. Thus, the abuse of notation in Remark \ref{rem:Lifts_to_universal_cover} should not give rise to ambiguity.
\end{rem}
The proof of Theorem \ref{thm:Explicit_LP_LG_model} requires a few intermediate results which we will prove in section~\ref{sec:Proof_of_Lemmas}. Assuming for the moment that these hold, the proof of Theorem \ref{thm:Explicit_LP_LG_model} is a straightforward computation:
\begin{proof}[Proof of Theorem \ref{thm:Explicit_LP_LG_model}]\label{pf:Proof_of_the_explicit_LP_expression}
We want to find an expression for 
\[
\potZo(z) = \SumOfEs(u_+^{-1}) + \SumOfFs(u_-) = - \SumOfEs(u_+) + \SumOfFs(u_-)
\]
in terms of the toric coordinates of $z\in\opendecomps$. 
First, we calculate $\SumOfFs(u_-)$:
\begin{lem}\label{lem:F(u-)=sum_of_variables}
For $u_-\in\opendunim$ we have $\SumOfFs(u_-) = \sum_{i=1}^\ellwP a_i$.
\end{lem}
Thus, we only need to find the term involving the quantum parameter, which comes from $\SumOfEs(u_+)=\sum_{i=1}^n\dChedual(u_+)$. We can rewrite each of the summands of this term as follows:
\begin{lem}\label{lem:expression_for_e_star_i}
For $z=u_+t\bwop u_-\in\opendecomps$ we have
\[
\dChedual(u_+) =  \sdr_i(t) \dfrac{\lan\bwo^{-1} u_-^{-1} \bwop^{-1}\bs_i\cdot\hwt,\hwt\ran}{\lan\bwo^{-1} u_-^{-1} \bwop^{-1}\cdot\hwt,\hwt\ran}.
\]
\end{lem}
It turns out that \emph{except} for $i=k$, these summands do not contribute to the sum:
\begin{lem}\label{lem:e_star_i_is_zero_for_i_unequal_to_k}
For $z=u_+t\bwop u_-\in\opendecomps$ and $i\neq k$ (where $k$ is such that $\P=\P_k$) we have $\dChedual(u_+) = 0$.
\end{lem}
Altogether, we have now found that
\[
\SumOfEs(u_+) = \dChedual[_k](u_+) = q\frac{\lan\bwo^{-1} u_-^{-1} \bwop^{-1}\bs_k\cdot\hwt[k],\hwt[k]\ran}{\lan\bwo^{-1} u_-^{-1} \bwop^{-1}\cdot\hwt[k],\hwt[k]\ran},
\]
where $q=\sdr_k(t)$. Now we of course need to calculate the denominator and numerator of this quotient:
\begin{lem}\label{lem:denominator_of_e_star_k}
 For $u_-\in\opendunim$ we have $\lan\bwo^{-1} u_-^{-1} \bwop^{-1}\cdot\hwt[k],\hwt[k]\ran=(-1)^\ellwP\prod_{i=1}^\ellwP a_i$, where $\ellwP=\ell(\wP)$.
\end{lem}
\begin{lem}\label{lem:numerator_of_e_star_k}
For $u_-\in\opendunim$ we have 
\[
\lan\bwo^{-1} u_-^{-1} \bwop^{-1}\bs_k\cdot\hwt[k],\hwt[k]\ran = (-1)^{\ellwP+1}\sum_{(i_j)\in\wPrimeSubExp} a_{i_1}\cdots a_{i_{\ellwPrime}},
\]
where $\wPrimeSubExp = \{(i_1,\ldots,i_{\ellwPrime})~|~1\le i_1<i_2<\ldots<i_{\ellwPrime}\le\ellwP \text{ and } \wPrime=s_{r_{i_1}}\cdots s_{r_{i_{\ellwPrime}}}\}$ and where we fixed the reduced expression $\wP=s_{r_1}\cdots s_{r_\ellwP}$ in section \ref{sec:notation}.
\end{lem}
Thus, we find
\[
\SumOfEs(u_+) = -q\frac{\sum_{(i_j)\in\wPrimeSubExp} a_{i_1}\cdots a_{i_{\ellwPrime}}}{\prod_{i=1}^\ellwP a_i}.
\]
Inserting this together with the expression for $\SumOfFs(u_-)$ into $\potZo(z)=-\SumOfEs(u_+)+\SumOfFs(u_-)$, we obtain the statement of Theorem \ref{thm:Explicit_LP_LG_model}.
\end{proof}

\section{Proof of the intermediate results}\label{sec:Proof_of_Lemmas}
We only need to prove the lemmas stated in the proof of Theorem \ref{thm:Explicit_LP_LG_model} to conclude the result.
\begin{rem}
The five intermediate lemmas in the proof of Theorem \ref{thm:Explicit_LP_LG_model} are generalizations of Lemma 5.5 of \cite{Pech_Rietsch_Odd_Quadrics}. Lemmas \ref{lem:F(u-)=sum_of_variables}, \ref{lem:expression_for_e_star_i} and \ref{lem:e_star_i_is_zero_for_i_unequal_to_k} follow similar reasoning in the general case. The proof of Lemma \ref{lem:denominator_of_e_star_k} requires a modification using the general structure of minuscule representations, while Lemma \ref{lem:numerator_of_e_star_k} follows a different line of reasoning than its counterpart in \cite{Pech_Rietsch_Odd_Quadrics}.
\end{rem}
The first lemma is very straightforward:
\begin{proof}[Proof of Lemma \ref{lem:F(u-)=sum_of_variables}.]
We want to show that $\SumOfFs(u_-) = \sum_{i=1}^\ellwP a_i$ for $u_-\in\opendunim$.

Recall from equation \eqref{eq:e*_and_f*} that $\dChfdual(y_j(a))=a\de_{ij}$. From this it follows that
\[
\dChfdual(u_-) = \dChfdual\bigl(\dy_{r_\ellwP}(a_\ellwP)\cdots\dy_{r_1}(a_1)\bigr) = a_1\de_{i,r_1} + \ldots + a_\ellwP\de_{i,r_\ellwP}.
\]
Summing over all $i\in\{1,\ldots,n\}$, we find $\SumOfFs(u_-) = \sum_{i=1}^n\dChfdual(u_-) = \sum_{i=1}^\ellwP a_i$.
\end{proof}
Unfortunately, the other term, $\SumOfEs(u_+)=\sum_{i=1}^n\dChedual(u_+)$, will not be as easy. We will first reformulate each of the terms $\dChe_i(u_+)$:
\begin{proof}[Proof of Lemma \ref{lem:expression_for_e_star_i}.]
We want to show that
\[
\dChedual(u_+) =  \sdr_i(t) \dfrac{\lan\bwo^{-1} u_-^{-1} \bwop^{-1}\bs_i\cdot\hwt,\hwt\ran}{\lan\bwo^{-1} u_-^{-1} \bwop^{-1}\cdot\hwt,\hwt\ran}.
\]
for $z=u_+t\bwop u_-\in\opendecomps$.

First, note that the map $\dunip\to\C:u\mapsto\lan u\bs_i\cdot\hwt,\hwt\ran$ is equal to the unique homomorphism $\dChedual$ sending $\dx_i(a)=\exp(a\,\dChe_i)$ to $a$ and the other one-parameter-subgroups to zero, so
\[
\dChedual(u_+) = \lan u_+\bs_i\cdot\hwt,\hwt\ran.
\]

Next, we use the fact that $u_+$ is a factor in the decomposition of $z\in\opendecomps$ as $z=u_+t\bwop u_-$ to find a decomposition for $u_+$ itself. By definition, we have $z=b_-\bwo^{-1}$ for some $b_-\in\dborelm$ as $\opendecomps\subset\dborelm\bwo^{-1}$. In Remark \ref{rem:Lifts_to_universal_cover} we fixed lifts of $z=u_+t\bwop u_-$ and the elements $\ds_i$ and $\bs_i$ to $\udG$. Thus, there is a unique lift of $b_-\in\dborelm$, also denoted by $b_-\in\udborelm$, such that $b_-\bwo^{-1}=z\in\udG$. This gives
\begin{equation}
u_+ = b_-\bwo^{-1} u_-^{-1}\bwop^{-1} t^{-1} ~\in~\udG.
\label{eq:decomposition_for_u+}
\end{equation}
Thus, we have to calculate $\lan b_-\bwo^{-1} u_-^{-1}\bwop^{-1} t^{-1}\bs_i\cdot\hwt,\hwt\ran$.

Now, $\bs_i\cdot\hwt$ has weight $\dfwt-\sdr_i$, so $t^{-1}\in\udtorus$ acts on this vector by scalar multiplication with
\[
[\dfwt-\sdr_i](t^{-1}) = \frac{\dfwt(t^{-1})}{\sdr_i(t^{-1})} = \frac{\sdr_i(t)}{\dfwt(t)}~\in~\C.
\]
(Note that weights and roots are written additively.) We conclude
\[
\dChedual(u_+) = \sdr_i(t)\frac{\lan b_-\bwo^{-1} u_-^{-1}\bwop^{-1} \bs_i\cdot\hwt,\hwt\ran}{\dfwt(t)}.
\]

Noting that $b_-\in\udborelm$ sends a vector to a linear combination of vectors of \emph{equal or lower} weight, we see that the only contribution of $b_-$ to $\lan b_-\bwo^{-1} u_-^{-1}\bwop^{-1} \bs_i\cdot\hwt,\hwt\ran$ will be the factor $\lan b_-\cdot\hwt,\hwt\ran$, so we find
\begin{equation}
\dChedual(u_+) = \sdr_i(t)\frac{\lan b_-\cdot\hwt,\hwt\ran}{\dfwt(t)}\,\lan\bwo^{-1} u_-^{-1}\bwop^{-1} \bs_i\cdot\hwt,\hwt\ran.
\label{eq:e*_of_u+_intermediate}
\end{equation}

Finally, we use the decomposition in \eqref{eq:decomposition_for_u+} together with the fact that $\lan u_+\cdot\hwt,\hwt\ran = 1$ (as $u_+\in\dunip$) to conclude that
\[
1 = \lan u_+\cdot\hwt,\hwt\ran = \lan b_-\bwo^{-1} u_-^{-1}\bwop^{-1} t^{-1}\cdot\hwt,\hwt\ran = \frac{\lan b_-\cdot\hwt,\hwt\ran}{\dfwt(t)}\,\lan\bwo^{-1} u_-^{-1}\bwop^{-1} \cdot\hwt,\hwt\ran,
\]
where we calculated the contributions of $b_-$ and $t^{-1}$ in an analogous way as above. Substituting this into \eqref{eq:e*_of_u+_intermediate}, we obtain
\[
\dChedual(u_+) =  \sdr_i(t) \dfrac{\lan\bwo^{-1} u_-^{-1} \bwop^{-1}\bs_i\cdot\hwt,\hwt\ran}{\lan\bwo^{-1} u_-^{-1} \bwop^{-1}\cdot\hwt,\hwt\ran},
\]
as in the statement of the lemma.
\end{proof}

Lemma \ref{lem:e_star_i_is_zero_for_i_unequal_to_k} claims that all of the summands of $\SumOfEs(u_+)=\sum_{i=1}^n\dChedual(u_+)$ are zero, except for $i=k$ (where $k$ is such that $\P=\P_k$).
\begin{proof}[Proof of Lemma \ref{lem:e_star_i_is_zero_for_i_unequal_to_k}] We need to show that $\dChedual(u_+) = 0$ for $z=u_+t\bwop u_-\in\opendecomps$ and $i\neq k$ (where $k$ is such that $\P=\P_k$).

Considering the expression for $\dChedual(u_+)$ of Lemma \ref{lem:expression_for_e_star_i}, we need to show that
\[
\lan\bwo^{-1} u_-^{-1} \bwop^{-1}\bs_i\cdot\hwt,\hwt\ran = 0 \qfor i\neq k.
\]
Recall that we assumed $u_-\in\opendunim\subset\dunimP=\dunim\cap\dborelp\bwop\bwo\dborelp$ (see Lemma \ref{lem:open_dunim_inside_dunimP}). Thus, we have $u_-^{-1}\in\dborelp\bwo^{-1}\bwop^{-1}\dborelp$; in other words, there are $b_1,b_2\in\dborelp$ such that $u_-^{-1} = b_1\bwo^{-1}\bwop^{-1} b_2$. Choosing lifts $b_1,b_2\in\udborelp$ such that $u_-^{-1} = b_1\bwo^{-1}\bwop^{-1} b_2$ as elements of $\udG$ (again abusing notation), it follows that we have to show that
\[
\lan\bwo^{-1} b_1\bwo^{-1}\bwop^{-1} b_2\bwop^{-1}\bs_i\cdot\hwt,\hwt\ran = 0 \qfor i\neq k.
\]
Now, $\bwo^{-1} b_1\bwo^{-1}\in\bwo^{-1}\udborelp\bwo^{-1} =\udborelm$, so using an analogous argument to the one in the proof of Lemma \ref{lem:expression_for_e_star_i}, we find that 
\[
\lan\bwo^{-1} b_1\bwo^{-1}\bwop^{-1} b_2\bwop^{-1}\bs_i\cdot\hwt,\hwt\ran = \lan\bwo^{-1} b_1\bwo^{-1}\cdot\hwt,\hwt\ran\lan\bwop^{-1} b_2\bwop^{-1}\bs_i\cdot\hwt,\hwt\ran.
\]
As $\lan\bwo^{-1} b_1\bwo^{-1}\cdot\hwt,\hwt\ran$ only contributes a scalar factor, we need to show that
\begin{equation}
\lan\bwop^{-1} b_2\bwop^{-1}\bs_i\cdot\hwt,\hwt\ran = 0 \qfor i\neq k.
\label{eq:e_star_i_is_zero_intermediate}
\end{equation}
In other words, it is enough to show that $\bwop^{-1} b_2\bwop^{-1}\bs_i\cdot\hwt$ has no components of weight $\dfwt$. This is a straightforward argument with weights: $\bs_i\cdot\hwt$ has weight $\dfwt-\sdr_i$, so that $\bwop^{-1}\bs_i\cdot\hwt$ has weight $\wop\bigl(\dfwt-\sdr_i\bigr)$, noting that $\wop^{-1}=\wop$ as it is the longest element of the Weyl group $\weylp=\lan s_i~|~i\neq k\ran$. Now, as $b_2\in\udborelp$, all components of $b_2\bwop^{-1}\bs_i\cdot\hwt$ will have weight $\wop\bigl(\dfwt-\sdr_i\bigr)+\drt_+$ for some (possibly trivial) sum $\drt_+$ of positive roots. Thus, we find that $\bwop^{-1} b_2\bwop^{-1}\bs_i\cdot\hwt$ has components of weight
\[
\wop\Bigl(\wop\bigl(\dfwt-\sdr_i\bigr)+\drt_+\Bigr) = \dfwt-\sdr_i+\wop(\drt_+),
\]
again using $\wop^{-1}=\wop$. Thus, one of these components has weight $\dfwt$ if and only if $\wop(\drt_+)=\sdr_i$. However, since $\wop$ is the longest element of the Weyl group $\weylp=\lan s_i~|~i\neq k\ran$, we know that it maps all the simple roots $\sdr_i$ with $i\neq k$ to negative roots, but then we must have that $\drt_+=\wop(\sdr_i)\in\ndroots$ is a \emph{negative} root and definitely not a sum of positive roots, which gives a contradiction. Thus, all components have weight unequal to $\dfwt$, implying that \eqref{eq:e_star_i_is_zero_intermediate} holds, which in turn implies the lemma.
\end{proof}

Combining Lemmas \ref{lem:expression_for_e_star_i} and \ref{lem:e_star_i_is_zero_for_i_unequal_to_k}, we conclude that
\begin{equation}
\SumOfEs(u_+) = \dChedual[_k](u_+) = q\frac{\lan\bwo^{-1} u_-^{-1} \bwop^{-1}\bs_k\cdot\hwt[k],\hwt[k]\ran}{\lan\bwo^{-1} u_-^{-1} \bwop^{-1}\cdot\hwt[k],\hwt[k]\ran}.
\label{eq:e*_k_intermediate_as_fraction}
\end{equation}
Lemma \ref{lem:denominator_of_e_star_k} calculates the denominator of this quotient and Lemma \ref{lem:numerator_of_e_star_k} calculates its numerator:
\begin{proof}[Proof of Lemma \ref{lem:denominator_of_e_star_k}.]
We need to show that for $u_-\in\opendunim$ we have 
\[
\lan\bwo^{-1} u_-^{-1} \bwop^{-1}\cdot\hwt[k],\hwt[k]\ran=(-1)^\ellwP\prod_{i=1}^\ellwP a_i,
\]
where $\ellwP=\ell(\wP)$.

Using Lemma \ref{lem:action_of_weyl_elts} \textbf{(iv)} we find
\[
\lan\bwo^{-1} u_-^{-1} \bwop^{-1}\cdot\hwt[k],\hwt[k]\ran = \lan\bwo^{-1} u_-^{-1}\cdot\hwt[k],\hwt[k]\ran.
\]
By definition of $\opendunim$ (see Definition \ref{df:opendunim}), $u_-^{-1}$ has a decomposition of the form
\[
u_-^{-1} = \dy_{r_1}(-a_1) \cdots \dy_{r_\ellwP}(-a_\ellwP),
\]
where the sequence of indices $(r_1,\ldots,r_\ellwP)$ is the same as the one used in the reduced expression $\wP=s_{r_1}\cdots s_{r_\ellwP}$ fixed in equation \eqref{eq:Fixed_reduced_expression_for_wP_and_wop}. 
Now, $\dy_i(a) = \exp(a\,\dChf_i) = 1 + a\,\dChf_i + \frac12 a^2(\dChf_i)^2+\ldots$, but only the first two terms act non-trivially on the representation, since $(\dChf_i)^2\cdot v=0$ for all $v\in\dfwtrep[k]$ according to Theorem \ref{thm:Green_structure_minuscule_reps} \textbf{(iv)}. We conclude that
\[
u_-^{-1}\cdot\hwt[k] = (1-a_1\,\dChf_{r_1})\cdots(1-a_\ellwP\,\dChf_{r_\ellwP})\cdot\hwt[k].
\]
Note that this is a sum of vectors of different weights, the term of highest weight being $\hwt[k]$ (obtained by taking the term with all the identity factors), and the term of lowest weight being (see Lemma \ref{lem:action_of_weyl_elts} \textbf{(iv)})
\[
(-a_1)\cdots(-a_\ellwP)\dChf_{r_1}\cdots\dChf_{r_\ellwP}\cdot\hwt[k] = (-1)^\ellwP\bigl(\tprod[_{i=1}^\ellwP]a_i\bigr)\lwt[k].
\]
Only the lowest weight term contributes a coefficient to $\lan\bwo^{-1} u_-^{-1}\cdot\hwt[k],\hwt[k]\ran$ as $\bwo^{-1}\lwt[k]=\hwt[k]$ and $\bwo^{-1}$ is a bijection. Thus, we obtain
\[
\lan\bwo^{-1} u_-^{-1} \bwop^{-1}\cdot\hwt[k],\hwt[k]\ran = (-1)^\ellwP\tprod[_{i=1}^\ellwP] a_i
\]
as we wanted to show.
\end{proof}

Now we turn to the numerator of \eqref{eq:e*_k_intermediate_as_fraction}:

\begin{proof}[Proof of Lemma \ref{lem:numerator_of_e_star_k}.]
We need to show that for $u_-\in\opendunim$ we have 
\[
\lan\bwo^{-1} u_-^{-1} \bwop^{-1}\bs_k\cdot\hwt[k],\hwt[k]\ran = (-1)^{\ellwP+1}\sum_{(i_j)\in\wPrimeSubExp} a_{i_1}\cdots a_{i_{\ellwPrime}},
\]
where $\wPrimeSubExp = \{(i_1,\ldots,i_{\ellwPrime})~|~1\le i_1<i_2<\ldots<i_{\ellwPrime}\le\ellwP \text{ and } \wPrime=s_{r_{i_1}}\cdots s_{r_{i_{\ellwPrime}}}\}$ was defined in Definition \ref{df:wPrimeSubExp} and where we fixed the reduced expression $\wP=s_{r_1}\cdots s_{r_\ellwP}$ in equation \eqref{eq:Fixed_reduced_expression_for_wP_and_wop}.

As we saw in the proof of Lemma \ref{lem:denominator_of_e_star_k}, we only need to consider the lowest-weight term of the vector $u_-^{-1} \bwop^{-1}\bs_k\cdot\hwt[k]$, as it is the only term mapped to $\hwt[k]$ by $\bwo^{-1}$. However, $\bwop^{-1}\bs_k$ acts non-trivially on $\hwt[k]$, whereas in the proof of Lemma \ref{lem:denominator_of_e_star_k} $\bwop^{-1}$ acted trivially on $\hwt[k]$ by Lemma \ref{lem:action_of_weyl_elts} \textbf{(iv)}.

Recall that we had fixed the reduced expression $\wop = s_{q_1}\cdots s_{q_m}$ in equation \eqref{eq:Fixed_reduced_expression_for_wP_and_wop} and that $\bs_i^{-1} = \ds_i$, so we find that $\bwop^{-1} = \ds_{q_m}\cdots\ds_{q_1}$. Moreover, note that $\bs_k\cdot\hwt[k] = \dChf_k\cdot\hwt[k] = -\ds_k\cdot\hwt[k]$ (see Corollary \ref{cor:Action_of_s_e_f}). All in all, we find that
\[
\bwop^{-1}\bs_k\cdot\hwt[k] = - \ds_{q_m}\cdots\ds_{q_1}\ds_k\cdot\hwt[k].
\]
In section \ref{sec:structure_of_minuscule_representations} we have written $\wPPrime\in\cosets$ for the minimal coset representative of the coset $\wop s_k\weylp$ (note that $\wop^{-1}=\wop$) and written $\ellwPPrime=\ell(\wPPrime)\le\ell(\wP)=\ellwP$ for its length. Since $s_{q_m}\cdots s_{q_1}$ is a reduced expression for $\wop\in\weylp$ and $s_k\notin\weylp$, we deduce that $s_{q_m}\cdots s_{q_1}s_k$ is a reduced expression for $\wop s_k$. Thus, by Lemma \ref{lem:action_of_weyl_elts} \textbf{(i)} and \textbf{(ii)}, we deduce that $\bwop^{-1}\bs_k\cdot\hwt[k] = -\ds_{q_m}\cdots\ds_{q_1}\ds_k\cdot\hwt[k] = -\dwop\ds_k\cdot\hwt[k]$ and that
\begin{equation}
\bwop^{-1}\bs_k\cdot\hwt[k] = -\dwop\ds_k\cdot\hwt[k] = -\dwPPrime\cdot\hwt[k] = (-1)^{\ellwPPrime+1} \dChf_{j_1}\cdots\dChf_{j_{\ellwPPrime}}\cdot\hwt[k],
\label{eq:wop_sk}
\end{equation}
where $\wPPrime = s_{j_1}\cdots s_{j_{\ellwPPrime}}$ is a reduced expression.

Note that in the case $\ellwPPrime=\ellwP$ (which only occurs for $\CP^n=\Gr(1,n+1)=\LGA_n^\SC/\P_1\cong\LGA_n^\SC/\P_n$) the following arguments become trivial, see Remark \ref{rem:The_Case_ellwPPrime_equals_ellwP} below. 

Next, we need to multiply this vector by $u_-^{-1}$, which due to Theorem \ref{thm:Green_structure_minuscule_reps} \textbf{(iv)} reduces to multiplying by $(1-a_1\,\dChf_{r_1})\cdots(1-a_\ellwP\,\dChf_{r_\ellwP})$ in the representation. As we mentioned at the start of the proof, we only need to look at the coefficient in front of $\lwt[k]$ in this product. From Lemma \ref{lem:action_of_weyl_elts} \textbf{(iv)} we know that $\dChf_{i_1}\cdots\dChf_{i_\ellwP}\cdot\hwt[k]=\lwt[k]$ if and only if $s_{i_1}\cdots s_{i_\ell}=\wP$ and this is a reduced expression. So we need exactly those terms of $(1-a_1\,\dChf_{r_1})\cdots(1-a_\ellwP\,\dChf_{r_\ellwP})$ that complete $\dChf_{j_1}\cdots\dChf_{j_{\ellwPPrime}}\cdot\hwt[k]$ to $\dChf_{r_{i_1}}\cdots\dChf_{r_{i_{\ellwP-\ellwPPrime}}}\dChf_{j_1}\cdots\dChf_{j_{\ellwPPrime}}\cdot\hwt[k]$ in such a way that the indices satisfy $s_{r_{i_1}}\cdots s_{r_{i_{\ellwP-\ellwPPrime}}}s_{j_1}\cdots s_{j_{\ellwPPrime}}=\wP$. However, $s_{j_1}\cdots s_{j_{\ellwPPrime}}=\wPPrime$ and $s_{r_1}\cdots s_{r_\ell}=\wP$, so we obtain a contributing term for every subexpression $s_{r_{i_1}}\cdots s_{r_{i_{\ellwP-\ellwPPrime}}}$ of the fixed reduced expression for $\wP$ such that $s_{r_{i_1}}\cdots s_{r_{i_{\ellwP-\ellwPPrime}}}\wPPrime = \wP$. Therefore, every subexpression of $\wPrime=\wP(\wPPrime)^{-1}$ in the fixed reduced expression of $\wP$ gives a contributing term.

Now, in Definition \ref{df:wPrimeSubExp} we defined the set indexing these subexpressions as
\[
\wPrimeSubExp = \{(i_1,\ldots,i_{\ellwPrime})~|~1\le i_1<i_2<\ldots<i_{\ellwPrime}\le\ellwP \text{ and } \wPrime=s_{r_{i_1}}\cdots s_{r_{i_{\ellwPrime}}}\},
\]
where $\ellwPrime=\ell(\wPrime)=\ell(\wP)-\ell(\wPPrime)=\ellwP-\ellwPPrime$. In conclusion, for every $(i_1,\ldots,i_{\ellwP})\in\wPrimeSubExp$, we obtain the following term
\[
-(-a_{i_1})\cdots(-a_{i_{\ellwPrime}})\dChf_{r_{i_1}}\cdots\dChf_{r_{i_{\ellwPrime}}}(-1)^{\ellwPPrime}\dChf_{j_1}\cdots\dChf_{j_{\ellwPPrime}}\cdot\hwt[k] = (-1)^{\ellwP+1}a_{i_1}\cdots a_{i_{\ellwPrime}}\,\lwt[k]
\]
and we find that 
\[
\lan\bwo^{-1} u_-^{-1} \bwop^{-1}\bs_k\cdot\hwt[k],\hwt[k]\ran = (-1)^{\ellwP+1}\sum_{(i_j)\in\wPrimeSubExp} a_{i_1}\cdots a_{i_{\ellwPrime}},
\]
as we wanted to show.
\end{proof}
\noindent This concludes the last of the intermediate results for the proof of Theorem \ref{thm:Explicit_LP_LG_model}.
\begin{rem}
Note that at no point in the proof of Lemma \ref{lem:numerator_of_e_star_k} do we fix a reduced expression for~$\wPrime$, so if $(i_j),(i'_j)\in\wPrimeSubExp$, then we do \emph{not} necessarily have $r_{i_j} = r_{i'_j}$ for all $j$.
\end{rem}
\begin{rem}\label{rem:The_Case_ellwPPrime_equals_ellwP}
Note that if $\ellwPPrime=\ellwP$, we have $\wPPrime=\wP$. In this case, equation \eqref{eq:wop_sk} becomes
\[
\bwop^{-1}\bs_k\cdot\hwt[k] = -\ds_{q_m}\cdots\ds_{q_1}\ds_k\cdot\hwt[k] = -\dwP\cdot\hwt[k] = (-1)^{\ellwP+1} \dChf_{r_1}\cdots\dChf_{r_{\ellwP}}\cdot\hwt[k]=(-1)^{\ellwP+1}\lwt[k],
\]
Since $u_-\in\dunim$, we find that $u_-^{-1}$ acts trivially on this, so that 
\[
\lan\bwo^{-1} u_-^{-1} \bwop^{-1}\bs_k\cdot\hwt[k],\hwt[k]\ran = (-1)^{\ellwP+1}.
\]
Of course, $\wPPrime=\wP$ implies that $\wPrime=\wP(\wPPrime)^{-1}=1$. Thus, subexpressions of $\wPrime$ inside $\wP$ have zero length and there is only one such subexpression so we find $\wPrimeSubExp=\{\varnothing\}$. We conclude that $\sum_{(i_j)\in\wPrimeSubExp}\prod_{j=1}^\ellwPrime a_{r_{i_j}} = 1$, taking the empty product to be $1$. Thus,
\[
\lan\bwo^{-1} u_-^{-1} \bwop^{-1}\bs_k\cdot\hwt[k],\hwt[k]\ran = (-1)^{\ellwP+1}\sum_{(i_j)\in\wPrimeSubExp}\prod_{j=1}^\ellwPrime a_{r_{i_j}}
\]
also holds in case $\ellwPPrime=\ellwP$.
\end{rem}

\section{Reformulating the quantum term using quiver subsets}\label{sec:quiver_enumeration}
In the last two sections we proved that Theorem \ref{thm:Explicit_LP_LG_model} gives a local Laurent polynomial expression for the potential constructed by Rietsch in \cite{Rietsch_Mirror_Construction}. However, the drawback of the current expression is the effort required to find all the subexpressions of $\wPrime$ inside the fixed reduced expression of $\wP$. In this section we will use a quiver associated to $\wP$ to enumerate all these subexpressions. For this, we need to use the fact that both $\wP$ and $\wPrime$ are \emph{fully commutative}:
\begin{df}
An element $w\in\weyl$ is called \emph{fully commutative} if every reduced expression of $w$ can be obtained from a given reduced expression by commuting its factors.
\end{df}
\begin{lem}\label{lem:wP_and_wPrime_are_fully_commutative}
Both $\wP\in\cosets$ and $\wPrime\in\weyl$ are fully commutative.
\end{lem}
\begin{proof}
Full commutativity of $\wP$ (and in fact of every element of $\cosets$) follows from Theorem 6.1 of \cite{Stembridge_Fully_commutative_elements_of_Coxeter_groups}. Full commutativity of $\wPrime$ is now a consequence of the full commutativity of $\wP$ due to Proposition 2.4 of \cite{Stembridge_Fully_commutative_elements_of_Coxeter_groups}, which states that every element in $W$ obtained from a fully commutative element by removing simple reflections at the right (or left) is fully commutative itself.
\end{proof}

In \cite{CMP_Quantum_cohomology_of_minuscule_homogeneous_spaces} a quiver is associated to $\wP$ using the full commutativity property. This quiver is a modification of the quiver introduced in \cite{Perrin_Small_resolutions_of_minuscule_Schubert_varieties}. 
It is defined as follows:
\begin{df}[\cite{CMP_Quantum_cohomology_of_minuscule_homogeneous_spaces}, Definition 2.1]\label{df:PerrinQuiver}
Given a fixed reduced expression $\wP=s_{r_1}\cdots s_{r_{\ell}}$, e.g.~the one fixed in equation \eqref{eq:Fixed_reduced_expression_for_wP_and_wop}.
\begin{itemize}
	\item For $\rtb\in\{\sr_1,\ldots,\sr_n\}$, let $\NumOcc(\rtb)$ be the number of occurrences of $s_\rtb$ in the reduced expression, i.e.~$\NumOcc(\rtb)=\#\{j~|~s_{r_j}=s_\rtb\}$.
	\item For $(\rtb,j)$ such that $1\le j \le\NumOcc(\rtb)$, let $\wPindex(\rtb,j)$ be the index of the $j$-th occurrence of $s_\rtb$ in the reduced expression (from left to right), i.e.~let $\wPindex(\rtb,j)$ be the index such that $s_{r_{\wPindex(\rtb,j)}}=s_\rtb$ and $\#\{\jtil\le\wPindex(\rtb,j)~|~s_{r_{\jtil}}=s_\rtb\}=j$. Also, set $\wPindex(\rtb,0)=0$ and $\wPindex(\rtb,\NumOcc(i)+1)=\infty$.
\end{itemize}
The quiver is now defined as follows:
\begin{itemize}
	\item Draw for the $j$-th occurrence of $s_\rtb$ in the reduced expression for $\wP$ a vertex labeled $(\rtb,j)$, i.e.~the vertices are $(\rtb,j)$ for $\rtb\in\{\sr_1,\ldots,\sr_n\}$ and $j\in\{1,\ldots,\NumOcc(\rtb)\}$.
	\item Draw an arrow from $(\rtb,j)$ to $({\rtb'},j')$ if $s_\rtb$ and $s_{\rtb'}$ do not commute and if the $j'$-th occurrence of $s_{\rtb'}$ is the is the first one to the right of the $j$-th occurrence of $s_\rtb$ in the reduced expression for $\wP$ and it occurs before the $(j+1)$-th occurrence of $s_\rtb$, i.e.~draw an arrow from $(\rtb,j)$ to $({\rtb'},j')$ if $(s_\rtb s_{\rtb'})^2\neq1$ and $\wPindex({\rtb'},j'-1)<\wPindex(\rtb,j)<\wPindex({\rtb'},j')<\wPindex(\rtb,j+1)$.
\end{itemize}
We will denote the resulting quiver by $\PerrinQ_{\X}$.
\end{df}
Note that the resulting quiver does not depend on the reduced expression for $\wP$ as it is fully commutative; it suffices to check that the quiver is the same after commuting two simple reflections. 
\begin{rem}
In Definition \ref{df:PerrinQuiver}, we associate the quiver $\PerrinQ_\X$ to the variety $\X=\G/\P_k$. However, as the Laurent polynomial $\potZo$ is defined on $\opendecomps\subset\dG$, it would be more proper to associate the quiver to the variety $\dP_k\backslash\dG$ (the left-quotient of $\dG$ by $\dP_k$). This variety has not played a role here, but is closely related to the variety $\mX$ of equation \eqref{eq:Richardson_var} and Theorem \ref{thm:Lie-theoretic_LG_model}. In fact, in the articles \cite{Rietsch_Marsh_Grassmannians,Pech_Rietsch_Odd_Quadrics,Pech_Rietsch_Williams_Quadrics,Pech_Rietsch_Lagrangian_Grassmannians} it is shown that $\mX$ is isomorphic to a subvariety of $\dP_k\backslash\dG$ and an expression for the pull-back of $\pot:\mX\times\invdtorus\to\C$ (see Definition \ref{df:potential}) to $\dP_k\backslash\dG$ is given. It is expected that such an isomorphic subvariety exists in general. For now, however, we note that $\G$ and $\dG$ have the same Weyl groups $\weyl$ and $\weylp=\lan s_i~|~i\neq k\ran\subset\weyl$ as these only depend on the Coxeter diagram underlying the Dynkin diagram of $\dG$ and the vertex $k$, so that $\wP$ and the associated quivers are actually the same. Thus, we will continue to write $\PerrinQ_\X$ even though it would be more proper to write $\PerrinQ_{\dP_k\backslash\dG}$.
\end{rem}

This quiver has as vertices the factors of the reduced expression of $\wP$, so every subexpression of $\wPrime$ inside the reduced expression will become a subset of vertices of this quiver.
\begin{df}\label{df:wPrimeSubsets}
We denote the set of subsets of vertices of the quiver $\PerrinQ_{\X}$ that are associated to reduced subexpressions of $\wPrime$ inside the reduced expression of $\wP$ fixed in section \ref{sec:notation} by $\wPrimeSubsets$. In other words,
\[
\wPrimeSubsets = \Bigl\{\bigl((\rtb_1,j_1),\ldots,(\rtb_\ellwPrime,j_\ellwPrime)\bigr)~\Big|~s_{\rtb_1}\cdots s_{\rtb_{\ellwPrime}} = \wPrime\Bigr\},
\]
where $\ellwPrime=\ell(\wPrime)$. Note that we implicitly order the vertices $(\rtb_1,j_1),\ldots,(\rtb_\ellwPrime,j_\ellwPrime)$ such that $\wPindex(\rtb_1,j_1)<\ldots<\wPindex(\rtb_\ellwPrime,j_\ellwPrime)$, but we will still refer to the elements of $\wPrimeSubsets$ as \emph{subsets}.\footnote{This is to distinguish them from the \emph{sequences} of subindices that are the elements of $\wPrimeSubExp$ from Definition \ref{df:wPrimeSubExp}.}
\end{df}
\begin{rem}\label{rem:Bijection_between_Subsets_and_Subsequences}
Recall from Definition \ref{df:wPrimeSubExp} that we defined $\wPrimeSubExp$ as the set of sequences of subindices $(i_1,\ldots,i_{\ellwPrime})$ such that $s_{r_{i_1}}\cdots s_{r_{i_\ellwPrime}}=\wPrime$ is a reduced subexpression of $\wPrime$ in the fixed reduced expression $\wP=s_{r_1}\cdots s_{r_\ellwP}$. Note that $\wPindex$ gives rise to a bijection between $\wPrimeSubsets$ and $\wPrimeSubExp$, which we will also denote by $\wPindex$:
\begin{equation}
\wPindex:\bigl((\rtb_1,j_1),\ldots,(\rtb_\ellwPrime,j_\ellwPrime)\bigr)\mapsto\bigl(\wPindex(\rtb_1,j_1),\ldots,\wPindex(\rtb_\ellwPrime,j_\ellwPrime)\bigr),
\label{eq:Bijection_Subsets_to_Subsequences}
\end{equation}
This bijection gives the translation between the subsets of vertices of the quiver $\PerrinQ_{\X}$ (in $\wPrimeSubsets$) and their associated reduced subexpressions (in $\wPrimeSubExp$).
\end{rem}

\begin{ex}\label{ex:MiniGrassmannian}
To illustrate the quiver and the subsets associated to the subexpressions, consider the example of the Grassmannian $\X=\Gr(4,6)=\SL_6/\P_4$ of type $\LGA_5$. Fixing for $\wP$ the reduced expression $\wP=(s_2s_3s_4s_5)(s_1s_2s_3s_4)$, we find that $\PerrinQ_X$ is of the form below. Here all edges are arrows are pointing downwards and the $i$-th column of vertices contains the vertices $(\sr_i,j)$ with $j$ increasing from $1$ to $m^P(\sr_i)$ from top to bottom. Above the quiver, we have drawn the labeled Coxeter diagram in such a way that the $i$-th vertex of the diagram is above the $i$-th column of the quiver. The vertex labeled $4$ is marked in the Coxeter diagram to signify that we are considering $\X=\SL_6/\P_4$, i.e.~to signify that $k=4$ in $\X=\G/\P_k$. For each vertex $(\rtb,j)$, we also give the value of $\wPindex(\rtb,j)$.
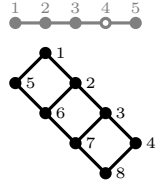
\[
\begin{tikzpicture}[scale = .4, style= very thick]
	\draw[gray]
	(-1,3) -- (3,3);
	\draw[gray, fill=gray]
	(-1,3) circle (.15)
	(0,3) circle (.15)
	(1,3) circle (.15)
	(3,3) circle (.15);
	\draw[gray, fill=white]
	(2,3) circle (.15);
	\draw[gray] 
	(-1,3) node [above]{\tiny$1$}
	(0,3) node [above]{\tiny$2$}
	(1,3) node [above]{\tiny$3$}
	(2,3) node [above]{\tiny$4$}
	(3,3) node [above]{\tiny$5$};
	\draw 
	(0,2) -- (3,-1) -- (2,-2) -- (-1,1) -- (0,2)
	(1,1) -- (0,0) 	(2,0) -- (1,-1);
	\draw[black, fill=black] 
	(0,2) circle (.15)
	(1,1) circle (.15)
	(2,0) circle (.15)
	(3,-1) circle (.15)
	(2,-2) circle (.15)
	(1,-1) circle (.15)
	(0,0) circle (.15)
	(-1,1) circle (.15);
	\draw
	(0,2) node [right]{\tiny$1$}
	(1,1) node [right]{\tiny$2$}
	(2,0) node [right]{\tiny$3$}
	(3,-1) node [right]{\tiny$4$}
	(2,-2) node [right]{\tiny$8$}
	(1,-1) node [right]{\tiny$7$}
	(0,0) node [right]{\tiny$6$}
	(-1,1) node [right]{\tiny$5$};
\end{tikzpicture}
\]
The reduced expression $\wPrime=s_2s_3s_4$ is unique, and the subexpressions $(\textcolor{red}{\underline{\boldsymbol{s_2s_3s_4}}}s_5)(s_1s_2s_3s_4)$, $(\textcolor{red}{\underline{\boldsymbol{s_2s_3}}}s_4s_5)(s_1s_2s_3\textcolor{red}{\underline{\boldsymbol{s_4}}})$, $(\textcolor{red}{\underline{\boldsymbol{s_2}}}s_3s_4s_5)(s_1s_2\textcolor{red}{\underline{\boldsymbol{s_3s_4}}})$, and $(s_2s_3s_4s_5)(s_1\textcolor{red}{\underline{\boldsymbol{s_2s_3s_4}}})$ correspond in the quiver $\PerrinQ_{\X}$ to marking the vertices
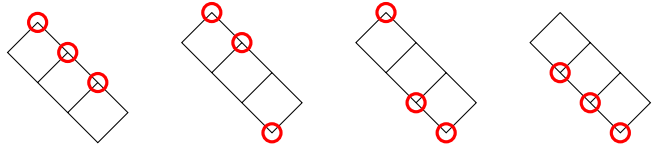
\[
\begin{tikzpicture}[scale = .4]
	\draw
	(0,2) -- (3,-1) -- (2,-2) -- (-1,1) -- (0,2)
	(1,1) -- (0,0) 	(2,0) -- (1,-1);
	\draw[red,style = very thick] 
	(0,2) circle (.3)
	(1,1) circle (.3)
	(2,0) circle (.3);
\end{tikzpicture}
\qquad
\begin{tikzpicture}[scale = .4]
	\draw
	(0,2) -- (3,-1) -- (2,-2) -- (-1,1) -- (0,2)
	(1,1) -- (0,0) 	(2,0) -- (1,-1);
	\draw[red,style = very thick] 
	(0,2) circle (.3)
	(1,1) circle (.3)
	(2,-2) circle (.3);
\end{tikzpicture}
\qquad
\begin{tikzpicture}[scale = .4]
	\draw
	(0,2) -- (3,-1) -- (2,-2) -- (-1,1) -- (0,2)
	(1,1) -- (0,0) 	(2,0) -- (1,-1);
	\draw[red,style = very thick] 
	(0,2) circle (.3)
	(1,-1) circle (.3)
	(2,-2) circle (.3);
\end{tikzpicture}
\qquad
\begin{tikzpicture}[scale = .4]
	\draw
	(0,2) -- (3,-1) -- (2,-2) -- (-1,1) -- (0,2)
	(1,1) -- (0,0) 	(2,0) -- (1,-1);
	\draw[red,style = very thick] 
	(0,0) circle (.3)
	(1,-1) circle (.3)
	(2,-2) circle (.3);
\end{tikzpicture}
\]
respectively, where we suppressed the vertices.
\end{ex}

It turns out that we can use the quiver $\PerrinQ_{\X}$ to find all the reduced subexpressions of $\wPrime$ in $\wP$ using two operations that are straightforward when considered as operations on subsets of the quiver. First, we need the following observation:
\begin{rem}
From Definition \ref{df:PerrinQuiver}, it is clear that we can only have an arrow $(i,j)\to(i',j')$ if $\wPindex(i,j)<\wPindex(i',j')$, so that any path $(i,j)\to(i_1,j_1)\to\ldots\to(i',j')$ between the two vertices correspond to simple reflections $s_{r_{\wPindex(i,j)}}$, $s_{r_{\wPindex(i_1,j_1)}}$, \ldots, $s_{r_{\wPindex(i',j')}}$ that appear in that order in the reduced expression of $\wP$.
\end{rem}
The two operations are given as follows:
\begin{lem}\label{lem:legal_subexpression_moves}
Suppose $(i_1,\ldots,i_{\ellwPrime})\in\wPrimeSubExp$, i.e.~suppose it is a sequence of subindices such that $s_{r_{i_1}}\cdots s_{r_{i_\ellwPrime}} = \wPrime$ is a reduced subexpression inside the reduced expression $\wP=s_{r_1}\cdots s_{r_\ellwP}$. Denote by $S\in\wPrimeSubsets$ the corresponding subset obtained by (the inverse of) the bijection in \eqref{eq:Bijection_Subsets_to_Subsequences}.
\begin{romanize}
\item For every $\itil$ such that $r_{\itil}=r_{i_j}$ and $i_{j-1}<\itil<i_{j+1}$, the sequence of subindices obtained by replacing $i_j$ with $\itil$, i.e.~$(i_1,\ldots,i_{j-1},\itil,i_{j+1},\ldots,i_\ellwPrime)$, also gives a reduced subexpression of $\wPrime$ in $\wP$. 

In terms of subsets in $\wPrimeSubsets$, this says that we can replace a vertex $(\rtb,j)\in S$ with a vertex $(\rtb,\jtil)$ if every $(\rtb',j')$ with $\wPindex(\rtb,j)<\wPindex(\rtb',j')<\wPindex(\rtb,\jtil)$ is \emph{not} an element of $S$ when $j<\jtil$ (or $\wPindex(\rtb,j)>\wPindex(\rtb',j')>\wPindex(\rtb,\jtil)$ when $j>\jtil$). 

Examples of this operation are the following (using the conventions of Example \ref{ex:MiniGrassmannian}):
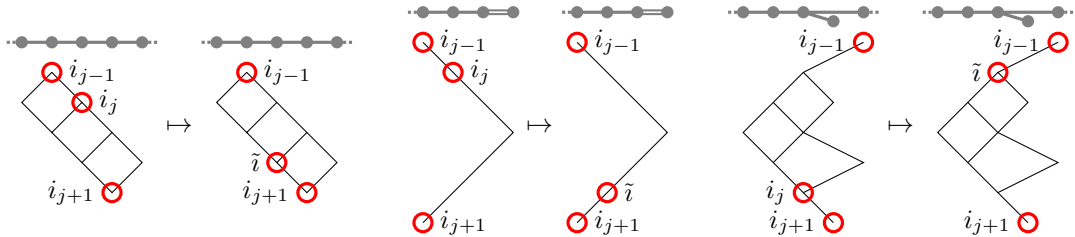
\[
\begin{tikzpicture}[scale = .4,baseline=0em]
	\draw[gray, very thick] 	(-1,3) -- (3,3);
	\draw[gray, very thick, densely dotted] (-1.5,3) -- (-1,3)		(3.5,3) -- (3,3);
	\draw[gray, very thick, fill=gray]
	(-1,3) circle (.15)
	(0,3) circle (.15)
	(1,3) circle (.15)
	(2,3) circle (.15)
	(3,3) circle (.15);
	\draw 
	(0,2) -- (3,-1) -- (2,-2) -- (-1,1) -- (0,2)
	(1,1) -- (0,0) 	(2,0) -- (1,-1);
	\draw[red,style = very thick]
	(0,2) circle (.3)
	(1,1) circle (.3)
	(2,-2) circle (.3);
	\node at (0,2) [right=.3em]{$i_{j-1}$};
	\node at (1,1) [right=.3em]{$i_{j}$};
	\node at (2,-2) [left=.3em]{$i_{j+1}$};
\end{tikzpicture}
\mapsto
\begin{tikzpicture}[scale = .4,baseline=0em]
	\draw[gray, very thick] 	(-1,3) -- (3,3);
	\draw[gray, very thick, densely dotted] (-1.5,3) -- (-1,3)		(3.5,3) -- (3,3);
	\draw[gray, very thick, fill=gray]
	(-1,3) circle (.15)
	(0,3) circle (.15)
	(1,3) circle (.15)
	(2,3) circle (.15)
	(3,3) circle (.15);
	\draw
	(0,2) -- (3,-1) -- (2,-2) -- (-1,1) -- (0,2)
	(1,1) -- (0,0) 	(2,0) -- (1,-1);
	\draw[red, style = very thick]
	(0,2) circle (.3)
	(1,-1) circle (.3)
	(2,-2) circle (.3);
	\node at (0,2) [right=.3em]{$i_{j-1}$};
	\node at (1,-1) [left=.3em]{$\itil$};
	\node at (2,-2) [left=.3em]{$i_{j+1}$};
\end{tikzpicture} \qquad \begin{tikzpicture}[scale=.4,baseline=0em]
	\draw[gray, very thick] 	(-1,4) -- (1,4);
	\draw[gray, thick, double] (1,4) -- (2,4);
	\draw[gray, very thick, densely dotted] (-1.5,4) -- (-1,4);
	\draw[gray, very thick, fill=gray]
	(-1,4) circle (.15)
	(0,4) circle (.15)
	(1,4) circle (.15)
	(2,4) circle (.15);
	\draw (-1,3) -- (2,0) -- (-1,-3);
	\draw[red,style=very thick] (-1,3) circle (.3);
	\draw[red,style=very thick] (0,2) circle (.3);
	\draw[red,style=very thick] (-1,-3) circle (.3);
	\node at (-1,3) [right=.3em]{$i_{j-1}$};
	\node at (0,2) [right=.3em]{$i_{j}$};
	\node at (-1,-3) [right=.3em]{$i_{j+1}$};
\end{tikzpicture} 
\mapsto 
\begin{tikzpicture}[scale=.4,baseline=0em]
	\draw[gray, very thick] 	(-1,4) -- (1,4);
	\draw[gray, thick, double] (1,4) -- (2,4);
	\draw[gray, very thick, densely dotted] (-1.5,4) -- (-1,4);
	\draw[gray, very thick, fill=gray]
	(-1,4) circle (.15)
	(0,4) circle (.15)
	(1,4) circle (.15)
	(2,4) circle (.15);
	\draw (-1,3) -- (2,0) -- (-1,-3);
	\draw[red,style=very thick] (-1,3) circle (.3);
	\draw[red,style=very thick] (0,-2) circle (.3);
	\draw[red,style=very thick] (-1,-3) circle (.3);
	\node at (-1,3) [right=.3em]{$i_{j-1}$};
	\node at (0,-2) [right=.3em]{$\itil$};
	\node at (-1,-3) [right=.3em]{$i_{j+1}$};
\end{tikzpicture}\qquad \begin{tikzpicture}[scale=.4,baseline=0em]
	\draw[gray, very thick] 	(-2,4) -- (2,4) (0,4) -- (1,3.7);
	\draw[gray, very thick, densely dotted] (-2.5,4) -- (-2,4)		(2.5,4) -- (2,4);
	\draw[gray, very thick, fill=gray]
	(-2,4) circle (.15)
	(-1,4) circle (.15)
	(0,4) circle (.15)
	(1,3.7) circle (.15)
	(2,4) circle (.15);
	\draw (0,2) -- (-2,0) -- (0,-2) -- (2,-1) -- (0,0) -- (1,1) -- (0,2);
	\draw (2,3) -- (0,2)  (0,-2) -- (1,-3);
	\draw (-1,1) -- (0,0) -- (-1,-1);
	\draw[red,very thick] (2,3) circle (.3);
	\draw[red,very thick] (0,-2) circle (.3);
	\draw[red,very thick] (1,-3) circle (.3);
	\node at (2,3) [left=.3em]{$i_{j-1}$}; 
	\node at (0,-2) [left=.3em]{$i_j$}; 
	\node at (1,-3) [left=.3em]{$i_{j+1}$}; 
\end{tikzpicture} \mapsto \begin{tikzpicture}[scale=.4,baseline=0em]
	\draw[gray, very thick] 	(-2,4) -- (2,4) (0,4) -- (1,3.7);
	\draw[gray, very thick, densely dotted] (-2.5,4) -- (-2,4)		(2.5,4) -- (2,4);
	\draw[gray, very thick, fill=gray]
	(-2,4) circle (.15)
	(-1,4) circle (.15)
	(0,4) circle (.15)
	(1,3.7) circle (.15)
	(2,4) circle (.15);
	\draw (0,2) -- (-2,0) -- (0,-2) -- (2,-1) -- (0,0) -- (1,1) -- (0,2);
	\draw (2,3) -- (0,2)  (0,-2) -- (1,-3);
	\draw (-1,1) -- (0,0) -- (-1,-1);
	\draw[red,very thick] (2,3) circle (.3);
	\draw[red, very thick] (0,2) circle (.3);
	\draw[red,very thick] (1,-3) circle (.3);
	\node at (2,3) [left=.3em]{$i_{j-1}$}; 
	\node at (0,2) [left=.3em]{$\itil$}; 
	\node at (1,-3) [left=.3em]{$i_{j+1}$}; 
\end{tikzpicture}
\]
\item For every $\itil<i_j$ with $r_{\itil}=r_{i_j}$ such that there exists a $j'$ with $i_{j'}<\itil<i_{j'+1}\le i_j$ and $(s_{r_{i_j}}s_{r_{i_{\jtil}}})^2=1$ for all $\jtil\in\{j'+1,\ldots,j-1\}$, we have that the sequence of subindices $(i_1,\ldots,i_{j'},\itil,i_{j'+1},\ldots,i_{j-1},i_{j+1},\ldots,i_\ellwPrime)$ is an element of $\wPrimeSubExp$ as well.

Similarly, for every $\itil>i_j$ with $r_{\itil}=r_{i_j}$ such that there exists a $j'$ with $i_j\le i_{j'}<\itil<i_{j'+1}$ and $(s_{r_{i_j}}s_{r_{i_{\jtil}}})^2=1$ for all $\jtil\in\{j+1,\ldots, j'\}$, the sequence of subindices $(i_1,\ldots,i_{j-1},i_{j+1},\ldots,i_{j'},\itil,i_{j'+1},\ldots,i_\ellwPrime)$ is also an element of $\wPrimeSubExp$.

In terms of subsets in $\wPrimeSubsets$, this says that we can replace a vertex $(\rtb,j)\in S$ with a vertex $(\rtb,\jtil)$ if for every path $(\rtb,j)\to(\rtb_1,j_1)\to\ldots\to(\rtb,\jtil)$ when $j<\jtil$ (or $(\rtb,\jtil)\to(\rtb_1,j_1)\to\ldots\to(\rtb,j)$ when $j>\jtil$ respectively) there is no vertex $(\rtb',j')\in S$ contained in the path such that $(s_\rtb s_{\rtb'})^2\neq1$. Examples of this operation are:
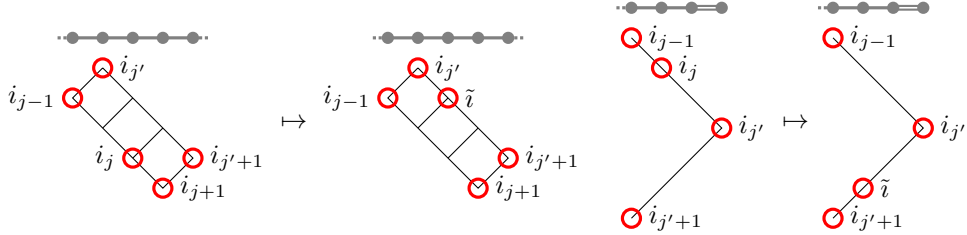
\[
\begin{tikzpicture}[scale = .4, baseline=0em]
	\draw[gray, very thick] 	(-1,3) -- (3,3);
	\draw[gray, very thick, densely dotted] (-1.5,3) -- (-1,3)		(3.5,3) -- (3,3);
	\draw[gray, very thick, fill=gray]
	(-1,3) circle (.15)
	(0,3) circle (.15)
	(1,3) circle (.15)
	(2,3) circle (.15)
	(3,3) circle (.15);
	\draw (0,2) -- (3,-1) -- (2,-2) -- (-1,1) -- (0,2);
	\draw (1,1) -- (0,0); 	\draw (2,0) -- (1,-1);
	\draw[red, very thick]
	(0,2) circle (.3)
	(1,-1) circle (.3)
	(-1,1) circle (.3)
	(3,-1) circle (.3)
	(2,-2) circle (.3);
	\node at (0,2) [right=.3em]{$i_{j'}$};
	\node at (3,-1) [right=.3em]{$i_{j'+1}$};
	\node at (-1,1) [left=.3em]{$i_{j-1}$};
	\node at (1,-1) [left=.3em]{$i_j$};
	\node at (2,-2) [right=.3em]{$i_{j+1}$};
\end{tikzpicture}
\mapsto
\begin{tikzpicture}[scale = .4, baseline=0em]
	\draw[gray, very thick] 	(-1,3) -- (3,3);
	\draw[gray, very thick, densely dotted] (-1.5,3) -- (-1,3)		(3.5,3) -- (3,3);
	\draw[gray, very thick, fill=gray]
	(-1,3) circle (.15)
	(0,3) circle (.15)
	(1,3) circle (.15)
	(2,3) circle (.15)
	(3,3) circle (.15);
	\draw (0,2) -- (3,-1) -- (2,-2) -- (-1,1) -- (0,2);
	\draw (1,1) -- (0,0); 	\draw (2,0) -- (1,-1);
	\draw[red, very thick]
	(0,2) circle (.3)
	(1,1) circle (.3)
	(-1,1) circle (.3)
	(3,-1) circle (.3)
	(2,-2) circle (.3);
	\node at (0,2) [right=.3em]{$i_{j'}$};
	\node at (3,-1) [right=.3em]{$i_{j'+1}$};
	\node at (-1,1) [left=.3em]{$i_{j-1}$};
	\node at (1,1) [right=.3em]{$\itil$};
	\node at (2,-2) [right=.3em]{$i_{j+1}$};
\end{tikzpicture}
\quad 
\begin{tikzpicture}[scale=.4,baseline=0em]
	\draw[gray, very thick] 	(-1,4) -- (1,4);
	\draw[gray, thick, double] (1,4) -- (2,4);
	\draw[gray, very thick, densely dotted] (-1.5,4) -- (-1,4);
	\draw[gray, very thick, fill=gray]
	(-1,4) circle (.15)
	(0,4) circle (.15)
	(1,4) circle (.15)
	(2,4) circle (.15);
	\draw (-1,3) -- (2,0) -- (-1,-3);
	\draw[red, very thick]
	(2,0) circle (.3)
	(0,2) circle (.3);
	\draw[red,style=very thick] (-1,3) circle (.3);
	\draw[red,style=very thick] (-1,-3) circle (.3);
	\node at (-1,3) [right=.3em]{$i_{j-1}$};
	\node at (0,2) [right=.3em]{$i_{j}$};
	\node at (2,0) [right=.3em]{$i_{j'}$};
	\node at (-1,-3) [right=.3em]{$i_{j'+1}$};
\end{tikzpicture} 
\mapsto
\begin{tikzpicture}[scale=.4, baseline=0em]
	\draw[gray, very thick] 	(-1,4) -- (1,4);
	\draw[gray, thick, double] (1,4) -- (2,4);
	\draw[gray, very thick, densely dotted] (-1.5,4) -- (-1,4);
	\draw[gray, very thick, fill=gray]
	(-1,4) circle (.15)
	(0,4) circle (.15)
	(1,4) circle (.15)
	(2,4) circle (.15);
	\draw (-1,3) -- (2,0) -- (-1,-3);
	\draw[red, very thick]
	(2,0) circle (.3);
	\draw[red,style=very thick] (-1,3) circle (.3);
	\draw[red,style=very thick] (0,-2) circle (.3);
	\draw[red,style=very thick] (-1,-3) circle (.3);
	\node at (-1,3) [right=.3em]{$i_{j-1}$};
	\node at (2,0) [right=.3em]{$i_{j'}$};
	\node at (0,-2) [right=.3em]{$\itil$};
	\node at (-1,-3) [right=.3em]{$i_{j'+1}$};
\end{tikzpicture}
\]
\end{romanize}
\end{lem}
\begin{proof}
\textbf{(i)} By assumption, $(i_1,\ldots,i_{j-1},\itil,i_{j+1},\ldots,i_\ellwPrime)$ is an increasing sequence of subindices such that $s_{r_{i_1}}\cdots s_{r_{i_{j-1}}} s_{r_{\itil}} s_{r_{i_{j+1}}}\cdots s_{r_{i_\ellwPrime}} = s_{r_{i_1}}\cdots s_{r_{i_{j-1}}} s_{r_{i_j}} s_{r_{i_{j+1}}}\cdots s_{r_{i_\ellwPrime}} = \wPrime$, and therefore a reduced expression of $\wPrime$ in $\wP$ as well.

\textbf{(ii)} If $\itil<i_j$, we can commute the factor $s_{r_{i_j}}$ to the left in the given reduced subexpression to obtain the reduced expression $s_{r_{i_1}}\cdots s_{r_{i_{j'}}}s_{r_{i_j}}s_{r_{i_{j'+1}}}\cdots s_{r_{i_{j-1}}}s_{r_{i_{j+1}}}\cdots s_{r_{i_\ellwPrime}}$ for $\wPrime$. Moreover, this is a subexpression of $\wPrime$ in $\wP$ since there is a $\itil$ with $i_{j'}<\itil<i_{j'+1}$ and $r_{\itil}=r_{i_j}$, so that $(i_1,\ldots, i_{j'},\itil,i_{j'+1},\ldots,i_{j-1},i_j,\ldots,i_\ellwPrime)$ is an increasing sequence of subindices.

The case $\itil>i_j$ is analogous, except that we commute the factor $s_{r_{i_j}}$ to the right.
\end{proof}
\begin{rem}
Note that the operation \textbf{(i)} is actually a special case of the operation \textbf{(ii)} where no commutation takes place.
\end{rem}
It turns out that these operations suffice to obtain \emph{all} the reduced subexpressions of $\wPrime$ in $\wP$. To show this, we introduce a total order on the set of reduced subexpressions by taking the lexicographical order $\lex$ on $\wPrimeSubExp$. In other words, we have $(i_1,\ldots,i_\ellwPrime)\lex(i'_1,\ldots,i'_\ellwPrime)$ if and only if there exists a $j$ such that $i_j < i'_j$ and $i_{j'}=i'_{j'}$ for $j'\in\{1,\ldots,j-1\}$. Let $(\ihat_1,\ldots,\ihat_\ellwPrime)$ be the minimal sequence associated to a reduced subexpression of $\wPrime$ in $\wP$.

\begin{prop}
Every reduced subexpression of $\wPrime$ in $\wP$ can be obtained using the operations of Lemma \ref{lem:legal_subexpression_moves} on the minimal reduced subexpression $\wPrime=s_{r_{\ihat_1}}\cdots s_{r_{\ihat_\ellwPrime}}$.
\end{prop}
\begin{proof}
We will show that every non-minimal reduced subexpression can be made smaller using one of the operations in Lemma \ref{lem:legal_subexpression_moves}. This gives a sequence of operations from any given reduced subexpression to the minimal one. Since it is evident that each operation is invertible, we obtain the statement.

Therefore, let $(i_1,\ldots, i_{\ellwPrime})\in\wPrimeSubExp$ be non-minimal with respect to the lexicographical order. By definition, there exists a $j$ such that $i_j>\ihat_j$ and $i_{j'}=\ihat_{j'}$ for all $j'\in\{1,\ldots,j-1\}$. We distinguish two cases: $r_{i_j}=r_{\ihat_j}$ and $r_{i_j}\neq r_{\ihat_j}$.

In the case $r_{i_j}=r_{\ihat_j}$, we can apply operation \textbf{(i)} to directly obtain the reduced subexpression with subindices $(i_1,\ldots,i_{j-1},\ihat_j,i_{j+1},\ldots,i_\ellwPrime)\lex(i_1,\ldots,i_{j-1},i_j,i_{j+1},\ldots,i_{\ellwPrime})$.

Now, consider the case $r_{i_j}\neq r_{\ihat_j}$. Suppose that the simple reflection $s_{r_{\ihat_j}}$ occurs for the $N$-th time in the minimal reduced subexpression of $\wPrime$ in $\wP$, then there must exist some $j'>j$ such that $s_{r_{i_{j'}}}$ is the $N$-th occurrence of the same simple reflection in the subexpression $(i_1,\ldots i_\ellwPrime)$. Indeed, $\wPrime$ is fully commutative so that each simple reflection appears the same number of times in each reduced expression. (Note that we have $j'>j$ since $s_{r_{\ihat_{\jtil}}}=s_{r_{i_{\jtil}}}$ for $\jtil\in\{1,\ldots,j-1\}$ and $s_{r_{i_j}}\neq s_{r_{\ihat_j}}$.) We also know that the simple reflections $s_{r_{i_{\jtil}}}$ commute with $s_{r_{i_{j'}}}$ for $\jtil\in\{j,\ldots,j'-1\}$ because of full commutativity of $\wPrime$, since these simple reflections have commuted with $s_{r_{\ihat_j}}=s_{r_{i_{j'}}}$ going from $(\ihat_1,\ldots,\ihat_\ellwPrime)$ to $(i_1,\ldots,i_\ellwPrime)$. Thus, we can apply operation \textbf{(ii)} to obtain $(i_1,\ldots,i_{j-1},\ihat_j,i_j,\ldots,i_{j'-1},i_{j'+1},\ldots,i_\ellwPrime)\lex(i_1,\ldots, i_\ellwPrime)$ as a reduced subexpression.
\end{proof}

Combining this with Theorem \ref{thm:Explicit_LP_LG_model} and Remark \ref{rem:Bijection_between_Subsets_and_Subsequences}, we conclude the following:

\begin{cor}\label{cor:Explicit_LP_LG_model_with_quiver_subsets}
Let $\X=\G/\P$ be a \emph{cominuscule} complete homogeneous space with $\G$ a \emph{simply-connected}, simple, complex algebraic group and $\P=\P_k$ a (maximal) parabolic subgroup. The~restriction $\potZo$ of $\potZ$ to $\opendecomps$ has the following Laurent polynomial expression:
\begin{equation}
\potZo(z) = \sum_{i=1}^\ellwP a_i + q\frac{\sum_{S\in\wPrimeSubsets}\prod_{(\rtb,j)\in S} a_{\wPindex(\rtb,j)}}{\prod_{i=1}^\ellwP a_i}.
\label{eq:LP_expression_for_the_potential_using_quiver_subsets}
\end{equation}
Here $z\in\opendecomps$ is uniquely decomposed as $z=u_+t\bwop u_-$ with $u_-=\dy_{r_\ellwP}(a_\ellwP)\cdots\dy_{r_1}(a_1)\in\opendunim$ as in Corollary \ref{cor:opendecoms_unique_decomp}. Also, $q\in\C^*$ is given by $q=\sdr_k(t)$ (with $t\in\invdtorus$), and the subindex relabeling $\wPindex$ is defined in Definition \ref{df:PerrinQuiver}. The set $\wPrimeSubsets$ is defined in Definition \ref{df:wPrimeSubsets} and all its elements are obtained using the operations of Lemma \ref{lem:legal_subexpression_moves}.
\end{cor}

\section{Laurent polynomial potentials for all the cominuscule homogeneous spaces}\label{sec:Application}
Theorem \ref{thm:Explicit_LP_LG_model} allows us to calculate Laurent polynomial potentials for the cominuscule homogeneous spaces listed in Table \ref{tab:CominusculeSpaces}, and Corollary \ref{cor:Explicit_LP_LG_model_with_quiver_subsets} gives us a tractable way to find all the terms. In this section, we will give reduced expressions for $\wP$ and $\wPrime$, the quivers $\PerrinQ_{\X}$ for all the cominuscule homogeneous varieties and we will work out the sets $\wPrimeSubsets$ and the resulting Laurent polynomial expressions for representative examples. 

The Laurent polynomials we obtain for quadrics (type $\LGB_n$ and $\LGD_n$) and Lagrangian Grassmannians (type $\LGC_n$) are identical to those found in \cite{Pech_Rietsch_Williams_Quadrics} (Propositions 2.2 and 3.11) and \cite{Pech_Rietsch_Lagrangian_Grassmannians} (Proposition A.1). 
This is to be expected, as Theorem \ref{thm:Explicit_LP_LG_model} is obtained by a generalization of the methods used there.

However, to the best of our knowledge, the Laurent polynomials for orthogonal Grassmannians (type $\LGD_n$) for general $n$, the Cayley plane (type $\LGE_6$) and the Freudenthal variety (type $\LGE_7$) have not yet been given. Moreover, all the potentials have a uniform structure resembling Givental's Laurent polynomial potential for projective complete intersections \cite{Givental_Equivariant_Gromov_Witten_invariants}, namely they are the sum of the toric coordinates plus a quantum term consisting of a homogeneous polynomial divided by the product of all the toric coordinates.

\subsection{The Grassmannian,} $\X=\Gr(k,n)=\SL_n/\P_k$, considered as a homogeneous space for the special linear group, the simply-connected complex Lie group of type $\LGA_{n-1}$. Note that the parabolic subgroup is given by
\[
\P_k = \mtrx{{cc} \GL_k & \Mat_{k\times(n-k)} \\ 0 & \GL_{n-k}}\cap\SL_n.
\]
We make two assumptions on $k$: Firstly, we assume that $k\notin\{1,n-1\}$: for $k=1$ and $k=n-1$ we find $\wPrime=1$. Secondly, we assume without loss of generality that $k>n-k$: for the remaining cases apply the Dynkin diagram bijection $i\mapsto n-i$. The longest Weyl group element has minimal coset representative
\[
\wP = (s_{n-k}s_{n+1-k}\cdots s_{n-1})(s_{n-1-k}s_{n-k}\cdots s_{n-2})\cdots(s_1s_2\cdots s_k),
\]
having $n-k$ products in parentheses each with $k$ factors. On the other hand, we find for $\wPrime$ the reduced expression
\[
\wPrime = (s_{n-k}\cdots s_{n-2})(s_{n-k-1}\cdots s_{n-3})\cdots(s_2\cdots s_k),
\]
having $n-k-1$ products in parentheses each with $k-1$ factors. 

The quiver $\PerrinQ_{\X}$ can be written as a $(n-k-1)\times(k-1)$-rectangle and the reduced subexpression for $\wPrime$ that is minimal in the lexicographical order is the $(n-k-2)\times(k-2)$-rectangle obtained by removing the bottom row and rightmost column.
\begin{ex}
Consider $\Gr(4,7)$, which is homogeneous for $\SL_7$ of type $\LGA_6$. We find that $\wP=(s_3s_4s_5s_6)(s_2s_3s_4s_5)(s_1s_2s_3s_4)$ and $\wPrime=(s_3s_4s_5)(s_2s_3s_4)$. The quiver is of the following form, using the conventions of Example \ref{ex:MiniGrassmannian}:
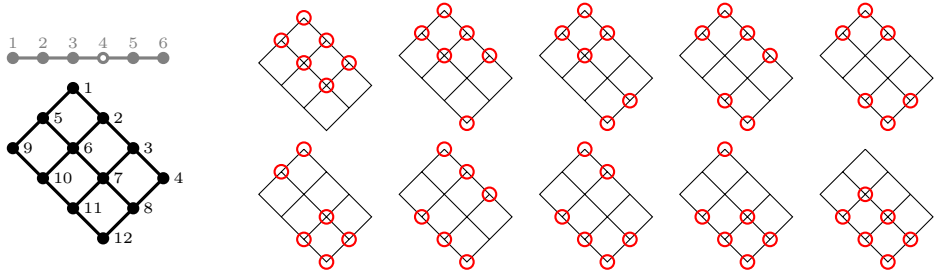
\[
\begin{tikzpicture}[scale=.4,style=very thick,baseline=0em]
	\draw[gray]
	(-2,3) -- (3,3);
	\draw[gray,fill=gray]
	(-2,3) circle (.15)
	(-1,3) circle (.15)
	(0,3) circle (.15)
	(2,3) circle (.15)
	(3,3) circle (.15);
	\draw[gray,fill=white]
	(1,3) circle (.15);
	\node at (-2,3)[above,gray]{\tiny$1$};
	\node at (-1,3)[above,gray]{\tiny$2$};
	\node at (0,3)[above,gray]{\tiny$3$};
	\node at (1,3)[above,gray]{\tiny$4$};
	\node at (2,3)[above,gray]{\tiny$5$};
	\node at (3,3)[above,gray]{\tiny$6$};
	\draw 
	(0,2) -- (3,-1) -- (1,-3) -- (-2,0) -- (0,2)
	(-1,1) -- (2,-2)
	(1,1) -- (-1,-1)
	(2,0) -- (0,-2);
	\draw[black, fill=black] 
	(0,2) circle (.15) 
	(1,1) circle (.15)
	(2,0) circle (.15)
	(3,-1) circle (.15)
	(-1,1) circle (.15)
	(0,0) circle (.15)
	(1,-1) circle (.15)
	(2,-2) circle (.15)
	(-2,0) circle (.15)
	(-1,-1) circle (.15)
	(0,-2) circle (.15)
	(1,-3) circle (.15);
	\node at (0,2)[right]{\tiny$1$};
	\node at (1,1)[right]{\tiny$2$};
	\node at (2,0)[right]{\tiny$3$};
	\node at (3,-1)[right]{\tiny$4$};
	\node at (-1,1)[right]{\tiny$5$};
	\node at (0,0)[right]{\tiny$6$};
	\node at (1,-1)[right]{\tiny$7$};
	\node at (2,-2)[right]{\tiny$8$};
	\node at (-2,0)[right]{\tiny$9$};
	\node at (-1,-1)[right]{\tiny$10$};
	\node at (0,-2)[right]{\tiny$11$};
	\node at (1,-3)[right]{\tiny$12$};
\end{tikzpicture}
\qquad
\begin{tabular}{ccccc}
\begin{tikzpicture}[scale = .3]
	\draw 
	(0,2) -- (3,-1) -- (1,-3) -- (-2,0) -- (0,2)
	(-1,1) -- (2,-2)
	(1,1) -- (-1,-1)
	(2,0) -- (0,-2);
	\draw[red,thick] 
	(0,2) circle (.3)
	(1,1) circle (.3)
	(2,0) circle (.3)
	(-1,1) circle (.3)
	(0,0) circle (.3)
	(1,-1) circle (.3);
\end{tikzpicture}
&
\begin{tikzpicture}[scale = .3]
	\draw 
	(0,2) -- (3,-1) -- (1,-3) -- (-2,0) -- (0,2)
	(-1,1) -- (2,-2)
	(1,1) -- (-1,-1)
	(2,0) -- (0,-2);
	\draw[red,thick] 
	(0,2) circle (.3)
	(1,1) circle (.3)
	(2,0) circle (.3)
	(-1,1) circle (.3)
	(0,0) circle (.3)
	(1,-3) circle (.3);
\end{tikzpicture}
&
\begin{tikzpicture}[scale = .3]
	\draw 
	(0,2) -- (3,-1) -- (1,-3) -- (-2,0) -- (0,2)
	(-1,1) -- (2,-2)
	(1,1) -- (-1,-1)
	(2,0) -- (0,-2);
	\draw[red,thick] 
	(0,2) circle (.3)
	(1,1) circle (.3)
	(2,-2) circle (.3)
	(-1,1) circle (.3)
	(0,0) circle (.3)
	(1,-3) circle (.3);
\end{tikzpicture}
&
\begin{tikzpicture}[scale = .3]
	\draw 
	(0,2) -- (3,-1) -- (1,-3) -- (-2,0) -- (0,2)
	(-1,1) -- (2,-2)
	(1,1) -- (-1,-1)
	(2,0) -- (0,-2);
	\draw[red,thick] 
	(0,2) circle (.3)
	(1,1) circle (.3)
	(2,0) circle (.3)
	(-1,1) circle (.3)
	(0,-2) circle (.3)
	(1,-3) circle (.3);
\end{tikzpicture}
&
\begin{tikzpicture}[scale = .3]
	\draw 
	(0,2) -- (3,-1) -- (1,-3) -- (-2,0) -- (0,2)
	(-1,1) -- (2,-2)
	(1,1) -- (-1,-1)
	(2,0) -- (0,-2);
	\draw[red,thick] 
	(0,2) circle (.3)
	(1,1) circle (.3)
	(2,-2) circle (.3)
	(-1,1) circle (.3)
	(0,-2) circle (.3)
	(1,-3) circle (.3);
\end{tikzpicture}
\\
\begin{tikzpicture}[scale = .3]
	\draw 
	(0,2) -- (3,-1) -- (1,-3) -- (-2,0) -- (0,2)
	(-1,1) -- (2,-2)
	(1,1) -- (-1,-1)
	(2,0) -- (0,-2);
	\draw[red,thick] 
	(0,2) circle (.3)
	(1,-1) circle (.3)
	(2,-2) circle (.3)
	(-1,1) circle (.3)
	(0,-2) circle (.3)
	(1,-3) circle (.3);
\end{tikzpicture}
&
\begin{tikzpicture}[scale = .3]
	\draw 
	(0,2) -- (3,-1) -- (1,-3) -- (-2,0) -- (0,2)
	(-1,1) -- (2,-2)
	(1,1) -- (-1,-1)
	(2,0) -- (0,-2);
	\draw[red,thick] 
	(0,2) circle (.3)
	(1,1) circle (.3)
	(2,0) circle (.3)
	(-1,-1) circle (.3)
	(0,-2) circle (.3)
	(1,-3) circle (.3);
\end{tikzpicture}
&
\begin{tikzpicture}[scale = .3]
	\draw 
	(0,2) -- (3,-1) -- (1,-3) -- (-2,0) -- (0,2)
	(-1,1) -- (2,-2)
	(1,1) -- (-1,-1)
	(2,0) -- (0,-2);
	\draw[red,thick] 
	(0,2) circle (.3)
	(1,1) circle (.3)
	(2,-2) circle (.3)
	(-1,-1) circle (.3)
	(0,-2) circle (.3)
	(1,-3) circle (.3);
\end{tikzpicture}
&
\begin{tikzpicture}[scale = .3]
	\draw 
	(0,2) -- (3,-1) -- (1,-3) -- (-2,0) -- (0,2)
	(-1,1) -- (2,-2)
	(1,1) -- (-1,-1)
	(2,0) -- (0,-2);
	\draw[red,thick] 
	(0,2) circle (.3)
	(1,-1) circle (.3)
	(2,-2) circle (.3)
	(-1,-1) circle (.3)
	(0,-2) circle (.3)
	(1,-3) circle (.3);
\end{tikzpicture}
&
\begin{tikzpicture}[scale = .3]
	\draw 
	(0,2) -- (3,-1) -- (1,-3) -- (-2,0) -- (0,2)
	(-1,1) -- (2,-2)
	(1,1) -- (-1,-1)
	(2,0) -- (0,-2);
	\draw[red,thick] 
	(0,0) circle (.3)
	(1,-1) circle (.3)
	(2,-2) circle (.3)
	(-1,-1) circle (.3)
	(0,-2) circle (.3)
	(1,-3) circle (.3);
\end{tikzpicture}
\end{tabular}
\]
The nine subsets associated to reduced subexpressions of $\wPrime$ in $\wP$ are drawn to the right. So, taking $z=u_+t\bwop u_-$ with $(u_-)^{-1} = \dy_3(-a_1)\dy_4(-a_2)\cdots\dy_4(-a_{10})$ and $q=\sdr_4(t)$, we find
\[
\potZo(z) = \sum_{i=1}^{12}a_i + q\frac{P(a_i)}{a_1a_2a_3a_4a_5a_6a_7a_8a_9a_{10}a_{11}a_{12}},
\]
where
\ali{
\hspace{-1em}
P(a_i) &= a_1a_2a_3a_5a_6a_7 + a_1a_2a_3a_5a_6a_{12}+a_1a_2a_5a_6a_8a_{12}+a_1a_2a_3a_5a_{11}a_{12}+a_1a_2a_5a_8a_{11}a_{12} \\
&{}+a_1a_5a_7a_8a_{11}a_{12} + a_1a_2a_3a_{10}a_{11}a_{12} + a_1a_2a_8a_{10}a_{11}a_{12} + a_1a_7a_8a_{10}a_{11}a_{12}+a_6a_7a_8a_{10}a_{11}a_{12},
}
with the summands written in order of the subsets. 

Note that there is a clear bijection between the subsets of the quiver $\PerrinQ_{\X}$ and Young diagrams that fit inside a $2\times3$-rectangle: the unmarked vertices correspond to the contour of the Young diagram. For example:
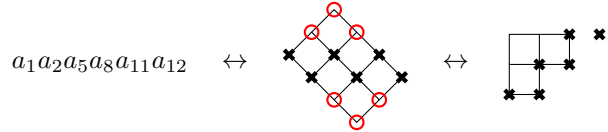
\[
a_1a_2a_5a_8a_{11}a_{12} \quad\leftrightarrow \quad
\begin{tikzpicture}[scale = .3,baseline = -0.5em]
	\draw 
	(0,2) -- (3,-1) -- (1,-3) -- (-2,0) -- (0,2)
	(-1,1) -- (2,-2)
	(1,1) -- (-1,-1)
	(2,0) -- (0,-2);
	\draw[red,thick] 
	(0,2) circle (.3)
	(1,1) circle (.3)
	(2,-2) circle (.3)
	(-1,1) circle (.3)
	(0,-2) circle (.3)
	(1,-3) circle (.3);
	\draw[ultra thick]
	(-2.21,0.21) -- (-1.79,-0.21)   (-2.21,-0.21) -- (-1.79,0.21)
	(-1.21,-0.79) -- (-0.79,-1.21)   (-1.21,-1.21) -- (-0.79,-0.79)
	(-0.21,0.21) -- (0.21,-0.21)   (-0.21,-0.21) -- (0.21,0.21)
	(0.79,-0.79) -- (1.21,-1.21)   (1.21,-0.79) -- (0.79,-1.21)
	(1.79,0.21) -- (2.21,-0.21)   (2.21,0.21) -- (1.79,-0.21)
	(2.79,-.79) -- (3.21,-1.21)   (2.79,-1.21) -- (3.21,-0.79);
\end{tikzpicture}
\quad\leftrightarrow\quad
\begin{tikzpicture}[scale = .4,baseline=1em]
	\draw 
	(0,0) -- (1,0) -- (1,1) -- (2,1) -- (2,2) -- (0,2) -- (0,0)
	(0,1) -- (1,1) -- (1,2);
	\draw[ultra thick]
	(-0.15,0.15) -- (0.15,-0.15)   (-0.15,-0.15) -- (0.15,0.15)
	(0.85,0.15) -- (1.15,-0.15)   (0.85,-0.15) -- (1.15,0.15)
	(0.85,1.15) -- (1.15,0.85)   (0.85,0.85) -- (1.15,1.15)
	(1.85,1.15) -- (2.15,0.85)   (1.85,0.85) -- (2.15,1.15)
	(1.85,2.15) -- (2.15,1.85)   (1.85,1.85) -- (2.15,2.15)
	(2.85,2.15) -- (3.15,1.85)   (2.85,1.85) -- (3.15,2.15);
\end{tikzpicture}
\]
This bijection works in general, so we obtain an alternative description of the quantum term for a general Grassmannian $\Gr(k,n)$ that sums over the Young diagrams that fit inside a $(n-k-1)\times(k-1)$-rectangle.
\end{ex}
\begin{rem}
Other Laurent polynomial Landau-Ginzburg models have already been given for Grassmannians. Particularly relevant is the potential $\potEHX:(\C^*)^{k(n-k)}\times\C^*\to\C$ given in \cite{Eguchi_Hori_Xiong_Gravitation_quantum_cohomology}, equation (B.25) (see also \cite{Batyrev_CF_K_VS_Conifold_transitions_and_mirror_symmetry_for_Calabi_Yau_complete_intersections_in_Grassmannians}, Conjecture 4.2.2, as well as \cite{Rietsch_Marsh_Grassmannians}, section 6.3), which is described as follows. 

For $i\in\{1,\ldots,n-k\}$ and $j\in\{1,\ldots,k\}$, write $[i,j]$ for the Young diagram corresponding to the partition $(j,1,1,\ldots,1)$ of length $i$; for example, $[2,3]$ corresponds to the diagram \tikz[scale=0.15,baseline=1pt]{\draw (0,0) -- (0,2) -- (3,2) -- (3,1) --(1,1) -- (1,0) -- (0,0) (0,1)--(1,1)--(1,2) (2,2)--(2,1);}. Denote the set of these diagrams by $\Lambda_s$. Now, denote the coordinates of $(\C^*)^{k(n-k)}$ by $z_{[i,j]}$ for $[i,j]\in\Lambda_s$, denote the coordinate of the remaining factor $\C^*$ by $z_{\infty}=q$ and finally set $z_\varnothing=1$. Consider the quiver with as vertices $\Lambda_s^*=\Lambda_s\cup\{\varnothing,\infty\}$ and as arrows 
\[
\{[i,j]\to[i,j+1],~[i,j]\to[i+1,j]\}~\cup~\{\varnothing\to[1,1]\}~\cup~\{[n-k,k]\to\infty\}.
\]
For example, for $\G(4,7)$ we obtain:
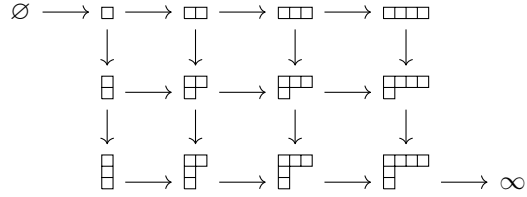
\[
\begin{tikzcd}[sep = scriptsize]
\varnothing 
\rar &{\tikz[scale=0.15]{\draw (0,0) -- (0,1) -- (1,1) -- (1,0) -- (0,0)}} \dar 
\rar & {\tikz[scale=0.15]{\draw (0,0) -- (2,0) -- (2,1) -- (0,1) -- (0,0) (1,0) -- (1,1)}} \dar
\rar & {\tikz[scale=0.15]{\draw (0,0) -- (3,0) -- (3,1) -- (0,1) -- (0,0) (1,0) -- (1,1) (2,0) -- (2,1)}} \dar
\rar & {\tikz[scale=0.15]{\draw (0,0) -- (4,0) -- (4,1) -- (0,1) -- (0,0) (1,0) -- (1,1) (2,0) -- (2,1) (3,0) -- (3,1)}} \dar \\
	& {\tikz[scale=0.15]{\draw (0,0) -- (0,2) -- (1,2) -- (1,0) -- (0,0) (0,1) -- (1,1)}} \dar 
\rar & {\tikz[scale=0.15]{\draw (0,0) -- (0,2) -- (2,2) -- (2,1) -- (1,1) -- (1,0) -- (0,0) (0,1) -- (1,1) -- (1,2)}} \dar 
\rar & {\tikz[scale=0.15]{\draw (0,0) -- (0,2) -- (3,2) -- (3,1) -- (1,1) -- (1,0) -- (0,0) (0,1) -- (1,1) -- (1,2) (2,1) -- (2,2)}} \dar
\rar & {\tikz[scale=0.15]{\draw (0,0) -- (0,2) -- (4,2) -- (4,1) -- (1,1) -- (1,0) -- (0,0) (0,1) -- (1,1) -- (1,2) (2,1) -- (2,2) (3,1) -- (3,2)}} \dar \\
	& {\tikz[scale=0.15]{\draw (0,0) -- (0,3) -- (1,3) -- (1,0) -- (0,0) (0,1) -- (1,1) (0,2) -- (1,2)}} 
\rar & {\tikz[scale=0.15]{\draw (0,0) -- (0,3) -- (2,3) -- (2,2) -- (1,2) -- (1,0) -- (0,0) (0,2) -- (1,2) -- (1,3) (0,1)--(1,1)}} 
\rar & {\tikz[scale=0.15]{\draw (0,0) -- (0,3) -- (3,3) -- (3,2) -- (1,2) -- (1,0) -- (0,0) (0,2) -- (1,2) -- (1,3) (0,1)--(1,1) (2,2) -- (2,3)}} 
\rar & {\tikz[scale=0.15]{\draw (0,0) -- (0,3) -- (4,3) -- (4,2) -- (1,2) -- (1,0) -- (0,0) (0,2) -- (1,2) -- (1,3) (0,1)--(1,1) (2,2) -- (2,3) (3,2) -- (3,3)}} 
\rar & \infty
\end{tikzcd}
\]
Denoting $[1,0]=\varnothing$ and $[n-k,k+1]=\infty$ (so that $z_{[1,0]} = z_\varnothing=1$ and $z_{[n-k,k+1]}=z_\infty=q$), consider for $[i,j]\in\Lambda_s^*$ the following quotients:
\begin{equation}
T_{[i,j]} = \frac{z_{[i+1,j]}+z_{[i,j+1]}}{z_{[i,j]}},
\label{eq:EHX_terms_of_potential}
\end{equation}
where $z_{[i,j]}=0$ for $[i,j]\notin\Lambda_s^*$. In~other words, for each vertex $[i,j]\in\Lambda_s^*$, $T_{[i,j]}$ consists of the sum of the coordinates at the ends of the outgoing arrows divided by the coordinate of the vertex. For example, 
\ali{
&T_{\tikz[scale=0.1]{\draw (0,0) -- (0,1) -- (1,1) -- (1,0) -- (0,0)}} 
= \frac{z_{\tikz[scale=0.1]{\draw (0,0) -- (2,0) -- (2,1) -- (0,1) -- (0,0) (1,0) -- (1,1)}} 
	+ z_ {\tikz[scale=0.1]{\draw (0,0) -- (0,2) -- (1,2) -- (1,0) -- (0,0) (0,1) -- (1,1)}}}
	{z_{\tikz[scale=0.1]{\draw (0,0) -- (0,1) -- (1,1) -- (1,0) -- (0,0)}}}, \quad
T_{[n-k,1]} 
= \frac{z_{[n-k,2]}}{z_{[n-k,1]}},\quad 
T_{[1,k]}=\frac{z_{[2,k]}}{z_{[1,k]}}, \quad 
T_{\varnothing}=T_{[1,0]}=z_{\tikz[scale=0.1]{\draw (0,0) -- (0,1) -- (1,1) -- (1,0) -- (0,0)}}, \\
&T_{[n-k,k]}=\frac q{z_{[n-k,k]}} \qand 
T_{\infty}=T_{[n-k,k+1]} = 0.
}

The Laurent polynomial potential of \cite{Eguchi_Hori_Xiong_Gravitation_quantum_cohomology} is the following:
\begin{equation}
\potEHX(z_{[i,j]}) = \sum_{[i,j]\in\Lambda_s^*} T_{[i,j]},
\label{eq:EHX_potential}
\end{equation}
where the quantum term is $T_{[n-k,k]}=\frac q{z_{[n-k,k]}}$. In the example of $\Gr(4,7)$, this becomes:
\[
\potEHX(z_{[i,j]}) 
= z_{\tikz[scale=0.1]{\draw (0,0) -- (0,1) -- (1,1) -- (1,0) -- (0,0)}} 
+ \frac{z_{\tikz[scale=0.1]{\draw (0,0) -- (2,0) -- (2,1) -- (0,1) -- (0,0) (1,0) -- (1,1)}} 
		+ z_ {\tikz[scale=0.1]{\draw (0,0) -- (0,2) -- (1,2) -- (1,0) -- (0,0) (0,1) -- (1,1)}}}
		{z_{\tikz[scale=0.1]{\draw (0,0) -- (0,1) -- (1,1) -- (1,0) -- (0,0)}}} 
+ \ldots 
+ \frac{z_{\tikz[scale=0.1]{\draw (0,0) -- (0,3) -- (4,3) -- (4,2) -- (1,2) -- (1,0) -- (0,0) (0,2) -- (1,2) -- (1,3) (0,1)--(1,1) (2,2) -- (2,3) (3,2) -- (3,3)}}}
		{z_{\tikz[scale=0.1]{\draw (0,0) -- (0,3) -- (3,3) -- (3,2) -- (1,2) -- (1,0) -- (0,0) (0,2) -- (1,2) -- (1,3) (0,1)--(1,1) (2,2) -- (2,3)}}}
+ \frac{z_{\tikz[scale=0.1]{\draw (0,0) -- (0,3) -- (4,3) -- (4,2) -- (1,2) -- (1,0) -- (0,0) (0,2) -- (1,2) -- (1,3) (0,1)--(1,1) (2,2) -- (2,3) (3,2) -- (3,3)}}}
		{z_{\tikz[scale=0.1]{\draw (0,0) -- (0,2) -- (4,2) -- (4,1) -- (1,1) -- (1,0) -- (0,0) (0,1) -- (1,1) -- (1,2) (2,1) -- (2,2) (3,1) -- (3,2)}}}
+ \frac q{z_{\tikz[scale=0.1]{\draw (0,0) -- (0,3) -- (4,3) -- (4,2) -- (1,2) -- (1,0) -- (0,0) (0,2) -- (1,2) -- (1,3) (0,1)--(1,1) (2,2) -- (2,3) (3,2) -- (3,3)}}},
\]
which has a total of thirteen terms (twelve plus a quantum term).

The potential $\potEHX$ of \eqref{eq:EHX_potential} is shown in \cite{Rietsch_Marsh_Grassmannians}, Theorem 4.6, to be a local Laurent polynomial expression for the Landau-Ginzburg model used there. That model is also shown in Proposition 6.7 of \cite{Rietsch_Marsh_Grassmannians} to be isomorphic to Rietsch's Lie-theoretic Landau-Ginzburg model given in \cite{Rietsch_Mirror_Construction}, see also Theorem \ref{thm:Lie-theoretic_LG_model} here. By construction, the Laurent polynomial potential $\potZo$ of Theorem \ref{thm:Explicit_LP_LG_model} and Corollary \ref{cor:Explicit_LP_LG_model_with_quiver_subsets} here is a local expression for Rietsch's Lie-theoretic Landau-Ginzburg model. Thus, both $\potZo$ and $\potEHX$ are local Laurent polynomial expressions for the same model. However, it is clear that $\potZo$ and $\potEHX$ are not isomorphic: the quantum term of $\potEHX$ is a Laurent monomial, namely $T_{[n-k,k]}=\frac q{z_{[n-k,k]}}$, whereas the quantum term of $\potZo$ is not.

On the other hand, it is straightforward to find a birational map $\Phi$ such that $\Phi^*\potEHX=\potZo$. 
First, consider the following bijection: for $[i,j]\in\Lambda_s\setminus\{[n-k,k]\}$, let $\phi([i,j])=i\cdot k-j+1$, and let $\phi(\varnothing)=(n-k-1)k+1$. Clearly, $\phi$ is a bijection $\Lambda_s^*\setminus\{[n-k,k],\infty\}\to\{1,\ldots,k(n-k)\}$. Define $\Phi:(\C^*)^{k(n-k)}\to(\C^*)^{k(n-k)}$ to be the birational map such that $\Phi^*(a_i) = T_{\phi^{-1}(i)}$. By definition, we find (recalling $T_\infty=0$)
\ali{
\Phi^*\potZo &= \Phi^*\left(\sum_{i=1}^{k(n-k)}a_i  + q\frac{P(a_j)}{a_1\cdots a_{k(n-k)}}\right) \\
&= \sum_{[i,j]\in\Lambda_s^*\setminus\{[n-k,k]\}} T_{[i,j]} + q\frac{P(T_{[i',j']})}{T_{\varnothing}T_{[1,1]}T_{[1,2]}\cdots T_{[n-k,k-1]}},
}
where $P$ is the homogeneous polynomial in the numerator of the quantum term of Corollary \ref{cor:Explicit_LP_LG_model_with_quiver_subsets}. Thus, what remains to be shown is that the quantum term simplifies to $T_{[n-k,k]}=\frac q{z_{[n-k,k]}}$, which is a straightforward computation for any given $k$ and $n$. 

For example, in the case $\Gr(4,7)$ we can simplify the following products: 
\[
T_\varnothing 
T_{\tikz[scale=0.1]{\draw (0,0) -- (0,1) -- (1,1) -- (1,0) -- (0,0)}} 
=
z_{\tikz[scale=0.1]{\draw (0,0) -- (0,1) -- (1,1) -- (1,0) -- (0,0)}}
\frac{z_{\tikz[scale=0.1]{\draw (0,0) -- (2,0) -- (2,1) -- (0,1) -- (0,0) (1,0) -- (1,1)}} 
+ z_ {\tikz[scale=0.1]{\draw (0,0) -- (0,2) -- (1,2) -- (1,0) -- (0,0) (0,1) -- (1,1)}}}
{z_{\tikz[scale=0.1]{\draw (0,0) -- (0,1) -- (1,1) -- (1,0) -- (0,0)}}} 
=
z_{\tikz[scale=0.1]{\draw (0,0) -- (2,0) -- (2,1) -- (0,1) -- (0,0) (1,0) -- (1,1)}} 
+ z_ {\tikz[scale=0.1]{\draw (0,0) -- (0,2) -- (1,2) -- (1,0) -- (0,0) (0,1) -- (1,1)}},
\quad
T_{\tikz[scale=0.1]{\draw (0,0) -- (4,0) -- (4,1) -- (0,1) -- (0,0) (1,0) -- (1,1) (2,0) -- (2,1) (3,0) -- (3,1)}}
T_{\tikz[scale=0.1]{\draw (0,0) -- (0,2) -- (4,2) -- (4,1) -- (1,1) -- (1,0) -- (0,0) (0,1) -- (1,1) -- (1,2) (2,1) -- (2,2) (3,1) -- (3,2)}} 
= 
\frac{z_{\tikz[scale=0.1]{\draw (0,0) -- (0,2) -- (4,2) -- (4,1) -- (1,1) -- (1,0) -- (0,0) (0,1) -- (1,1) -- (1,2) (2,1) -- (2,2) (3,1) -- (3,2)}}}
{z_{\tikz[scale=0.1]{\draw (0,0) -- (4,0) -- (4,1) -- (0,1) -- (0,0) (1,0) -- (1,1) (2,0) -- (2,1) (3,0) -- (3,1)}}}
\frac{z_{\tikz[scale=0.1]{\draw (0,0) -- (0,3) -- (4,3) -- (4,2) -- (1,2) -- (1,0) -- (0,0) (0,2) -- (1,2) -- (1,3) (0,1)--(1,1) (2,2) -- (2,3) (3,2) -- (3,3)}}}
{z_{\tikz[scale=0.1]{\draw (0,0) -- (0,2) -- (4,2) -- (4,1) -- (1,1) -- (1,0) -- (0,0) (0,1) -- (1,1) -- (1,2) (2,1) -- (2,2) (3,1) -- (3,2)}}}
=
\frac{z_{\tikz[scale=0.1]{\draw (0,0) -- (0,3) -- (4,3) -- (4,2) -- (1,2) -- (1,0) -- (0,0) (0,2) -- (1,2) -- (1,3) (0,1)--(1,1) (2,2) -- (2,3) (3,2) -- (3,3)}}}
{z_{\tikz[scale=0.1]{\draw (0,0) -- (4,0) -- (4,1) -- (0,1) -- (0,0) (1,0) -- (1,1) (2,0) -- (2,1) (3,0) -- (3,1)}}},
\quad 
T_{\tikz[scale=0.1]{\draw (0,0) -- (0,3) -- (1,3) -- (1,0) -- (0,0) (0,1) -- (1,1) (0,2) -- (1,2)}}
T_{\tikz[scale=0.1]{\draw (0,0) -- (0,3) -- (2,3) -- (2,2) -- (1,2) -- (1,0) -- (0,0) (0,2) -- (1,2) -- (1,3) (0,1)--(1,1)}}
T_{\tikz[scale=0.1]{\draw (0,0) -- (0,3) -- (3,3) -- (3,2) -- (1,2) -- (1,0) -- (0,0) (0,2) -- (1,2) -- (1,3) (0,1)--(1,1) (2,2) -- (2,3)}} 
=
\frac{z_{\tikz[scale=0.1]{\draw (0,0) -- (0,3) -- (4,3) -- (4,2) -- (1,2) -- (1,0) -- (0,0) (0,2) -- (1,2) -- (1,3) (0,1)--(1,1) (2,2) -- (2,3) (3,2) -- (3,3)}}}
{z_{\tikz[scale=0.1]{\draw (0,0) -- (0,3) -- (1,3) -- (1,0) -- (0,0) (0,1) -- (1,1) (0,2) -- (1,2)}}}.
\]
Thus, the denominator becomes
\ali{
(
T_\varnothing 
T_{\tikz[scale=0.1]{\draw (0,0) -- (0,1) -- (1,1) -- (1,0) -- (0,0)}} 
)&
(
T_{\tikz[scale=0.1]{\draw (0,0) -- (2,0) -- (2,1) -- (0,1) -- (0,0) (1,0) -- (1,1)}}
T_{\tikz[scale=0.1]{\draw (0,0) -- (3,0) -- (3,1) -- (0,1) -- (0,0) (1,0) -- (1,1) (2,0) -- (2,1)}}
T_{\tikz[scale=0.1]{\draw (0,0) -- (0,2) -- (1,2) -- (1,0) -- (0,0) (0,1) -- (1,1)}}
T_{\tikz[scale=0.1]{\draw (0,0) -- (0,2) -- (2,2) -- (2,1) -- (1,1) -- (1,0) -- (0,0) (0,1) -- (1,1) -- (1,2)}}
T_{\tikz[scale=0.1]{\draw (0,0) -- (0,2) -- (3,2) -- (3,1) -- (1,1) -- (1,0) -- (0,0) (0,1) -- (1,1) -- (1,2) (2,1) -- (2,2)}}
)
(T_{\tikz[scale=0.1]{\draw (0,0) -- (4,0) -- (4,1) -- (0,1) -- (0,0) (1,0) -- (1,1) (2,0) -- (2,1) (3,0) -- (3,1)}}
T_{\tikz[scale=0.1]{\draw (0,0) -- (0,2) -- (4,2) -- (4,1) -- (1,1) -- (1,0) -- (0,0) (0,1) -- (1,1) -- (1,2) (2,1) -- (2,2) (3,1) -- (3,2)}} 
)
(
T_{\tikz[scale=0.1]{\draw (0,0) -- (0,3) -- (1,3) -- (1,0) -- (0,0) (0,1) -- (1,1) (0,2) -- (1,2)}}
T_{\tikz[scale=0.1]{\draw (0,0) -- (0,3) -- (2,3) -- (2,2) -- (1,2) -- (1,0) -- (0,0) (0,2) -- (1,2) -- (1,3) (0,1)--(1,1)}}
T_{\tikz[scale=0.1]{\draw (0,0) -- (0,3) -- (3,3) -- (3,2) -- (1,2) -- (1,0) -- (0,0) (0,2) -- (1,2) -- (1,3) (0,1)--(1,1) (2,2) -- (2,3)}} 
)\\
={}&
(
z_{\tikz[scale=0.1]{\draw (0,0) -- (2,0) -- (2,1) -- (0,1) -- (0,0) (1,0) -- (1,1)}} 
+ z_ {\tikz[scale=0.1]{\draw (0,0) -- (0,2) -- (1,2) -- (1,0) -- (0,0) (0,1) -- (1,1)}}
)
\left(
\frac{z_{\tikz[scale=0.1]{\draw (0,0) -- (3,0) -- (3,1) -- (0,1) -- (0,0) (1,0) -- (1,1) (2,0) -- (2,1)}}
+z_{\tikz[scale=0.1]{\draw (0,0) -- (0,2) -- (2,2) -- (2,1) -- (1,1) -- (1,0) -- (0,0) (0,1) -- (1,1) -- (1,2)}}}
{z_{\tikz[scale=0.1]{\draw (0,0) -- (2,0) -- (2,1) -- (0,1) -- (0,0) (1,0) -- (1,1)}}}
\frac{z_{\tikz[scale=0.1]{\draw (0,0) -- (4,0) -- (4,1) -- (0,1) -- (0,0) (1,0) -- (1,1) (2,0) -- (2,1) (3,0) -- (3,1)}}
+z_{\tikz[scale=0.1]{\draw (0,0) -- (0,2) -- (3,2) -- (3,1) -- (1,1) -- (1,0) -- (0,0) (0,1) -- (1,1) -- (1,2) (2,1) -- (2,2)}}}
{z_{\tikz[scale=0.1]{\draw (0,0) -- (3,0) -- (3,1) -- (0,1) -- (0,0) (1,0) -- (1,1) (2,0) -- (2,1)}}}
\frac{z_{\tikz[scale=0.1]{\draw (0,0) -- (0,2) -- (2,2) -- (2,1) -- (1,1) -- (1,0) -- (0,0) (0,1) -- (1,1) -- (1,2)}}
+z_{\tikz[scale=0.1]{\draw (0,0) -- (0,3) -- (1,3) -- (1,0) -- (0,0) (0,1) -- (1,1) (0,2) -- (1,2)}}}
{z_{\tikz[scale=0.1]{\draw (0,0) -- (0,2) -- (1,2) -- (1,0) -- (0,0) (0,1) -- (1,1)}}}
\frac{z_{\tikz[scale=0.1]{\draw (0,0) -- (0,2) -- (3,2) -- (3,1) -- (1,1) -- (1,0) -- (0,0) (0,1) -- (1,1) -- (1,2) (2,1) -- (2,2)}}
+z_{\tikz[scale=0.1]{\draw (0,0) -- (0,3) -- (2,3) -- (2,2) -- (1,2) -- (1,0) -- (0,0) (0,2) -- (1,2) -- (1,3) (0,1)--(1,1)}}}
{z_{\tikz[scale=0.1]{\draw (0,0) -- (0,2) -- (2,2) -- (2,1) -- (1,1) -- (1,0) -- (0,0) (0,1) -- (1,1) -- (1,2)}}}
\frac{z_{\tikz[scale=0.1]{\draw (0,0) -- (0,2) -- (4,2) -- (4,1) -- (1,1) -- (1,0) -- (0,0) (0,1) -- (1,1) -- (1,2) (2,1) -- (2,2) (3,1) -- (3,2)}}
+z_{\tikz[scale=0.1]{\draw (0,0) -- (0,3) -- (3,3) -- (3,2) -- (1,2) -- (1,0) -- (0,0) (0,2) -- (1,2) -- (1,3) (0,1)--(1,1) (2,2) -- (2,3)}} }
{z_{\tikz[scale=0.1]{\draw (0,0) -- (0,2) -- (3,2) -- (3,1) -- (1,1) -- (1,0) -- (0,0) (0,1) -- (1,1) -- (1,2) (2,1) -- (2,2)}}}
\right)
\left(
\frac{z_{\tikz[scale=0.1]{\draw (0,0) -- (0,3) -- (4,3) -- (4,2) -- (1,2) -- (1,0) -- (0,0) (0,2) -- (1,2) -- (1,3) (0,1)--(1,1) (2,2) -- (2,3) (3,2) -- (3,3)}}}
{z_{\tikz[scale=0.1]{\draw (0,0) -- (4,0) -- (4,1) -- (0,1) -- (0,0) (1,0) -- (1,1) (2,0) -- (2,1) (3,0) -- (3,1)}}}
\right)
\left(
\frac{z_{\tikz[scale=0.1]{\draw (0,0) -- (0,3) -- (4,3) -- (4,2) -- (1,2) -- (1,0) -- (0,0) (0,2) -- (1,2) -- (1,3) (0,1)--(1,1) (2,2) -- (2,3) (3,2) -- (3,3)}}}
{z_{\tikz[scale=0.1]{\draw (0,0) -- (0,3) -- (1,3) -- (1,0) -- (0,0) (0,1) -- (1,1) (0,2) -- (1,2)}}}
\right).
}
It is more work to simplify $P(T_{[i',j']})$, but in the end we find
\[
P(T_{[i',j']}) = 
(
z_{\tikz[scale=0.1]{\draw (0,0) -- (2,0) -- (2,1) -- (0,1) -- (0,0) (1,0) -- (1,1)}} 
+ z_ {\tikz[scale=0.1]{\draw (0,0) -- (0,2) -- (1,2) -- (1,0) -- (0,0) (0,1) -- (1,1)}}
)
\left(
\frac{z_{\tikz[scale=0.1]{\draw (0,0) -- (3,0) -- (3,1) -- (0,1) -- (0,0) (1,0) -- (1,1) (2,0) -- (2,1)}}
+z_{\tikz[scale=0.1]{\draw (0,0) -- (0,2) -- (2,2) -- (2,1) -- (1,1) -- (1,0) -- (0,0) (0,1) -- (1,1) -- (1,2)}}}
{z_{\tikz[scale=0.1]{\draw (0,0) -- (2,0) -- (2,1) -- (0,1) -- (0,0) (1,0) -- (1,1)}}}
\frac{z_{\tikz[scale=0.1]{\draw (0,0) -- (4,0) -- (4,1) -- (0,1) -- (0,0) (1,0) -- (1,1) (2,0) -- (2,1) (3,0) -- (3,1)}}
+z_{\tikz[scale=0.1]{\draw (0,0) -- (0,2) -- (3,2) -- (3,1) -- (1,1) -- (1,0) -- (0,0) (0,1) -- (1,1) -- (1,2) (2,1) -- (2,2)}}}
{z_{\tikz[scale=0.1]{\draw (0,0) -- (3,0) -- (3,1) -- (0,1) -- (0,0) (1,0) -- (1,1) (2,0) -- (2,1)}}}
\frac{z_{\tikz[scale=0.1]{\draw (0,0) -- (0,2) -- (2,2) -- (2,1) -- (1,1) -- (1,0) -- (0,0) (0,1) -- (1,1) -- (1,2)}}
+z_{\tikz[scale=0.1]{\draw (0,0) -- (0,3) -- (1,3) -- (1,0) -- (0,0) (0,1) -- (1,1) (0,2) -- (1,2)}}}
{z_{\tikz[scale=0.1]{\draw (0,0) -- (0,2) -- (1,2) -- (1,0) -- (0,0) (0,1) -- (1,1)}}}
\frac{z_{\tikz[scale=0.1]{\draw (0,0) -- (0,2) -- (3,2) -- (3,1) -- (1,1) -- (1,0) -- (0,0) (0,1) -- (1,1) -- (1,2) (2,1) -- (2,2)}}
+z_{\tikz[scale=0.1]{\draw (0,0) -- (0,3) -- (2,3) -- (2,2) -- (1,2) -- (1,0) -- (0,0) (0,2) -- (1,2) -- (1,3) (0,1)--(1,1)}}}
{z_{\tikz[scale=0.1]{\draw (0,0) -- (0,2) -- (2,2) -- (2,1) -- (1,1) -- (1,0) -- (0,0) (0,1) -- (1,1) -- (1,2)}}}
\frac{z_{\tikz[scale=0.1]{\draw (0,0) -- (0,2) -- (4,2) -- (4,1) -- (1,1) -- (1,0) -- (0,0) (0,1) -- (1,1) -- (1,2) (2,1) -- (2,2) (3,1) -- (3,2)}}
+z_{\tikz[scale=0.1]{\draw (0,0) -- (0,3) -- (3,3) -- (3,2) -- (1,2) -- (1,0) -- (0,0) (0,2) -- (1,2) -- (1,3) (0,1)--(1,1) (2,2) -- (2,3)}} }
{z_{\tikz[scale=0.1]{\draw (0,0) -- (0,2) -- (3,2) -- (3,1) -- (1,1) -- (1,0) -- (0,0) (0,1) -- (1,1) -- (1,2) (2,1) -- (2,2)}}}
\right)
\left(
\frac{z_{\tikz[scale=0.1]{\draw (0,0) -- (0,3) -- (4,3) -- (4,2) -- (1,2) -- (1,0) -- (0,0) (0,2) -- (1,2) -- (1,3) (0,1)--(1,1) (2,2) -- (2,3) (3,2) -- (3,3)}}}
{z_{\tikz[scale=0.1]{\draw (0,0) -- (4,0) -- (4,1) -- (0,1) -- (0,0) (1,0) -- (1,1) (2,0) -- (2,1) (3,0) -- (3,1)}}}
\right)
\left(
\frac1
{z_{\tikz[scale=0.1]{\draw (0,0) -- (0,3) -- (1,3) -- (1,0) -- (0,0) (0,1) -- (1,1) (0,2) -- (1,2)}}}
\right).
\]
So the quotient is indeed $T_{[n-k,k]}=\frac{q}{z_{[n-k,k]}}$.

\end{rem}

\subsection{The quadric,} $\X=Q_d$. Note that both odd- and even-dimensional quadrics are homogeneous for $\Spin_{d+2}$, and that the parabolic subgroup is associated to the first vertex of the Dynkin diagram. Note, however, that $\Spin_{d+2}$ is of a different type depending on whether $d+2=2n+1$ is odd (type $\LGB_{n}$) or $d+2=2n$ is even (type $\LGD_{n}$), but the resulting homogeneous spaces are nonetheless similar enough to be considered at the same time. We find for $\wP$ the reduced expressions:
\[
\wP = \left\{\begin{array}{ll}
s_1s_2\cdots s_{n-1}(s_n)s_{n-1}s_{n-2}\cdots s_1, & \text{for $d=2n-1$,}\\
s_1s_2\cdots s_{n-2}(s_{n-1}s_n)s_{n-2}s_{n-3}\cdots s_1, & \text{for $d=2n-2$.}\\
\end{array}\right.
\]
For both odd and even quadrics we find $\wPrime=s_1$ and in both cases the simple reflection $s_1$ only appears as the first and the last factor, so it is easy enough to find the Laurent polynomial potential without using Corollary \ref{cor:Explicit_LP_LG_model_with_quiver_subsets}. We find the same potential in both cases, namely:
\[
\potZo(z) = \sum_{i=1}^d a_i + q\frac{a_1+a_d}{\prod_{i=1}^d a_i},
\]
where $d$ is the dimension of the quadric, and we decomposed $z=u_+t\bwop u_-$ with $u_-=\dy_1(-a_1)\dy_2(-a_2)\cdots\dy_1(-a_d)$ and $q=\sdr_1(t)$. Note that this Laurent polynomial expression is indeed identical to the ones obtained in \cite{Pech_Rietsch_Williams_Quadrics}, Propositions 2.2 and 3.11. For completeness' sake, let us consider two examples of quadrics and draw the associated quivers:
\begin{ex}
Consider the odd quadric $Q_7$ of type $\LGB_4$ and the even quadric $Q_8$ of type $\LGD_5$. These have for $\wP$ the reduced expressions $s_1s_2s_3(s_4)s_3s_2s_1$ and $s_1s_2s_3(s_4s_5)s_3s_2s_1$ respectively. Thus, we find the quivers
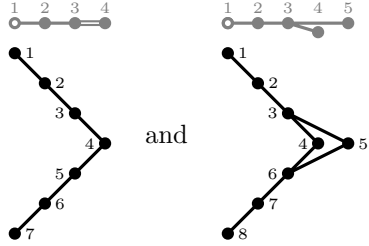
\[
\begin{tikzpicture}[scale=.4,style=very thick,baseline=0em]
	\draw[gray] (0,4) -- (2,4); \draw[gray, thick, double] (2,4) -- (3,4);
	\draw[gray,fill=gray]
	(1,4) circle (.15)
	(2,4) circle (.15)
	(3,4) circle (.15);
	\draw[gray,fill=white]
	(0,4) circle (.15);
	\node at (0,4)[above,gray]{\tiny$1$};
	\node at (1,4)[above,gray]{\tiny$2$};
	\node at (2,4)[above,gray]{\tiny$3$};
	\node at (3,4)[above,gray]{\tiny$4$};
	\draw 
	(0,3) -- (3,0) -- (0,-3);
	\draw[black, fill=black] 
	(0,3) circle (.15) 
	(1,2) circle (.15)
	(2,1) circle (.15)
	(3,0) circle (.15)
	(2,-1) circle (.15)
	(1,-2) circle (.15)
	(0,-3) circle (.15);
	\node at (0,3)[right]{\tiny$1$};
	\node at (1,2)[right]{\tiny$2$};
	\node at (2,1)[left]{\tiny$3$};
	\node at (3,0)[left]{\tiny$4$};
	\node at (2,-1)[left]{\tiny$5$};
	\node at (1,-2)[right]{\tiny$6$};
	\node at (0,-3)[right]{\tiny$7$};
\end{tikzpicture}
\qand
\begin{tikzpicture}[scale=.4,style=very thick,baseline=0em]
	\draw[gray] (0,4) -- (4,4) (2,4) -- (3,3.7);
	\draw[gray,fill=gray]
	(0,4) circle (.15)
	(1,4) circle (.15)
	(2,4) circle (.15)
	(3,3.7) circle (.15)
	(4,4) circle (.15);
	\draw[gray,fill=white]
	(0,4) circle (.15);
	\node at (0,4)[above,gray]{\tiny$1$};
	\node at (1,4)[above,gray]{\tiny$2$};
	\node at (2,4)[above,gray]{\tiny$3$};
	\node at (3,4)[above,gray]{\tiny$4$};
	\node at (4,4)[above,gray]{\tiny$5$};
	\draw 
	(0,3) -- (3,0) -- (0,-3)
	(2,1) -- (4,0) -- (2,-1);
	\draw[black, fill=black] 
	(0,3) circle (.15) 
	(1,2) circle (.15)
	(2,1) circle (.15)
	(3,0) circle (.15)
	(4,0) circle (.15)
	(2,-1) circle (.15)
	(1,-2) circle (.15)
	(0,-3) circle (.15);
	\node at (0,3)[right]{\tiny$1$};
	\node at (1,2)[right]{\tiny$2$};
	\node at (2,1)[left]{\tiny$3$};
	\node at (3,0)[left]{\tiny$4$};
	\node at (4,0)[right]{\tiny$5$};
	\node at (2,-1)[left]{\tiny$6$};
	\node at (1,-2)[right]{\tiny$7$};
	\node at (0,-3)[right]{\tiny$8$};
\end{tikzpicture}
\]
Also in the quivers it is clear that there are only two subsets associated to $\wPrime=s_1$ each containing one vertex: the top and the bottom one.
\end{ex}
\begin{rem}
Note that we have chosen to draw the Coxeter diagram of type $\LGD_5$ as \begin{tikzpicture}[scale=.3,style=very thick]
	\draw[gray] (0,0) -- (4,0) (2,0) -- (3,-.3);
	\draw[gray,fill=gray]
	(0,0) circle (.15)
	(1,0) circle (.15)
	(2,0) circle (.15)
	(3,-.3) circle (.15)
	(4,0) circle (.15);
	\draw[gray,fill=gray]
	(0,0) circle (.15);\end{tikzpicture} instead of as \begin{tikzpicture}[scale=.3,style=very thick]
	\draw[gray] (0,0) -- (2,0) (2,0) -- (3,-.3) (2,0) -- (3,.3);
	\draw[gray,fill=gray]
	(0,0) circle (.15)
	(1,0) circle (.15)
	(2,0) circle (.15)
	(3,-.3) circle (.15)
	(3,.3) circle (.15);
	\draw[gray,fill=gray]
	(0,0) circle (.15);\end{tikzpicture} or \begin{tikzpicture}[scale=.3,style=very thick]
	\draw[gray] (0,0) -- (3,0) (2,0) -- (2,-.6);
	\draw[gray,fill=gray]
	(0,0) circle (.15)
	(1,0) circle (.15)
	(2,0) circle (.15)
	(3,0) circle (.15)
	(2,-.6) circle (.15);
	\draw[gray,fill=gray]
	(0,0) circle (.15);\end{tikzpicture}. This is to avoid confusing the vertex $(4,1)$ with either $(5,1)$ or $(3,1)$ and $(3,2)$, respectively. We will continue to do this with the Coxeter diagrams of type $\LGD_n$ and $\LGE_n$. This causes the quivers for the orthogonal Grassmannians (subsection \ref{subsec:OG}), the Cayley plane (subsection \ref{subsec:Cayley}) and the Freudenthal variety (subsection \ref{subsec:Freudenthal}) to lose their symmetry, but this is purely aesthetic.
\end{rem}

\subsection{The Lagrangian Grassmannian,} $\X=\LG(n,2n)=\Sp_{2n}/P_n$, considered as a homogeneous space for the symplectic group of type $\LGC_n$. Note that the parabolic subgroup is given by
\[
P_n = \mtrx{{cc} \GL_n & \Mat_{n\times n} \\ 0 & \GL_n}\cap\Sp_{2n}.
\]
The coset representative for the longest Weyl group element has reduced expression:
\[
\wP 
= (s_n)(s_{n-1}s_n)(s_{n-2}s_{n-1}s_n)\cdots(s_1s_2\cdots s_n),
\]
Similarly, we find for $\wPrime$
\[
\wPrime 
= (s_n)(s_{n-1}s_n)(s_{n-2}s_{n-1}s_n)\cdots(s_2s_3\cdots s_n).
\]
The quiver $\PerrinQ_{\X}$ can be written as the triangle that is the left half of an $n\times n$-square, and the reduced subexpression for $\wPrime$ that is minimal with respect to the lexicographical order is the triangle that is the left half of the $(n-1)\times(n-1)$-square obtained after removing the bottom row and the rightmost column. Note that the Laurent polynomial expression of Theorem \ref{thm:Explicit_LP_LG_model} coincides with the description given in Proposition A.1 of \cite{Pech_Rietsch_Lagrangian_Grassmannians}.

\begin{ex}\label{ex:LP_for_Lagrangian_Grassmannians}
Consider $\LG(4,8)$ which is homogeneous for $\Sp_8$ of type $\LGC_4$. We find that $\wP=(s_4)(s_3s_4)(s_2s_3s_4)(s_1s_2s_3s_4)$, $\wPrime=(s_4)(s_3s_4)(s_2s_3s_4)$ and the following quiver:
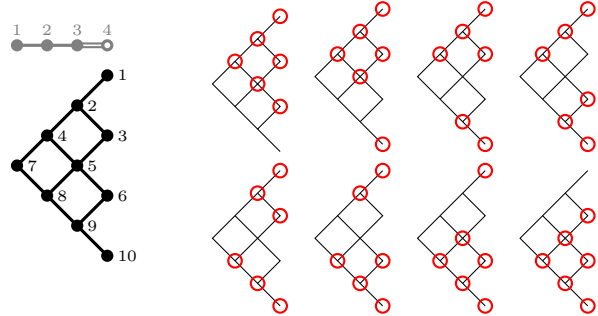
\[
\begin{tikzpicture}[scale=.4,style=very thick,baseline=1em]
	\draw[gray] (1,5) -- (3,5); \draw[gray,thick,double] (3,5) -- (4,5);
	\draw[gray,fill=gray]
	(1,5) circle (.15)
	(2,5) circle (.15)
	(3,5) circle (.15);
	\draw[gray,fill=white]
	(4,5) circle (.15);
	\draw[gray]
	(1,5) node [above]{\tiny$1$}
	(2,5) node [above]{\tiny$2$}
	(3,5) node [above]{\tiny$3$}
	(4,5) node [above]{\tiny$4$};
	\draw 
	(4,4) -- (1,1) -- (4,-2)
	(3,3) -- (4,2) -- (2,0)
	(2,2) -- (4,0) -- (3,-1);
	\draw[black, fill=black] 
	(4,4) circle (.15) 
	(3,3) circle (.15)
	(4,2) circle (.15)
	(2,2) circle (.15)
	(3,1) circle (.15)
	(4,0) circle (.15)
	(1,1) circle (.15)
	(2,0) circle (.15)
	(3,-1) circle (.15)
	(4,-2) circle (.15);
	\draw
	(4,4) node [right]{\tiny$1$}
	(3,3) node [right]{\tiny$2$}
	(4,2) node [right]{\tiny$3$}
	(2,2) node [right]{\tiny$4$}
	(3,1) node [right]{\tiny$5$}
	(4,0) node [right]{\tiny$6$}
	(1,1) node [right]{\tiny$7$}
	(2,0) node [right]{\tiny$8$}
	(3,-1) node [right]{\tiny$9$}
	(4,-2) node [right]{\tiny$10$};
\end{tikzpicture}
\qquad
\begin{tabular}{cccc}
\begin{tikzpicture}[scale = .3,rotate = -90]
	\draw 
	(-3,1) -- (0,-2) -- (3,1)
	(-2,0) -- (-1,1) -- (1,-1)
	(2,0) -- (1,1) -- (-1,-1);
	\draw[thick,red] 
	(-3,1) circle (.3cm)
	(-2,0) circle (.3cm) 
	(-1,1) circle (.3cm)
	(-1,-1) circle (.3cm)
	(0,0) circle (.3cm)
	(1,1) circle (.3cm); 
\end{tikzpicture}
&
\begin{tikzpicture}[scale = .3,rotate = -90]
	\draw (-3,1) -- (0,-2) -- (3,1);
	\draw (-2,0) -- (-1,1) -- (1,-1);
	\draw (2,0) -- (1,1) -- (-1,-1);
	\draw[thick,red] (-3,1) circle (.3cm); 
	\draw[thick,red] (-2,0) circle (.3cm); 
	\draw[thick,red] (-1,1) circle (.3cm); 
	\draw[thick,red] (-1,-1) circle (.3cm); 
	\draw[thick,red] (0,0) circle (.3cm); 
	\draw[thick,red] (3,1) circle (.3cm); 
\end{tikzpicture}
&
\begin{tikzpicture}[scale = .3,rotate = -90]
	\draw (-3,1) -- (0,-2) -- (3,1);
	\draw (-2,0) -- (-1,1) -- (1,-1);
	\draw (2,0) -- (1,1) -- (-1,-1);
	\draw[thick,red] (-3,1) circle (.3cm); 
	\draw[thick,red] (-2,0) circle (.3cm); 
	\draw[thick,red] (-1,1) circle (.3cm); 
	\draw[thick,red] (-1,-1) circle (.3cm); 
	\draw[thick,red] (2,0) circle (.3cm); 
	\draw[thick,red] (3,1) circle (.3cm); 
\end{tikzpicture}
&
\begin{tikzpicture}[scale = .3,rotate = -90]
	\draw (-3,1) -- (0,-2) -- (3,1);
	\draw (-2,0) -- (-1,1) -- (1,-1);
	\draw (2,0) -- (1,1) -- (-1,-1);
	\draw[thick,red] (-3,1) circle (.3cm); 
	\draw[thick,red] (-2,0) circle (.3cm); 
	\draw[thick,red] (1,1) circle (.3cm); 
	\draw[thick,red] (-1,-1) circle (.3cm); 
	\draw[thick,red] (2,0) circle (.3cm); 
	\draw[thick,red] (3,1) circle (.3cm); 
\end{tikzpicture}
\\
\begin{tikzpicture}[scale = .3,rotate = -90]
	\draw (-3,1) -- (0,-2) -- (3,1);
	\draw (-2,0) -- (-1,1) -- (1,-1);
	\draw (2,0) -- (1,1) -- (-1,-1);
	\draw[thick,red] (-3,1) circle (.3cm); 
	\draw[thick,red] (-2,0) circle (.3cm); 
	\draw[thick,red] (-1,1) circle (.3cm); 
	\draw[thick,red] (1,-1) circle (.3cm); 
	\draw[thick,red] (2,0) circle (.3cm); 
	\draw[thick,red] (3,1) circle (.3cm); 
\end{tikzpicture}
&
\begin{tikzpicture}[scale = .3,rotate = -90]
	\draw (-3,1) -- (0,-2) -- (3,1);
	\draw (-2,0) -- (-1,1) -- (1,-1);
	\draw (2,0) -- (1,1) -- (-1,-1);
	\draw[thick,red] (-3,1) circle (.3cm); 
	\draw[thick,red] (-2,0) circle (.3cm); 
	\draw[thick,red] (1,1) circle (.3cm); 
	\draw[thick,red] (1,-1) circle (.3cm); 
	\draw[thick,red] (2,0) circle (.3cm); 
	\draw[thick,red] (3,1) circle (.3cm); 
\end{tikzpicture}
&
\begin{tikzpicture}[scale = .3,rotate = -90]
	\draw (-3,1) -- (0,-2) -- (3,1);
	\draw (-2,0) -- (-1,1) -- (1,-1);
	\draw (2,0) -- (1,1) -- (-1,-1);
	\draw[thick,red] (-3,1) circle (.3cm); 
	\draw[thick,red] (0,0) circle (.3cm); 
	\draw[thick,red] (1,1) circle (.3cm); 
	\draw[thick,red] (1,-1) circle (.3cm); 
	\draw[thick,red] (2,0) circle (.3cm); 
	\draw[thick,red] (3,1) circle (.3cm); 
\end{tikzpicture}
&
\begin{tikzpicture}[scale = .3,rotate = -90]
	\draw (-3,1) -- (0,-2) -- (3,1);
	\draw (-2,0) -- (-1,1) -- (1,-1);
	\draw (2,0) -- (1,1) -- (-1,-1);
	\draw[thick,red] (-1,1) circle (.3cm); 
	\draw[thick,red] (0,0) circle (.3cm); 
	\draw[thick,red] (1,1) circle (.3cm); 
	\draw[thick,red] (1,-1) circle (.3cm); 
	\draw[thick,red] (2,0) circle (.3cm); 
	\draw[thick,red] (3,1) circle (.3cm); 
\end{tikzpicture}
\end{tabular}
\]
To the right we have drawn the eight subsets corresponding to reduced subexpressions of $\wPrime$ inside the reduced expression for $\wP$. So, considering the decomposition $z=u_+t\bwop u_-$ with $(u_-)^{-1} = \dy_4(-a_1)\dy_3(-a_2)\cdots\dy_4(-a_{10})$ and $q=\sdr_4(t)$, we find
\[
\potZo(z) = \sum_{i=1}^{10}a_i + q\frac{P(a_i)}{a_1a_2a_3a_4a_5a_6a_7a_8a_9a_{10}},
\]
where
\ali{
P(a_i) &= a_1a_2a_3a_4a_5a_6 + a_1a_2a_3a_4a_5a_{10}+a_1a_2a_3a_4a_9a_{10} + a_1a_2a_4a_6a_9a_{10} \\
& + a_1a_2a_3a_8a_9a_{10} + a_1a_2a_6a_8a_9a_{10} + a_1a_5a_6a_8a_9a_{10} + a_3a_5a_6a_8a_9a_{10}
}
are the eight reduced expression for $\wPrime$ obtained from the diagrams above.
\end{ex}

\subsection{The orthogonal Grassmannian,}\label{subsec:OG} $\X=\OG(n,2n)=\Spin_{2n}/\P_n$, considered as a homogeneous space for the spin group of type $\LGD_n$. We fix for the minimal coset representative $\wP$ of the longest element the reduced expression
\[
\wP = \left\{\begin{array}{ll}
s_{n-1}(s_{n-2})s_n(s_{n-3}s_{n-2})s_{n-1}\cdots(s_1s_2\cdots s_{n-2})s_n, & \text{for $n$ odd,}\\
s_n(s_{n-2})s_{n-1}(s_{n-3}s_{n-2})s_n\cdots (s_1s_2\cdots s_{n-2})s_n, & \text{for $n$ even.}
\end{array}\right.
\]
Now, for $\wPrime$ we find the reduced expression:
\[
\wPrime = \left\{\begin{array}{ll}
s_{n-1}(s_{n-2})s_n(s_{n-3}s_{n-2})s_{n-1}\cdots(s_3s_4\cdots s_{n-2})s_n, & \text{for $n$ odd,}\\
s_n(s_{n-2})s_{n-1}(s_{n-3}s_{n-2})s_n\cdots (s_3s_4\cdots s_{n-2})s_n, & \text{for $n$ even.}
\end{array}\right.
\]
Note that these expressions are very similar to those of Lagrangian Grassmannians except that the expressions for the orthogonal Grassmannians alternates between $s_n$ and $s_{n-1}$. The quiver $\PerrinQ_{\X}$ is therefore the left half of an $(n-1)\times(n-1)$-square with the longest column split over two columns. The subset that corresponds to the reduced subexpression for $\wPrime$ that is minimal with respect to the lexicographical order is in this case the triangle that is the left half of the $(n-3)\times(n-3)$-square obtained after removing the two bottom rows and the two rightmost columns.

\begin{ex}\label{ex:OG}
Consider $\OG(5,10)$ which is homogeneous for $\Spin_{10}$ of type $\LGD_5$. We find that $\wP=s_4(s_3)s_5(s_2s_3)s_4(s_1s_2s_3)s_5$ and $\wPrime=s_4(s_3)s_5$. We find the quiver:
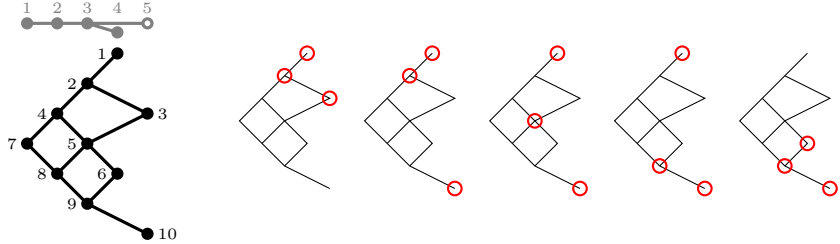
\[
\begin{tikzpicture}[scale=.4,style=very thick,baseline=1em]
	\draw[gray] (1,5) -- (5,5); \draw[gray] (3,5) -- (4,4.7);
	\draw[gray,fill=gray]
	(1,5) circle (.15)
	(2,5) circle (.15)
	(3,5) circle (.15)
	(4,4.7) circle (.15);
	\draw[gray,fill=white]
	(5,5) circle (.15);
	\node at (1,5)[above,gray]{\tiny$1$};
	\node at (2,5)[above,gray]{\tiny$2$};
	\node at (3,5)[above,gray]{\tiny$3$};
	\node at (4,5)[above,gray]{\tiny$4$};
	\node at (5,5)[above,gray]{\tiny$5$};
	\draw 
	(4,4) -- (1,1) -- (3,-1) -- (5,-2)
	(3,3) -- (5,2) -- (3,1) -- (2,0)
	(2,2) -- (4,0) -- (3,-1);
	\draw[black, fill=black] 
	(4,4) circle (.15) 
	(3,3) circle (.15)
	(5,2) circle (.15)
	(2,2) circle (.15)
	(3,1) circle (.15)
	(4,0) circle (.15)
	(1,1) circle (.15)
	(2,0) circle (.15)
	(3,-1) circle (.15)
	(5,-2) circle (.15);
	\node at (4,4)[left]{\tiny$1$};
	\node at (3,3)[left]{\tiny$2$};
	\node at (5,2)[right]{\tiny$3$};
	\node at (2,2)[left]{\tiny$4$};
	\node at (3,1)[left]{\tiny$5$};
	\node at (4,0)[left]{\tiny$6$};
	\node at (1,1)[left]{\tiny$7$};
	\node at (2,0)[left]{\tiny$8$};
	\node at (3,-1)[left]{\tiny$9$};
	\node at (5,-2)[right]{\tiny$10$};
\end{tikzpicture}
\qquad
\begin{tikzpicture}[scale = .3,rotate=-90,baseline=-1em]
	\draw (-3,1) -- (0,-2) -- (2,0) -- (3,2);
	\draw (-2,0) -- (-1,2) -- (0,0) -- (1,-1);
	\draw (2,0) -- (1,1) -- (-1,-1);
	\draw[thick,red] (-3,1) circle (.3cm); 
	\draw[thick,red] (-2,0) circle (.3cm); 
	\draw[thick,red] (-1,2) circle (.3cm); 
\end{tikzpicture}\quad
\begin{tikzpicture}[scale = .3,rotate=-90,baseline=-1em]
	\draw (-3,1) -- (0,-2) -- (2,0) -- (3,2);
	\draw (-2,0) -- (-1,2) -- (0,0) -- (1,-1);
	\draw (2,0) -- (1,1) -- (-1,-1);
	\draw[thick,red] (-3,1) circle (.3cm); 
	\draw[thick,red] (-2,0) circle (.3cm); 
	\draw[thick,red] (3,2) circle (.3cm); 
\end{tikzpicture}\quad
\begin{tikzpicture}[scale = .3,rotate=-90,baseline=-1em]
	\draw (-3,1) -- (0,-2) -- (2,0) -- (3,2);
	\draw (-2,0) -- (-1,2) -- (0,0) -- (1,-1);
	\draw (2,0) -- (1,1) -- (-1,-1);
	\draw[thick,red] (-3,1) circle (.3cm); 
	\draw[thick,red] (0,0) circle (.3cm); 
	\draw[thick,red] (3,2) circle (.3cm); 
\end{tikzpicture}\quad
\begin{tikzpicture}[scale = .3,rotate=-90,baseline=-1em]
	\draw (-3,1) -- (0,-2) -- (2,0) -- (3,2);
	\draw (-2,0) -- (-1,2) -- (0,0) -- (1,-1);
	\draw (2,0) -- (1,1) -- (-1,-1);
	\draw[thick,red] (-3,1) circle (.3cm); 
	\draw[thick,red] (2,0) circle (.3cm); 
	\draw[thick,red] (3,2) circle (.3cm); 
\end{tikzpicture}\quad
\begin{tikzpicture}[scale = .3,rotate=-90,baseline=-1em]
	\draw (-3,1) -- (0,-2) -- (2,0) -- (3,2);
	\draw (-2,0) -- (-1,2) -- (0,0) -- (1,-1);
	\draw (2,0) -- (1,1) -- (-1,-1);
	\draw[thick,red] (1,1) circle (.3cm); 
	\draw[thick,red] (2,0) circle (.3cm); 
	\draw[thick,red] (3,2) circle (.3cm); 
\end{tikzpicture}
\]
To the right we have drawn the five subsets corresponding to reduced subexpressions of $\wPrime$ inside the reduced expression for $\wP$. So, we find
\[
\potZo(z) = \sum_{i=1}^{10}a_i + q\frac{a_1a_2a_3 + a_1a_2a_{10} + a_1a_5a_{10} + a_1a_9a_{10}+a_6a_9a_{10}}{a_1a_2a_3a_4a_5a_6a_7a_8a_9a_{10}},
\]
where $z=u_+t\bwop u_-$ with $(u_-)^{-1} = \dy_4(-a_1)\dy_3(-a_2)\cdots\dy_5(-a_{10})$ and $q=\sdr_5(t)$.
\end{ex}

\subsection{The Cayley plane,}\label{subsec:Cayley} $\X = \OP^2 = \LGE_6^\SC/\P_6$, considered as a homogeneous space for the simply-connected Lie group $\LGE_6^\SC$ of type $\LGE_6$. We fix for $\wP$ the reduced expression
\[
\wP 
= s_1s_3s_4s_2s_5s_6s_4s_5s_3s_4s_2s_1s_3s_4s_5s_6
\]
and we find for $\wPrime$ the (unique) reduced expression
$\wPrime = s_1s_3s_4s_5s_6$.
Thus, the quiver is as follows:
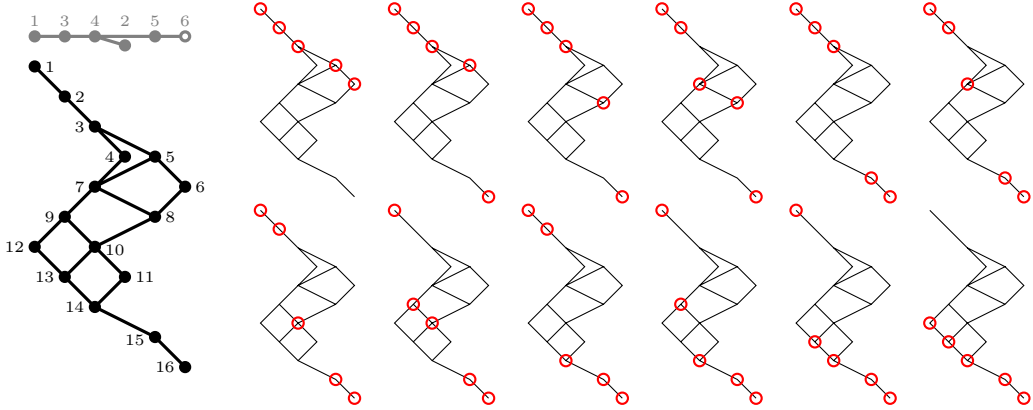
\[
\begin{tikzpicture}[scale=.4,style=very thick,baseline=0.25em]
	\draw[gray] (0,6) -- (5,6); \draw[gray] (2,6) -- (3,5.7);
	\draw[gray,fill=gray]
	(0,6) circle (.15)
	(1,6) circle (.15)
	(2,6) circle (.15)
	(3,5.7) circle (.15)
	(4,6) circle (.15);
	\draw[gray,fill=white]
	(5,6) circle (.15);
	\node at (0,6)[above,gray]{\tiny$1$};
	\node at (3,6)[above,gray]{\tiny$2$};
	\node at (1,6)[above,gray]{\tiny$3$};
	\node at (2,6)[above,gray]{\tiny$4$};
	\node at (4,6)[above,gray]{\tiny$5$};
	\node at (5,6)[above,gray]{\tiny$6$};
	\draw 
	(0,5) -- (3,2) -- (0,-1) -- (2,-3)--(4,-4)--(5,-5)
	(2,3) -- (4,2) -- (2,1) -- (4,0)
	(4,2) -- (5,1) -- (4,0) -- (2,-1) -- (3,-2) -- (2,-3)
	(1,0) -- (2,-1) -- (1,-2);
	\draw[black, fill=black] 
	(0,5) circle (.15) 
	(1,4) circle (.15)
	(2,3) circle (.15)
	(3,2) circle (.15)
	(4,2) circle (.15)
	(2,1) circle (.15)
	(5,1) circle (.15)
	(1,0) circle (.15)
	(4,0) circle (.15)
	(0,-1) circle (.15)
	(2,-1) circle (.15)
	(1,-2) circle (.15)
	(3,-2) circle (.15)
	(2,-3) circle (.15)
	(4,-4) circle (.15)
	(5,-5) circle (.15);
	\node at (0,5)[right]{\tiny$1$};
	\node at (1,4)[right]{\tiny$2$};
	\node at (2,3)[left]{\tiny$3$};
	\node at (3,2)[left]{\tiny$4$};
	\node at (4,2)[right]{\tiny$5$};
	\node at (5,1)[right]{\tiny$6$};
	\node at (2,1)[left]{\tiny$7$};
	\node at (4,0)[right]{\tiny$8$};
	\node at (1,0)[left]{\tiny$9$};
	\node at (2,-1)[below=1pt, right]{\tiny$10$};
	\node at (3,-2)[right]{\tiny$11$};
	\node at (0,-1)[left]{\tiny$12$};
	\node at (1,-2)[left]{\tiny$13$};
	\node at (2,-3)[left]{\tiny$14$};
	\node at (4,-4)[below=1pt,left]{\tiny$15$};
	\node at (5,-5)[left]{\tiny$16$};
\end{tikzpicture}
\quad
\begin{tabular}{cccccc}
\begin{tikzpicture}[scale = .25,baseline=-1em]
	\draw 
	(0,5) -- (3,2) -- (0,-1) -- (2,-3)--(4,-4)--(5,-5)
	(2,3) -- (4,2) -- (2,1) -- (4,0)
	(4,2) -- (5,1) -- (4,0) -- (2,-1) -- (3,-2) -- (2,-3)
	(1,0) -- (2,-1) -- (1,-2);
	\draw[thick,red]
	(0,5) circle (.3cm)
	(1,4) circle (.3cm) 
	(2,3) circle (.3cm)
	(4,2) circle (.3cm)
	(5,1) circle (.3cm);
\end{tikzpicture}
&
\begin{tikzpicture}[scale = .25,baseline=-1em]
	\draw 
	(0,5) -- (3,2) -- (0,-1) -- (2,-3)--(4,-4)--(5,-5)
	(2,3) -- (4,2) -- (2,1) -- (4,0)
	(4,2) -- (5,1) -- (4,0) -- (2,-1) -- (3,-2) -- (2,-3)
	(1,0) -- (2,-1) -- (1,-2);
	\draw[thick,red]
	(0,5) circle (.3cm)
	(1,4) circle (.3cm) 
	(2,3) circle (.3cm)
	(4,2) circle (.3cm)
	(5,-5) circle (.3cm);
\end{tikzpicture}
&
\begin{tikzpicture}[scale = .25,baseline=-1em]
	\draw 
	(0,5) -- (3,2) -- (0,-1) -- (2,-3)--(4,-4)--(5,-5)
	(2,3) -- (4,2) -- (2,1) -- (4,0)
	(4,2) -- (5,1) -- (4,0) -- (2,-1) -- (3,-2) -- (2,-3)
	(1,0) -- (2,-1) -- (1,-2);
	\draw[thick,red]
	(0,5) circle (.3cm)
	(1,4) circle (.3cm) 
	(2,3) circle (.3cm)
	(4,0) circle (.3cm)
	(5,-5) circle (.3cm);
\end{tikzpicture}
&
\begin{tikzpicture}[scale = .25,baseline=-1em]
	\draw 
	(0,5) -- (3,2) -- (0,-1) -- (2,-3)--(4,-4)--(5,-5)
	(2,3) -- (4,2) -- (2,1) -- (4,0)
	(4,2) -- (5,1) -- (4,0) -- (2,-1) -- (3,-2) -- (2,-3)
	(1,0) -- (2,-1) -- (1,-2);
	\draw[thick,red]
	(0,5) circle (.3cm)
	(1,4) circle (.3cm) 
	(2,1) circle (.3cm)
	(4,0) circle (.3cm)
	(5,-5) circle (.3cm);
\end{tikzpicture}
&
\begin{tikzpicture}[scale = .25,baseline=-1em]
	\draw 
	(0,5) -- (3,2) -- (0,-1) -- (2,-3)--(4,-4)--(5,-5)
	(2,3) -- (4,2) -- (2,1) -- (4,0)
	(4,2) -- (5,1) -- (4,0) -- (2,-1) -- (3,-2) -- (2,-3)
	(1,0) -- (2,-1) -- (1,-2);
	\draw[thick,red]
	(0,5) circle (.3cm)
	(1,4) circle (.3cm) 
	(2,3) circle (.3cm)
	(4,-4) circle (.3cm)
	(5,-5) circle (.3cm);
\end{tikzpicture}
&
\begin{tikzpicture}[scale = .25,baseline=-1em]
	\draw 
	(0,5) -- (3,2) -- (0,-1) -- (2,-3)--(4,-4)--(5,-5)
	(2,3) -- (4,2) -- (2,1) -- (4,0)
	(4,2) -- (5,1) -- (4,0) -- (2,-1) -- (3,-2) -- (2,-3)
	(1,0) -- (2,-1) -- (1,-2);
	\draw[thick,red]
	(0,5) circle (.3cm)
	(1,4) circle (.3cm) 
	(2,1) circle (.3cm)
	(4,-4) circle (.3cm)
	(5,-5) circle (.3cm);
\end{tikzpicture}
\\
\begin{tikzpicture}[scale = .25,baseline=-1em]
	\draw 
	(0,5) -- (3,2) -- (0,-1) -- (2,-3)--(4,-4)--(5,-5)
	(2,3) -- (4,2) -- (2,1) -- (4,0)
	(4,2) -- (5,1) -- (4,0) -- (2,-1) -- (3,-2) -- (2,-3)
	(1,0) -- (2,-1) -- (1,-2);
	\draw[thick,red]
	(0,5) circle (.3cm)
	(1,4) circle (.3cm) 
	(2,-1) circle (.3cm)
	(4,-4) circle (.3cm)
	(5,-5) circle (.3cm);
\end{tikzpicture}
&
\begin{tikzpicture}[scale = .25,baseline=-1em]
	\draw 
	(0,5) -- (3,2) -- (0,-1) -- (2,-3)--(4,-4)--(5,-5)
	(2,3) -- (4,2) -- (2,1) -- (4,0)
	(4,2) -- (5,1) -- (4,0) -- (2,-1) -- (3,-2) -- (2,-3)
	(1,0) -- (2,-1) -- (1,-2);
	\draw[thick,red]
	(0,5) circle (.3cm)
	(1,0) circle (.3cm) 
	(2,-1) circle (.3cm)
	(4,-4) circle (.3cm)
	(5,-5) circle (.3cm);
\end{tikzpicture}
&
\begin{tikzpicture}[scale = .25,baseline=-1em]
	\draw 
	(0,5) -- (3,2) -- (0,-1) -- (2,-3)--(4,-4)--(5,-5)
	(2,3) -- (4,2) -- (2,1) -- (4,0)
	(4,2) -- (5,1) -- (4,0) -- (2,-1) -- (3,-2) -- (2,-3)
	(1,0) -- (2,-1) -- (1,-2);
	\draw[thick,red]
	(0,5) circle (.3cm)
	(1,4) circle (.3cm) 
	(2,-3) circle (.3cm)
	(4,-4) circle (.3cm)
	(5,-5) circle (.3cm);
\end{tikzpicture}
&
\begin{tikzpicture}[scale = .25,baseline=-1em]
	\draw 
	(0,5) -- (3,2) -- (0,-1) -- (2,-3)--(4,-4)--(5,-5)
	(2,3) -- (4,2) -- (2,1) -- (4,0)
	(4,2) -- (5,1) -- (4,0) -- (2,-1) -- (3,-2) -- (2,-3)
	(1,0) -- (2,-1) -- (1,-2);
	\draw[thick,red]
	(0,5) circle (.3cm)
	(1,0) circle (.3cm) 
	(2,-3) circle (.3cm)
	(4,-4) circle (.3cm)
	(5,-5) circle (.3cm);
\end{tikzpicture}
&
\begin{tikzpicture}[scale = .25,baseline=-1em]
	\draw 
	(0,5) -- (3,2) -- (0,-1) -- (2,-3)--(4,-4)--(5,-5)
	(2,3) -- (4,2) -- (2,1) -- (4,0)
	(4,2) -- (5,1) -- (4,0) -- (2,-1) -- (3,-2) -- (2,-3)
	(1,0) -- (2,-1) -- (1,-2);
	\draw[thick,red]
	(0,5) circle (.3cm)
	(1,-2) circle (.3cm) 
	(2,-3) circle (.3cm)
	(4,-4) circle (.3cm)
	(5,-5) circle (.3cm);
\end{tikzpicture}
&
\begin{tikzpicture}[scale = .25,baseline=-1em]
	\draw 
	(0,5) -- (3,2) -- (0,-1) -- (2,-3)--(4,-4)--(5,-5)
	(2,3) -- (4,2) -- (2,1) -- (4,0)
	(4,2) -- (5,1) -- (4,0) -- (2,-1) -- (3,-2) -- (2,-3)
	(1,0) -- (2,-1) -- (1,-2);
	\draw[thick,red]
	(0,-1) circle (.3cm)
	(1,-2) circle (.3cm) 
	(2,-3) circle (.3cm)
	(4,-4) circle (.3cm)
	(5,-5) circle (.3cm);
\end{tikzpicture}
\end{tabular}
\]
To the right we have drawn the twelve subsets corresponding to subexpressions of $\wPrime$. 
We find for $z=u_+t\bwop u_-$ with $(u_-)^{-1}=\dy_1(-a_1)\dy_3(-a_2)\cdots\dy_6(-a_{16})$
\[
\potZo(z) = \sum_{i=1}^{16} a_i + q\frac{P(a_i)}{\prod_{i=1}^{16} a_i},
\]
with $q=\sdr_6(t)$ and in order of the drawn subsets
\ali{
\hspace{-3em}P(a_i) &= a_1a_2a_3a_5a_6+a_1a_2a_3a_5a_{16}+a_1a_2a_3a_{8}a_{16}+a_1a_2a_7a_{8}a_{16}+a_1a_2a_3a_{15}a_{16}+a_1a_2a_7a_{15}a_{16}\\
&{}+a_1a_2a_{10}a_{15}a_{16}+a_1a_9a_{10}a_{15}a_{16}+a_1a_2a_{14}a_{15}a_{16}+a_1a_9a_{14}a_{15}a_{16}+a_1a_{13}a_{14}a_{15}a_{16}+a_{12}a_{13}a_{14}a_{15}a_{16}
}

\subsection{The Freudenthal variety,}\label{subsec:Freudenthal} $\X=\LGE^\SC_7/\P_7$, considered as a homogeneous space for the simply-connected Lie group $\LGE^\SC_7$ of type $\LGE_7$. We fix for $\wP$ the reduced expression
\[
\wP = s_7s_6s_5s_4s_2s_3s_4s_5s_6s_7s_1s_3s_4s_2s_5s_6s_4s_5s_3s_4s_2s_1s_3s_4s_5s_6s_7,
\]
and $\wPrime$ has reduced expression 
$\wPrime = s_7s_6s_5s_4(s_2s_3)s_4s_5s_6s_7$.
The quiver $\PerrinQ_{\X}$ is of the form
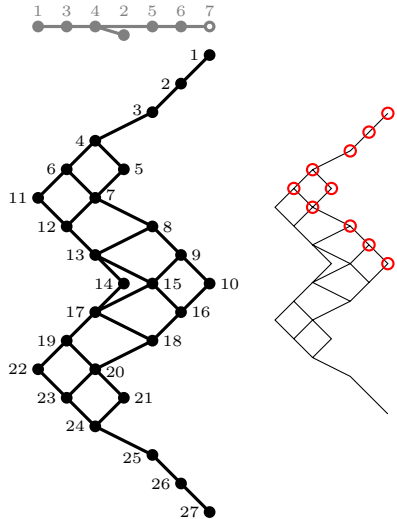
\[
\begin{tikzpicture}[scale=.38,style=very thick,baseline=0.5em]
	\draw[gray] (0,11) -- (6,11); \draw[gray] (2,11) -- (3,10.7);
	\draw[gray,fill=gray]
	(0,11) circle (.15)
	(1,11) circle (.15)
	(2,11) circle (.15)
	(3,10.7) circle (.15)
	(4,11) circle (.15)
	(5,11) circle (.15);
	\draw[gray,fill=white]
	(6,11) circle (.15);
	\node at (0,11)[above,gray]{\tiny$1$};
	\node at (3,11)[above,gray]{\tiny$2$};
	\node at (1,11)[above,gray]{\tiny$3$};
	\node at (2,11)[above,gray]{\tiny$4$};
	\node at (4,11)[above,gray]{\tiny$5$};
	\node at (5,11)[above,gray]{\tiny$6$};
	\node at (6,11)[above,gray]{\tiny$7$};
	\draw 
	(6,10) -- (4,8) -- (2,7) -- (0,5) -- (3,2) -- (0,-1) -- (2,-3)--(4,-4)--(6,-6)
	(2,7) -- (3,6) -- (2,5) -- (4,4) -- (2,3) -- (4,2) -- (2,1) -- (4,0)
	(4,2) -- (5,1) -- (4,0) -- (2,-1) -- (3,-2) -- (2,-3)
	(1,0) -- (2,-1) -- (1,-2)
	(1,6) -- (2,5) -- (1,4) (4,4) -- (6,2) -- (5,1) (5,3) -- (4,2);
	\draw[black, fill=black] 
	(6,10) circle (.15) 
	(5,9) circle (.15)
	(4,8) circle (.15)
	(2,7) circle (.15)
	(3,6) circle (.15)
	(1,6) circle (.15) 
	(2,5) circle (.15)
	(4,4) circle (.15)
	(5,3) circle (.15)
	(6,2) circle (.15)	
	(0,5) circle (.15) 
	(1,4) circle (.15)
	(2,3) circle (.15)
	(3,2) circle (.15)
	(4,2) circle (.15)
	(2,1) circle (.15)
	(5,1) circle (.15)
	(1,0) circle (.15)
	(4,0) circle (.15)
	(0,-1) circle (.15)
	(2,-1) circle (.15)
	(1,-2) circle (.15)
	(3,-2) circle (.15)
	(2,-3) circle (.15)
	(4,-4) circle (.15)
	(5,-5) circle (.15)
	(6,-6) circle (.15);
	\node at (6,10)[left]{\tiny$1$};
	\node at (5,9)[left]{\tiny$2$};
	\node at (4,8)[above=1pt,left]{\tiny$3$};
	\node at (2,7)[left]{\tiny$4$};
	\node at (3,6)[right]{\tiny$5$};
	\node at (1,6)[left]{\tiny$6$};
	\node at (2,5)[above=1pt,right]{\tiny$7$};
	\node at (4,4)[right]{\tiny$8$};
	\node at (5,3)[right]{\tiny$9$};
	\node at (6,2)[right]{\tiny$10$};
	\node at (0,5)[left]{\tiny$11$};
	\node at (1,4)[left]{\tiny$12$};
	\node at (2,3)[left]{\tiny$13$};
	\node at (3,2)[left]{\tiny$14$};
	\node at (4,2)[right]{\tiny$15$};
	\node at (5,1)[right]{\tiny$16$};
	\node at (2,1)[left]{\tiny$17$};
	\node at (4,0)[right]{\tiny$18$};
	\node at (1,0)[left]{\tiny$19$};
	\node at (2,-1)[below=1pt, right]{\tiny$20$};
	\node at (3,-2)[right]{\tiny$21$};
	\node at (0,-1)[left]{\tiny$22$};
	\node at (1,-2)[left]{\tiny$23$};
	\node at (2,-3)[left]{\tiny$24$};
	\node at (4,-4)[below=1pt,left]{\tiny$25$};
	\node at (5,-5)[left]{\tiny$26$};
	\node at (6,-6)[left]{\tiny$27$};
\end{tikzpicture}
\quad
\begin{tikzpicture}[scale = .25,baseline=-1em]
	\draw 
	(6,10) -- (4,8) -- (2,7) -- (0,5) -- (3,2) -- (0,-1) -- (2,-3)--(4,-4)--(6,-6)
	(2,7) -- (3,6) -- (2,5) -- (4,4) -- (2,3) -- (4,2) -- (2,1) -- (4,0)
	(4,2) -- (5,1) -- (4,0) -- (2,-1) -- (3,-2) -- (2,-3)
	(1,0) -- (2,-1) -- (1,-2)
	(1,6) -- (2,5) -- (1,4) (4,4) -- (6,2) -- (5,1) (5,3) -- (4,2);
	\draw[thick,red]
	(6,10) circle (.3cm)
	(5,9) circle (.3cm)
	(4,8) circle (.3cm)
	(2,7) circle (.3cm)
	(1,6) circle (.3cm) 
	(3,6) circle (.3cm)
	(2,5) circle (.3cm)
	(4,4) circle (.3cm)
	(5,3) circle (.3cm)
	(6,2) circle (.3cm);
\end{tikzpicture} 
\]
To the right we have drawn the subset corresponding to the reduced subexpression of $\wPrime$ that is minimal with respect to the lexicographical order. However, in this case, there are 78 reduced subexpressions, so instead of giving all the corresponding subsets, we will simply list in lexicographical order the elements of $\wPrimeSubExp$ consisting of sequences of subindices $(i_1,\ldots,i_{10})$ of $\wP=s_{r_1}\cdots s_{r_\ellwP}$ such that $\wPrime=s_{i_1}\cdots s_{i_{10}}$, see Table \ref{tab:Freudenthal_wPrime_SubExp}.

We find for $z=u_+t\bwop u_-$ with $(u_-)^{-1}=\dy_7(-a_1)\dy_6(-a_2)\cdots\dy_7(-a_{27})$ that
\[
\potZo(z) = \sum_{i=1}^{27} a_i + q\frac{P(a_i)}{\prod_{i=1}^{27} a_i},
\]
with $q=\sdr_7(t)$ and $P(a_i) = \sum_{(i_j)\in\wPrimeSubExp}\prod_{j=1}^{10} a_{i_j}$ is a homogeneous polynomial of degree 10 with 78 terms.

\begin{table}[htb]%
\tiny\ali{
\hspace{-3em}\wPrimeSubExp =\bigl\{
&(~\,1,~\,2,~\,3,~\,4,~\,5,~\,6,~\,7,~\,8,~\,9,10),(~\,1,~\,2,~\,3,~\,4,~\,5,~\,6,~\,7,~\,8,~\,9,27),(~\,1,~\,2,~\,3,~\,4,~\,5,~\,6,~\,7,~\,8,16,27), \\					
&(~\,1,~\,2,~\,3,~\,4,~\,5,~\,6,~\,7,~\,8,26,27),(~\,1,~\,2,~\,3,~\,4,~\,5,~\,6,~\,7,15,16,27),(~\,1,~\,2,~\,3,~\,4,~\,5,~\,6,~\,7,15,26,27), \\				
&(~\,1,~\,2,~\,3,~\,4,~\,5,~\,6,~\,7,18,26,27),(~\,1,~\,2,~\,3,~\,4,~\,5,~\,6,~\,7,25,26,27),(~\,1,~\,2,~\,3,~\,4,~\,5,~\,6,13,15,16,27), \\			
&(~\,1,~\,2,~\,3,~\,4,~\,5,~\,6,13,15,26,27),(~\,1,~\,2,~\,3,~\,4,~\,5,~\,6,13,18,26,27),(~\,1,~\,2,~\,3,~\,4,~\,5,~\,6,13,25,26,27), \\			
&(~\,1,~\,2,~\,3,~\,4,~\,5,~\,6,17,18,26,27),(~\,1,~\,2,~\,3,~\,4,~\,5,~\,6,17,25,26,27),(~\,1,~\,2,~\,3,~\,4,~\,5,~\,6,20,25,26,27), \\			
&(~\,1,~\,2,~\,3,~\,4,~\,5,~\,6,24,25,26,27),(~\,1,~\,2,~\,3,~\,4,~\,5,12,13,15,16,27),(~\,1,~\,2,~\,3,~\,4,~\,5,12,13,15,26,27), \\			
&(~\,1,~\,2,~\,3,~\,4,~\,5,12,13,18,26,27),(~\,1,~\,2,~\,3,~\,4,~\,5,12,13,25,26,27),(~\,1,~\,2,~\,3,~\,4,~\,5,12,17,18,26,27), \\			
&(~\,1,~\,2,~\,3,~\,4,~\,5,12,17,25,26,27),(~\,1,~\,2,~\,3,~\,4,~\,5,12,20,25,26,27),(~\,1,~\,2,~\,3,~\,4,~\,5,12,24,25,26,27), \\			
&(~\,1,~\,2,~\,3,~\,4,~\,5,19,20,25,26,27),(~\,1,~\,2,~\,3,~\,4,~\,5,19,24,25,26,27),(~\,1,~\,2,~\,3,~\,4,~\,5,23,24,25,26,27), \\			
&(~\,1,~\,2,~\,3,~\,4,~\,6,14,17,18,26,27),(~\,1,~\,2,~\,3,~\,4,~\,6,14,17,25,26,27),(~\,1,~\,2,~\,3,~\,4,~\,6,14,20,25,26,27), \\			
&(~\,1,~\,2,~\,3,~\,4,~\,6,14,24,25,26,27), (~\,1,~\,2,~\,3,~\,4,~\,6,21,24,25,26,27),(~\,1,~\,2,~\,3,~\,4,12,14,17,18,26,27), \\		
&(~\,1,~\,2,~\,3,~\,4,12,14,17,25,26,27),(~\,1,~\,2,~\,3,~\,4,12,14,20,25,26,27),(~\,1,~\,2,~\,3,~\,4,12,14,24,25,26,27), \\			
&(~\,1,~\,2,~\,3,~\,4,12,21,24,25,26,27),(~\,1,~\,2,~\,3,~\,4,14,19,20,25,26,27),(~\,1,~\,2,~\,3,~\,4,14,19,24,25,26,27), \\			
&(~\,1,~\,2,~\,3,~\,4,14,23,24,25,26,27),(~\,1,~\,2,~\,3,~\,4,19,21,24,25,26,27),(~\,1,~\,2,~\,3,~\,4,21,23,24,25,26,27), \\			
&(~\,1,~\,2,~\,3,~\,7,12,14,17,18,26,27),(~\,1,~\,2,~\,3,~\,7,12,14,17,25,26,27),(~\,1,~\,2,~\,3,~\,7,12,14,20,25,26,27), \\			
&(~\,1,~\,2,~\,3,~\,7,12,14,24,25,26,27),(~\,1,~\,2,~\,3,~\,7,12,21,24,25,26,27),(~\,1,~\,2,~\,3,~\,7,14,19,20,25,26,27), \\			
&(~\,1,~\,2,~\,3,~\,7,14,19,24,25,26,27),(~\,1,~\,2,~\,3,~\,7,14,23,24,25,26,27),(~\,1,~\,2,~\,3,~\,7,19,21,24,25,26,27), \\			
&(~\,1,~\,2,~\,3,~\,7,21,23,24,25,26,27),(~\,1,~\,2,~\,3,13,14,19,20,25,26,27),(~\,1,~\,2,~\,3,13,14,19,24,25,26,27), \\			
&(~\,1,~\,2,~\,3,13,14,23,24,25,26,27),(~\,1,~\,2,~\,3,13,19,21,24,25,26,27),(~\,1,~\,2,~\,3,13,21,23,24,25,26,27), \\			
&(~\,1,~\,2,~\,3,17,19,21,24,25,26,27),(~\,1,~\,2,~\,3,17,21,23,24,25,26,27),(~\,1,~\,2,~\,3,20,21,23,24,25,26,27), \\			
&(~\,1,~\,2,~\,8,13,14,19,20,25,26,27),(~\,1,~\,2,~\,8,13,14,19,24,25,26,27),(~\,1,~\,2,~\,8,13,14,23,24,25,26,27), \\			
&(~\,1,~\,2,~\,8,13,19,21,24,25,26,27),(~\,1,~\,2,~\,8,13,21,23,24,25,26,27),(~\,1,~\,2,~\,8,17,19,21,24,25,26,27), \\			
&(~\,1,~\,2,~\,8,17,21,23,24,25,26,27),(~\,1,~\,2,~\,8,20,21,23,24,25,26,27),(~\,1,~\,2,15,17,19,21,24,25,26,27), \\			
&(~\,1,~\,2,15,17,21,23,24,25,26,27),(~\,1,~\,2,15,20,21,23,24,25,26,27),(~\,1,~\,2,18,20,21,23,24,25,26,27), \\			
&(~\,1,~\,9,15,17,19,21,24,25,26,27),(~\,1,~\,9,15,17,21,23,24,25,26,27),(~\,1,~\,9,15,20,21,23,24,25,26,27), \\			
&(~\,1,~\,9,18,20,21,23,24,25,26,27),(~\,1,16,18,20,21,23,24,25,26,27),(10,16,18,20,21,23,24,25,26,27)\bigr\}	
}
\normalsize 
\caption{The full list of all the 78 elements of $\wPrimeSubExp$ for the Freudenthal variety $\LGE_7^\SC$.}
\label{tab:Freudenthal_wPrime_SubExp}
\end{table}

\bibliographystyle{amsalpha}   
\bibliography{LPpotComHomSp-Biblio}

\end{document}